\documentclass[11pt, fleqn]{amsart}
\usepackage{amsmath, amssymb, amsthm} 
\usepackage{hyperref}
\usepackage{mathtools}
\usepackage{graphicx}
\usepackage{adjustbox}
\usepackage{caption}
\usepackage{subcaption}
\usepackage{booktabs}
\usepackage{tabularx}
\usepackage{float}
\usepackage{longtable}
\usepackage{tikz-cd}

\usepackage{listings}
\usepackage{xcolor}
\newtheorem{theorem}{Theorem}[section]
\newtheorem{lemma}[theorem]{Lemma}

\newtheorem{conjecture}{Conjecture}
\newtheorem{proposition}{Proposition}
\newtheorem*{proposition*}{Proposition}

\theoremstyle{remark}
\newtheorem{remark}{Remark}
\theoremstyle{plain}
\theoremstyle{definition}

\title
{Coherent sequences of matrix spectra}
\author{Barry Brent}
\email{axcjh@bu.edu}
\begin{document}
\begin{abstract}  
We attach sequences of matrix spectra to sequences of rational numbers
and study them with numerical experiments.
The initial segments of the spectrum sequences
 associated
to each of several  rational sequences appear
to exhibit geometric regularities.
Identities  originating
in the theory of symmetric functions express
members of arithmetic sequences 
$\{h_n\}_{n \ge 0} \thinspace (h_0 = 1)$ as  $h_n = |J_{h,n}|/n !$ for particular matrices $J_{h,n}$, where $|\cdot|$ is the determinant. 
In this setting we studied numerically sequences $\{a(n)\}_{n \ge 1}$
obtained from the Fourier expansions of cusp forms.
Within the range of our observations, the sequences of spectra of certain ``treated'' matrices $\{J^{(c)}_{a,n}\}_{n=1,2,3,...}$ exhibit coherent behavior, meaning that the plots of the minimum moduli among the eigenvalues of  the $J^{(c)}_{a,n}$ appear to oscillate or to form approximately straight lines of non-negative slope once $n$ is large enough. 
\end{abstract}
\maketitle
\vspace{-0.5em}
\begin{center}
    %rm
    {22 August 2026}
\end{center}
\vspace{1em}
\section{Introduction}
We attach sequences of matrix spectra to sequences of rational numbers
and study them with numerical experiments.
The first 500 spectra associated
to each of several interesting sequences appear
to behave nicely.

Identities (\cite{M}, Lemmas 2.1 and 2.2 below)  originating
in the theory of symmetric functions express
members of arithmetic sequences $\{h_n\}_{n \ge 0}$ with $h_0 = 1$ as $h_n = |J_{h,n}|/n !$ for particular matrices $J_{h,n}$ where $|\cdot|$ is the determinant. (At times we will suppress the dependence of $J_{h,n}$ on $h$.)

For a square matrix $M$, let $\chi(M)$ be its characteristic polynomial. For constants $c$, we define an operator $J_{h,n} \mapsto J^{(c)}_{h,n}$ in Section 3 (``Treatments''.)  
Let $E_n^{(c)}$ be the 
set of zeros of $\chi(J_n^{(c)}(j))$ and let $\mu_n^{(c)} = \min |v|$, taking the minimum over $v$ in $E_n^{(c)}$.
Let $p_n$ be the $n^{th}$ prime, let $n \mapsto \tau(n)$ be Ramanujan's tau function, and let $\lambda_n = \tau(p_n)$.  Then for $h_0 = 1$ and $h_n = \lambda_n$ otherwise, it appears that for large-enough $n$ the graph of the pairs $(n,\mu_n^{(c)})$ lies on 
an oscillating curve for $1 \le c \le 3$ 
when $h_0 = 1$ and $h_n = \tau(p_n)$
(see Figures 2, 3 and 4).  More generally, for the other sequences $h$ we tested, plots for the treated matrices $J^{(c)}_{h,n} (n = 1, 2, ....$)  appear to behave coherently, meaning that either they appear, to the author's eye, to oscillate or to form approximately straight lines of non-negative slope after $n$ is large enough.

D. H. Lehmer asked whether or not tau has any zeros and showed (\cite{Le47}, Theorem 2)  that, if tau does have zeros, the smallest one has the form $\lambda_{n^*}$ for some $n^*$. In our setup, this would mean that the spectrum of $J_{\lambda, n^*}$ contains zero. This motivated our study of spectra associated to sequences. 

The behavior of $\mu_n$ data for the spectra of untreated  $J_{\lambda, n^*}$, on the other hand, are not obviously systematic. (They are plotted for an initial segment of 
$\{\lambda_n\}_{n \ge 1}$ 
in Figure 7.) The regularizing effect of treating these matrices, however, is dramatic. (See Figures 2 - 4.) The appearance of very nearly regular behavior in the treated matrices  raises the prospect that their spectra might be helpful in understanding Lehmer's question. But we have not found a way to guarantee,
for example, that the spectrum of $J_{\lambda, n^*}$ is zero-free when that of $J^{(c)}_{\lambda, n^*}$ is zero-free. That is also a part of our project. 

We saw similar coherent behavior 
in  sequences of Fourier coefficients of newforms coming from elliptic curves listed in Cremona's database \cite{LMFDB} (Figures 9 - 41), and in other cusp forms coming from spaces of higher weight and larger dimension (Figures 42 - 47), for which the vanishing of coefficients
is not uncommon.  We also collected data on simpler sequences that do not come from modular forms (Figures 48 - 49.) 

We relied on the LLM ``Claude'' \cite{Anthropic2026}
 to write code for signal analysis and for very fast generation of our colored visualizations.\footnote{See 
 \texttt{https://zenodo.org/records/21527548}.}
Because  waveforms in our plots evolve towards coherent behavior only for large-enough $n$, we had Claude write its signal analysis code to segment the data
at natural transitions and to analyze the part of the data with $n$ so large that the pairs $(n, \mu_n)$ belong to the final segment.  At the time of the current draft, Claude's signal analyses are available only within the corresponding SageMath Jupyter notebooks.
The Zenodo URL for each figure is listed in the appendix.

Claude has an idea about the coherent behavior we are examining. For now, we are treating this idea as a black box; we do not sketch it, or try to verify it. Instead we distilled a prediction based on it, which we tested with reasonable success (Sections 9 and 10 and Table 9.)\footnote{Claude's idea is here: \texttt{https://zenodo.org/records/21781169}.}

This version has been edited to come within arXiv guidelines. It omits many plots referred to within the text. The unedited version is here \newline
\texttt{https://zenodo.org/records/22063026}.
\section{Lemmas}
These lemmas appear, for example, in MacDonald's book \cite{M}, 
but the proofs
are not easy to locate; in fact, we
have not located them.
Home-made
proofs of these statements are stored in
our repository.\footnote{ Here: \texttt{https://zenodo.org/records/21528344}.}
\begin{lemma}\label{lem:eqD}
The equations  below are equivalent. (We will
refer to both of them as equation (D) in the sequel.)
Let $h_0=1$ and
\begin{equation}
n h_n = \sum_{r=1}^{n} j_r h_{n-r}\tag{D1}
\end{equation}
for $n \ge 1$.
With
$
H(t) = \sum_{n=0}^\infty h_n t^n$ and  $J(t) = \sum_{r=1}^\infty j_r t^r$,
\begin{equation}
t \frac{d}{dt}H(t) = H(t) J(t).\tag{D2}
\end{equation}
\end{lemma}
With $f_0$ possibly defined as equal to one,    
let \(\overline{f}\) denote the sequence 
\(\{f_n\}_{n \in \mathbb{Z}^+}\).
Let $J_n(j)$ and $H_n(h)$ be the matrices 
\[ 
\left(
\begin{matrix}
j_1 & -1 & 0 & \cdots & 0 \\
j_2 & j_1 & -2 & \cdots & 0 \\
\vdots & \vdots & \vdots & \vdots & \vdots \\
j_{n-1} & j_{n-2} &  \cdots & j_1 & -n + 1\\
j_n & j_{n-1} & j_{n-2} & \cdots & j_1\\
\end{matrix}
\right )
\]
and
\[ 
\left(
\begin{matrix}
h_1 & 1 & 0 & \cdots & 0 \\
2h_2 & h_1 & 1 & \cdots & 0 \\
\vdots & \vdots & \vdots & \vdots & \vdots \\
nh_n & h_{n-1} & h_{n-2} & \cdots & h_1\\
\end{matrix}
\right),
\]
respectively.
\begin{lemma}
Let $h_0 = 1$.
    Equation (D) implies
that
\begin{equation}
j_n = (-1)^{n+1} | H_n(h)|. 
\end{equation}
and
\begin{equation}
n! h_n =  | J_n(j)|
\end{equation}
for $n \ge 1$.
\end{lemma}

Next we state a corollary of the
following theorem in Vein and Dale's book 
~\cite{Vein1999}. 
\begin{proposition}
(Vein and Dale, Theorem 4.23)
Let
\(A_n =\)
\[
\begin{pmatrix}
a_1 & -1 & 0 & \cdots & 0 \\
a_2 & a_1 & -2 & \cdots & 0 \\
\vdots & \vdots & \vdots & \ddots & \vdots \\
a_{n-1} & a_{n-2} & \cdots & a_1 & -(n{-}1) \\
a_n & a_{n-1} & \cdots &a_2& a_1
\end{pmatrix}^T,
\]
where the $T$ operator is matrix transpose.
Let \(B_n(x) = |A_n -xI_n|\),
so that 
\(B_n(x) = (-1)^n \chi_{A_n}(x)\). 
 Then 
\[
B_n(x) = \sum_{r=0}^n 
\binom{n}{r}
|A_r|x^{n-r}.
\]
\end{proposition}
The corollary is the following
\begin{theorem}\label{thm:2.3}\footnote{We codify an
argument shown to us by
J. L\'opez-Bonilla. See also the personal communication of A. Petovi\'{c} 
to
the author here \texttt{https://zenodo.org/records/21601154}.}
\rm{[J. L\'opez-Bonilla]} With $h$ and $j$ as in the above lemmas,
\[\chi((J_n(j))(x)=
\sum_{r=0}^n 
\binom{n}{r}
r!  h_r (-1)^r x^{n-r}.
\]
\end{theorem}
\begin{proof}
Determinants are invariant under the transpose,
so, setting each $a_i$  equal to $j_i$ in the
lemmas above, the proposition can be restated in the following way:
Let
\(J_n(j) \) be as in the proof of Lemma 2.1.
Let \(C_n(x) = |(J_n(j) -xI_n|\),
so that 
\(C_n(x) = (-1)^n \chi({(J_n(j))}(x)\). 
 By Proposition 1, 
\[C_n(x) =
 \sum_{r=0}^n 
\binom{n}{r}
|J_r(j)|(-1)^{n-r} x^{n-r} = \]
(by Lemma 2.2)
\[
 \sum_{r=0}^n 
\binom{n}{r}
r!  h_r(-x)^{n-r}
= \sum_{r=0}^n 
\binom{n}{r}
r!  h_r (-1)^{n-r}x^{n-r}.
\]
Simplifying, 
\(
\chi({(J_n(j))}(x)=
\sum_{r=0}^n 
\binom{n}{r}
r!  h_r (-1)^r x^{n-r}.
\)
\end{proof}
\begin{remark}
This result allows us to bypass the calculation of determinants when
computing the polynomials $\chi$.  Its implementation is
described in Remark 3 below.
\end{remark}
Now let $M$ be an $n \times n$ matrix over a field.
 For any subset $S \subseteq \{1, 2, \ldots, n\}$, let $M[S,S]$ denote the submatrix of $M$ with rows and columns indexed by $S$. The following proposition is 
well known. (For example, see equation (1.2.13) in Horn and Johnson's book ~\cite{Horn2012}.)
\begin{proposition}
The coefficient $a_k$ of $x^k$ in $\chi(M)$
satisfies
\[
a_{n-k} = (-1)^k \sum_{\substack
{\#S = k\\{S \subseteq 
\{1, 2, \ldots, n\}}}} |M[S,S]|.
\]
\end{proposition}
\begin{remark}
Let $h_0=1$ and $h_n = \tau(n)$ when $n$ is positive. Let 
$\{j_n\}_{n=1, 2, ...}$ be the sequence paired with $h$ in equation (D1) above. Now let $n$ be positive. Then 
$\chi (J_n(j))(0) = 0$ if and only if $|J_n(j)| = 0$, so, by Lemma 2.2.2,
$\chi(J_n(j))(0) = 0$ if and only if $\tau(n)=0$. Now if $\chi(J_n(j))$ is irreducible over the rational numbers, zero is not a root of $\chi(J_n(j))(x)$ and  $\tau(n) \neq 0$. A similar argument works if we consider the sequence for which $h_n = \lambda(n)$ when $n$ is positive. 

We will not focus on this kind of sufficient condition in the sequel. We mention it because it is the only such condition we know at the time of the present draft. Instead we explore a geometric phenomenon: under certain conditions, which so far we cannot crisply describe,
the eigenvalues of ``treated'' versions of the 
$\chi(J_n(x))$ (defined in the next section) appear to behave coherently as point sets as $n$ increases. We would like to gain control over this and use it to bound the distance of these eigenvalues away from zero for certain sequences $h$, then to devise a version of the argument to bound the eigenvalues of the original (``untreated'') matrices away from zero as well. This seems to be a challenging task. What we would like to do at this stage is simply to understand what it is about the $\chi(J_n(x))$ that induces the observed coherent behavior (when in fact they do so.) We have collected lists of these characteristic polynomials for some sequences $h$ (Fourier coefficients of  modular forms), which are available in data files generated in the SageMath \emph{Jupyter} notebooks. 

Except for the case of weight two modular forms attached to elliptic curves by the Modularity Theorem of Wiles and Taylor (which was analyzed by Claude, not by us\footnote{See an essay written by Claude about this under the file name \texttt{claude 32a1 notes v3} here \texttt{https://zenodo.org/records/21598903}.}), we have not found ways to predict the nature of the relevant plots simply from characteristics of the modular forms that they come from.  We can only say that, obviously, the nature of the plots are intrinsic features of the characteristic polynomials, and that the nature of the map from sequences $\{h_n\}$ to these intrinsic features is the main issue for us.\end{remark}
\section[]{Treatments
} 
Throughout this section, sequences $H$ and $J$ are to be understood as related to each other by equation (D) in Lemmas 1 and 2.
For a sequence 
$j = \{
x_n
\}_{n =1, 2, 3, ...}$, appearing in equation (D1) and a constant $c$,
we write $j^{(c)}=\{c,x_1, x_2, ...\}$.
The sequence $h^{(c)}$ is defined to be related to $j^{(c)}$ by equation (D1). (So it is \emph{not} formed by prepending $c$ to $h$.)
We refer to these objects as \emph{treatments} with $c$.
\begin{remark}
We computed many characteristic polynomials
$J_n(j^{(c)})$, and
the most straightforward workflow to find 
that object for a given $n$
 may be (concisely) diagrammed as follows:

\begin{tikzcd}
h_n
  \arrow[r, "{(D1)}"]
& j_n
  \arrow[r]
& j^{(c)}_n
  \arrow[r]
& J_n(j^{(c)}).
\end{tikzcd}

\noindent
In our experience, however,
the following workflow is more efficient: 

\begin{tikzcd}
h_n
  \arrow[r, "{(D1)}"]
& j_n
  \arrow[r]
& j^{(c)}_n
  \arrow[r, "{(D1)}"]
& h_n^{((c)}
  \arrow[r, "Th. 2.3"]
&[1em] \chi(J_n(j^{(c)})).
\end{tikzcd}

\noindent
(The dependency of $\chi(J_n(j^{(c)}))$ upon $h_n^{(c)}$
in the second workflow
is implicit in the statement of 
Theorem 2.3.)
\noindent
\end{remark}
\section{The sequence \texorpdfstring{$\{\lambda_n\}$}{\{lambda n\}}.}
  We define a sequence $j_{\lambda}$ by
requiring $j_{\lambda}(n)$ and $h_n=\lambda(n)$ together to satisfy equation (D). 
By Lemma 2.2 (2),
$\lambda(n) = |J_{\lambda,n}(j)|/n!$.
Let $\Pi_n(x) = \chi(J_{\lambda,n}(j))(x) = |xI - J_{\lambda,n}(j)|$ be the characteristic polynomial of $J_{\lambda,n}(j)$. Plots 2, 3 and 4 show 
the behavior of the minimum modulus
among the eigenvalues of $\Pi^{(c)}_n (c = 1,2, 3)$ 
and $\Pi_n$ respectively. They 
are relevant to Lehmer's question (does Ramanujan's tau function ever vanish?)
because 
the following statements are 
equivalent:\newline
(a) $\lambda(n) = \tau(p_n) = 0$, \newline
(b) $|J_{\lambda,n}(j)| = 0$, and
\newline
(c) $\Pi_n(0) = 0$. 
\begin{remark}
Writing\footnote{See \texttt{undeformed tauprime2nov25no3.ipynb} (\texttt{https://zenodo.org/records/21601531}).} $\Pi_n(x)
= \sum_{k=0}^n a_{n,n-k} x^{n-k}
$, Theorem 2.3 gives
\begin{equation}
a_{n,n-k} = (-1)^k k! \binom{n}{k}
\lambda(k)).
\end{equation}
\end{remark}
\begin{remark}
The number of terms in the sum in Proposition 2
is $\binom{n}{k}$, so the simplest way to reconcile
Proposition 2 and equation (3) is to suppose that the structural symmetries of 
$J_{\lambda, n}(j)$ somehow force
$|J_{\lambda,n}(j)[S,S]|= 
(-1)^k k! \lambda(k)$ just when $S \subseteq 
\{1, 2, \ldots, n\}$ and $\#S = k$. But this turns out to be too 
good to be true.\footnote{See \texttt{too good.ipynb} (\texttt{https://zenodo.org/records/21601751}.)} Instead, there must be
systematic cancellations
coming from those symmetries.
So far, we have not understood these cancellations.
\end{remark}
Let
$
\chi(J_{\lambda,n}^{(c)}(j)) =
\Pi_n^{(c)}(x)
$ (say.) 
Let $E_n^{(c)}$ be the 
set of zeros of $\Pi_n^{(c)}(x)$ and let $\mu_n^{(c)} = \min |v|$ with the minimum taken over $v$ in $E_n^{(c)}$; then it appears that the graph of the pairs $(n,\mu_n^{(c)})$ lies on 
an oscillating curve for $1 \le c \le 3$ (see Figures 2, 3 and 4). As a control, we repeated the calculation for $h_n = \tau(p_n+ 1)$ with $c = 1$. Figure 6 shows the resulting plot. It shows no oscillatory behavior for this choice of $h$, suggesting that the choice of prime indices \it{per se} \rm was decisive, but this is false. (See Remark 6.) 
The deep spikes in the lower, logarithmic plot in Figure 2 suggest that the approach of the roots of $\Pi_n^{(1)}(x)$ to the origin may behave in a predictable way, but the details are
not clear to us because (among other reasons) our data is limited.  We do not know the 
details of the relationship of the roots of 
$\Pi_n(x)$ to those of the $\Pi_n^{(c)}(x)$, so 
the connection, if any, to Lehmer's question is
unclear.

For contrast, we reproduce in Figure 7 a plot 
of the mimimum moduli for each $n$ for 
the untreated characteristic  
polynomials $\Pi_n(x)$. We have not attempted the sort of analysis 
we just discussed for the $\Pi_n^{(1)}(x)$.

We want a global-in-$c$ understanding of  the treated plots.
After prepending $c$ to 
$
j_{\lambda}
$, denoting this treatment as 
$
j_{\lambda}^{(c)}
$ 
and once more applying equation (D1), we computed the sequence 
$
h_{\lambda}^{(c)}
$ from 
$
j_{\lambda}^{(c)}
$. We formed the list of pairs 
$(c, n! 
h_{\lambda}^{(c)}(n))$. (These are just the values of the determinants from Lemma 2.2 in this situation; but now we bypass the determinants and gain significant time savings.) 
 We found that, 
within the range of our observations, these pairs interpolate a polynomial $D_n(c)$ (say) of degree $n$.
For given $n$, we found the minimum modulus among the roots of $D_n(c)$
and plotted these minima against $n$. In the case $c \in 
\{0, 1, 2, ...,\}$, the  result was
the same kind of wave form that displayed in Figures 2 - 4, as we see in Figure 5.\footnote{The reader will notice (if the corresponding \it{SageMath }\rm notebook is consulted) that we ran the routines twice, at 64-bit and 100-bit precision. The results were indistinguishable.} 
\section{Coherent behavior in newforms from elliptic curves} 
Because the cusp form $\Delta$ is the generating function of the multiplicative Ramanujan function, which displays coherent behavior in the range of our observations, we searched a source of cusp forms that are generating functions of other multiplicative functions: Cremona's database \cite{LMFDB} of cusp forms $\sum_n a(n) q^n$
associated to elliptic curves in virtue of the Modularity Theorem.
We found similar coherent behavior. (But see the next section for examples of coherent behavior coming from cusp forms the coefficient functions of which are not multiplicative.)

The forms in Cremona's datebase have weight two and varying
levels. In the limited range of our observations, there is much variation in the coherent behavior, or lack of it, in treated curves. 

In contrast to the inclusion of plots of the logarithms of minimum moduli among eigenvalues associated to matrices representing Ramanujan's  original tau function, we omit such plots for the curves in Cremona's database. The logarithm plots in the former situation served to demonstrate that the minima were not hitting zero, but for the Cremona database curves, they must occasionally
do so, because
the corresponding Fourier coefficients $a(n)$ themselves occasionally vanish.
\begin{remark}\label{rem:control}
Among our figures we have plots of eigenvalue minimum moduli for the sequences $h_n = a(n)$ and $h_n = a(p_n)$, but also for the sequences $h_n = a(p_n + 1)$. Our original intention was to introduce controls like the one we described above, which was meant
to provide some evidence for the thesis that the $\lambda_n = \tau(p_n)$ were producing oscillations because of arithmetic characteristics of the inputs. We expected to see a pattern in which the sequences $h_n = a(p_n)$ induced oscillations
among the minimal eigenvalues, and the sequences $h_n = a(p_n + 1)$
did not, but there is no such pattern. This undercuts the rationale behind out original controls for $h_n = \lambda_n$.
We have carried forward all  controls into this article anyway,
for obvious reasons.
\end{remark}
\section{Cusp forms from spaces of weight exceeding two}
We have begun to test for coherent
behavior among cusp forms belonging to spaces of weight
greater than two. Some Hecke eigenforms
in such spaces have non-rational Fourier expansions. Decimal estimates of their coefficients carry error that could explode in our calculations. To avoid this issue, we examined the sums of the eigenforms over the automorphisms of their coefficient fields that fix $\mathbb{Q}$ (the traces.) These objects have rational expansions.

To explain this, we sketch Hecke theory up to the definition of newforms. Denoting the space of cusp forms at level $N$ and weight $k$ as $S_k(\Gamma_0(N))$, for $f, g \in 
S_k(\Gamma_0(N))$ the Petersson inner product is defined as 
$\langle f, g \rangle = $
$$\int_{\Gamma_0(N) \backslash \mathfrak{H}} f(z) \overline{g(z)} \, y^k \, \frac{dx \, dy}{y^2}. $$
Let $S_k^{\mathrm{old}}(\Gamma_0(N))$ be the subspace of all linear combinations of $f(dz)$ where $f \in S_k(\Gamma_0(M))$ for some $M|N$ with $M < N$ and $d|(N/M)$. Let $S_k^{\mathrm{new}}(\Gamma_0(N))$ be the
orthogonal complement of $S_k^{\mathrm{old}}(\Gamma_0(N))$
with respect to the Petersson inner product. Then 
$
S_k(\Gamma_0(N)) = S_k^{\mathrm{old}}(\Gamma_0(N))
\oplus S_k^{\mathrm{new}}(\Gamma_0(N))
$.

Let $T_1$ be the identity operator on $S_k(\Gamma_0(N))$
and,
for primes $p \nmid N$, let the Hecke operator $T_p$ 
satisfy
$$
T_p\left(\sum a_n q^n\right)
=
\sum_{n\ge1} \left(a_{pn} + p^{k-1} a_{n/p}\right) q^n,
$$
where $a_{n/p} = 0$ if $p \nmid n$. For $e \ge 2$, let
$
T_{p^e} = T_p T_{p^{e-1}} - p^{k-1} T_{p^{e-2}}.
$

For prime divisors $p$ of $N$, let the operator $U_p$ satisfy
$$
U_p\left(\sum a_n q^n\right) = \sum_{n\ge1} a_{pn} q^n,
$$
and set $U_{p^e} = U_p^e$ where the exponent on the right side means repeated composition.
For  $n = \prod_p p^{e_p}$, let
$
T_n = \prod_{p \nmid N} T_{p^{e_p}} \cdot \prod_{p \mid N} U_{p^{e_p}}
$.
A \emph{newform} in $S_k^{\mathrm{new}}(\Gamma_0(N))$ is a cusp form
$f \in S_k^{\mathrm{new}}(\Gamma_0(N))$ with $a_1(f)=1$ that is a simultaneous eigenvector for all Hecke operators $T_p$ for $p \nmid N$ and for all $U_p$ such that $p \mid N$.

We consider a separate linear map to extract rational expansions from the newforms. 
For $v$ in a vector space $V$ and a linear map 
 $\mu: V \rightarrow V$,
let $\text{tr}(\mu)$ be sum of the eigenvalues of $\mu$. 
For a finite extension of fields $L/K$ viewed as
a vector space over $K$, let 
$\mu_{\alpha}(v) = \alpha v$. 
We write
$
\operatorname{Tr}_{L/K}(\alpha) = 
\text{tr}(\mu_{\alpha})$
and consider a finite extension $E/\mathbb{Q}$  
such that all of the Fourier coefficients
in the expansion of a newform $f$ lie in $E$. 
$E/\mathbb{Q}$ is separable, so  
we may write $E=\mathbb{Q}(\alpha)$ for
some $\alpha$ in $E$ by the Primitive Element Theorem. 
It is well known that
$
\operatorname{Tr}_{E/\mathbb{Q}}(\alpha)
$
is rational.\footnote{For example, combine
\cite{Orr} (Lemma 19) and 
\cite{Marcus} (corollary to Theorem $4^\prime$.)
}

In our first example, we inspected the unique-up-to-conjugacy 
newform $f$ in $S_4(\Gamma_0(11))$, 
whose coefficient field is 
$\mathbb{Q}(\alpha)$ with 
$\alpha^2 - 2\alpha - 2 = 0$ 
(so $\alpha = 1 \pm \sqrt{3}$). 
The space has dimension~2 and a single orbit 
under $\mu_{\alpha}$ of degree~2. 
We carried out our usual procedures for the sequences 
$
h_n = \operatorname{Tr}_{\mathbb{Q}(\sqrt{3})/\mathbb{Q}}\big(a_n(f)\big)
$
and
$
h^*_n = \operatorname{Tr}_{\mathbb{Q}(\sqrt{3})/\mathbb{Q}}\big(a_{p_n}(f)\big),
$
setting $h_0 = h^*_0 = 1$ to conform to the hypotheses of equation (D).
The minimum-modulus eigenvalue plots for our $J$ matrices are shown
in Figure 45. In this example the function $n \mapsto h_n$
is not multiplicative on positive $n$, and this was the first such example we found with oscillatory behavior.
\section{Other sequences}
We do not know when coherence among the minimum moduli occur in the figures below. We have searched unsuccessfully for deciding criteria among the characteristics of modular forms. Now we are looking for such criteria among the characteristic polynomials. We are guessing that such criteria, if there are any, will be easier to spot for sequences $h$ that are constructed as simply as possible.
As a beginning, in Figures 48A we display the treatments with $c = 1$ of a minimum modulus plot for
the sequence $h_0=1, h_n = p_{n+1}-p_n-2, n \ge 1$. In Figure 48B, we show the corresponding plot for the sequence 
$\{h_n\}_{n \ge 0}$ 
with $h_0 = 1$ and $h_n = p_n$ for $n \ge 1$.
In Figure 48C, we show the corresponding
plot for $\{p_{n+1}-p_n\}_{n \ge 1}$.
 In Figure 49, again treated with $c = 1$, we depict minimum modulus plots for
$\{n^a\}_{n \ge 1},  1 \le a \le 4$. 
\section{Zeros among untreated and treated \emph{h}-sequences.}
We want to understand the relationship of treated and untreated
$h$-sequences, so we counted the number of zeros of the treated and untreated $h$-sequences that we have studied empirically. Table 1 counts the zeros among length-$500$ truncations of untreated $h$-sequences Cremona's from database of the expansions of weight two modular forms attached to elliptic curves in virtue of the Modularity Theorem \cite{Wiles, TW, BCDT}. (None of the $h$-sequences treated with $c = 1$ have zeros in this range.) Table 2 counts the zeros among the coefficients with prime index for both untreated and treated $h$-sequences (with $c = 1$.) Table 3 counts the zeros among the coefficients with index $p_n + 1, n = 1, 2, ..., 500$ for both untreated and treated $h$-sequences, again with $c = 1$.\footnote{See \texttt{https://zenodo.org/records/21831972}.}

We then formed $h$-sequences corresponding to the coefficients in truncations of precision $505$ of expansions of modular forms of weights $k$, level $N$  with $2 \le k \le 20, 
1 \le N \le 20$, using native capabilities of \emph{SageMath} to find them. It found $174$ forms. We used the treatment process of Section 3 to form the finite sequences
$\{h^{(1)}_n\}_{1 \le n \le 505}$ and found no zeros in 
them.\footnote{See \texttt{https://zenodo.org/records/21855317}.} After some preliminary tests varying $c$, it appears to be possible that even though the untreated sequences 
$\{h_n\}$ 
contain zeros, the treated sequences $\{h^{(c)}_n\}$ never do; this appears to be consistent with the observation above that no treated $h$-sequence derived from Cremona's database contains a zero in the observed range.
\section{Observations on the figures.}
In Tables 4 - 8, we visually classified the collection of plots for Cremona curves in the section ``Figures'' as nearly linear (``L'') oscillating (``O''), or oscillating around a line (``L/O''). 

Recalling the  terminology of the  Introduction: with $E^{(1)}$ depending on the matrix $J^{(1)}$ appropriate to the figure under discussion, if $R = \#\{v \in E_n^{(1)} \text{ s. t. } |v| = \mu_n^{(c)}
\text{ and } v \in \mathbb{R}\}$
denotes the number 
of real characteristic-polynomial zeros of minimum modulus for an index $n$ and $C$ is the number of such zeros with non-zero imaginary part for the same index, then  Tables 4 - 8 evaluates $R$ and $C$ for $n$ sufficiently large and records the largest $n$ ($=F$ in the tables) out of $500$
that breaks the pattern.\footnote{
\url{https://zenodo.org/records/22038430},
\url{https://zenodo.org/records/22018839}, 
\url{https://zenodo.org/records/22038318}.} 
We found plausible regularities for most $n$ residue classes modulo $6$ that we have examined so far. (We made an arbitrary choice to consider triples such that $F \le 399$ to be ``plausible.'')

By now we have computed a substantial number of rows, so we venture the following 
\begin{conjecture} For any Cremona curve,
either $(R,C) = (1,0)$ or $(R,C) = (0,2)$.
\end{conjecture}
\begin{remark}
We should emphasize that $R$ and $C$ count zeros, not complex-conjugacy classes of zeros.
\end{remark}
Table 4 comes close to permitting another conjecture, but we would then need to understand how to disqualify curve 44a1 from the claim, and we're not quite there yet. (Once again, Claude has its own thoughts about this.\footnote{See {\url{https://zenodo.org/records/22038004}}.})
If we ignore the anomalous curve 44a1, it appears that for the full sequence $h_n = a(n)$ of cusp forms attached to the Cremona curves we studied, $(R,C) = (1,0)$ if and only if the associated plot oscillates. This is the first such (nearly) perfect correlation of the character of the plots with another invariant of the curves that we have seen. Two unfortunate things about this finding are that the L \emph{vs.} O classification is so far entirely subjective and that we cannot explain either side of this correlation by means of the other: neither invariant is understood (by us, anyway.) 

Apparently we can remedy one of these problems--Claude advised us to think about the zeros of the cusp forms themselves: are they real, or do they have non-vanishing imaginary parts? We tested truncations of the  series that represent these forms and found that the results pair consistently with our subjective L \emph{vs.} O classification, hence, with our $(R,C)$ calculations too (next section, which Claude wrote and the present writer edited a little.) 

The present writer is no adept at Fourier analysis, so it is not absolutely clear to him how to describe the ``O'' class objectively; the figures below show a great deal of variation among the ones he wants to classify as oscillatory. (But see the signal analyses by Claude in our \emph{SageMath} notebooks.) We do expect that statistically assessing the linearity of the plots in the ``L'' class will not be challenging, and then we will be able to class the plots as ``L \emph{vs} not-L'' to render an objective partial verdict on the character of the plots.
\section{Minimal zeros of truncated generating functions curve by curve}
\label{sec:dominantZero}
\subsection*{The sequences tested}
For each elliptic curve $E$ in the list of thirty-four Cremona labels
appearing in the Figures section, let $a(n)$ denote the Fourier
coefficients of the weight two newform attached to $E$ by the Modularity
Theorem, and let
\[
h_0 = 1, \qquad h_n = a(n) \quad (1 \le n \le 200),
\]
so that $h$ is the untreated $h$-sequence of the figures with the
constant term $h_0=1$ prepended (replacing $a(0)=0$). The coefficients
were computed by \emph{SageMath}'s \texttt{E.anlist} (PARI
\texttt{ellan}) from the Cremona label.

\subsection*{The quantity computed}
For truncation degrees $M \in \{80, 110, 140, 170, 200\}$ form the
integer polynomial
\[
P_M(t) \;=\; \sum_{n=0}^{M} h_n\, t^n .
\]
(The Hasse bound (\cite{Hasse36} and Theorem V.1.1 in \cite{Silverman2009}) $|a(n)| \le d(n)\sqrt{n}$ gives the power series
$\sum h_n t^n$ radius of convergence $1$, so zeros of modulus below $1$
are meaningful objects of the series, approximated by those of its
truncations.) All $M$ roots of $P_M$ were computed in double precision
(companion-matrix eigenvalues, NumPy) and sorted by modulus. The root
of smallest modulus, $t_0(M)$, was then polished by Newton iteration at
forty decimal digits of working precision (\texttt{mpmath}) on the exact
integer coefficients, and normalized to the closed upper half plane
($\operatorname{Im} t_0 \ge 0$; the coefficients are real, so nonreal
roots occur in conjugate pairs and this loses nothing). Write
\[
t_0 = \rho e^{i\theta}, \qquad
\rho_2 = \text{modulus of the next-nearest root of } P_{200},
\]
where the conjugate $\overline{t_0}$ is excluded from the computation
of $\rho_2$, so that $\rho_2/\rho$ measures how well the smallest zero
(or conjugate pair) is separated from all other zeros.

\subsection*{Acceptance criteria and verdicts}
The following thresholds were fixed before the computation was run.
\begin{enumerate}
\item \emph{Stability in the truncation degree.} The polished values
$t_0(140)$, $t_0(170)$, $t_0(200)$ must agree to within a relative
difference of $10^{-2}$. (In the event, they agreed to all reported
digits for every curve: for the values of $\rho$ that occurred, the
truncation tail $\rho^{M}$ at $M = 140$ is already at or below the
double-precision level, e.g.\ $0.74^{140} < 10^{-18}$.)
\item \emph{Isolation.} The separation ratio must satisfy
$\rho_2/\rho > 1.05$.
\item \emph{Reality.} The zero is declared real when
$|\operatorname{Im} t_0| < 10^{-8}\,|t_0|$ after polishing, and complex
otherwise.
\end{enumerate}
Each curve then receives exactly one of three verdicts:
\begin{itemize}
\item[\textbf{R}\,:] criteria (1) and (2) hold and the smallest zero of
$P_M$ is \emph{real} (in every observed case, real and negative, i.e.\
$\theta = \pi$);
\item[\textbf{C}\,:] criteria (1) and (2) hold and the smallest zero is
a \emph{complex conjugate pair} ($0 < \theta < \pi$);
\item[\textbf{I}\,:] criterion (1) or (2) fails --- the smallest zero
is not stably resolved or not isolated --- and no reality verdict is
issued.
\end{itemize}
(Here, the \emph{verdicts} \textbf{R}  and \textbf{C} should not be confused with the \emph{zero-counts} $R$ and $C$ in the tables preceding Table 9; we will repair the notational clash in a subsequent draft.) The verdict concerns a single global object attached to the whole
sequence $h$ --- the location of the zero of $\sum h_n t^n$ nearest the
origin --- and is computed without reference to the matrices $J_n$,
their spectra, or the plots of $\mu_n$; the pairing with the
``L''/``O'' column of Table~\ref{tab:dominantZeroPerCurve} therefore compares two independently produced
classifications of the same list of curves.

\subsection*{Results}
% ===== BEGIN machine-generated content (do not edit by hand) =====
% ===== source: make_dominant_zero_section.py; to update, re-run =====
% ===== the script and paste the regenerated file here verbatim  =====
Summary of the pairing: of the thirty four curves, thiry one received a conclusive
verdict and three were inconclusive (\texttt{44a1}, \texttt{45a1},
\texttt{49a1}; separation ratios $1.004$, $1.024$, $1.021$).
The curve \texttt{46a1}, classified L/O in Table~4, received verdict \textbf{R}
and is excluded from the count below as ambiguous on the visual side.
On all thirty rows where both classifications are unambiguous, the
correspondence $\textbf{R}\leftrightarrow\mathrm{L}$,
$\textbf{C}\leftrightarrow\mathrm{O}$ holds: thirty of thirty.
We will venture the perhaps questionable argument that (if the eight \textbf{R} verdicts were distributed at random among the thirty unambiguous
rows) the probability of exactly reproducing the set of L rows would be rather
low.\footnote{For the testing code, see \url{https://zenodo.org/records/22062616}.}
% ===== END machine-generated content =====

\clearpage
\section{Tables}
\begin{table}[htbp]
\setlength{\parindent}{0pt}
\raggedright
\begin{tabular}{l|*{7}{c}}
Cremona label & 11a1 & 14a1 & 15a1 & 17a1 & 19a1 & 20a1 & 21a1 \\
\hline
Zero count & 75 & 146 & 106 & 169 & 185 & 277 & 61 \\
\end{tabular}

\vspace{1em}
\noindent
\begin{tabular}{l|*{7}{c}}
Cremona label & 24a1 & 26a1 & 26b1 & 27a1 & 30a1 & 32a1 & 33a1 \\
\hline
Zero count & 294 & 54 & 31 & 493 & 98 & 427 & 50 \\
\end{tabular}

\vspace{1em}
\noindent
\begin{tabular}{l|*{7}{c}}
Cremona label & 34a1 & 35a1 & 36a1 & 37a1 & 37b1 & 38a1 & 38b1 \\
\hline
Zero count & 144 & 213 & 444 & 87 & 229 & 114 & 14 \\
\end{tabular}

\vspace{1em}
\noindent
\begin{tabular}{l|*{7}{c}}
Cremona label & 39a1 & 40a1 & 42a1 & 43a1 & 44a1 & 45a1 & 46a1 \\
\hline
Zero count & 68 & 321 & 41 & 103 & 272 & 238 & 155 \\
\end{tabular}

\vspace{1em}
\noindent
\begin{tabular}{l|*{6}{c}}
Cremona label & 48a1 & 49a1 & 50a1 & 50b1 & 51a1 & 53a1 \\
\hline
Zero count & 290 & 365 & 127 & 127 & 175 & 121 \\
\end{tabular}
\caption{Number of zeros among the first $500$ coefficients in the
Fourier expansions of weight two modular forms associated to elliptic curves with the given labels in virtue of the Modularity Theorem. See \texttt{https://zenodo.org/records/21831972}.}
\end{table}
\begin{table}[htbp]
\setlength{\parindent}{0pt}
\raggedright
\begin{tabular}{l|*{7}{c}}
Label & 11a1 & 14a1 & 15a1 & 17a1 & 19a1 & 20a1 & 21a1 \\
\hline
Untreated & 10 & 19 & 12 & 15 & 11 & 8 & 10 \\
Treated & 0 & 1 & 1 & 1 & 0 & 0 & 2 \\
\end{tabular}

\vspace{1em}
\begin{tabular}{l|*{7}{c}}
Label & 24a1 & 26a1 & 26b1 & 27a1 & 30a1 & 32a1 & 33a1 \\
\hline
Untreated & 13 & 16 & 9 & 256 & 23 & 256 & 15 \\
Treated & 0 & 2 & 0 & 0 & 3 & 0 & 0 \\
\end{tabular}

\vspace{1em}
\begin{tabular}{l|*{7}{c}}
Label & 34a1 & 35a1 & 36a1 & 37a1 & 37b1 & 38a1 & 38b1 \\
\hline
Untreated & 13 & 12 & 256 & 8 & 12 & 16 & 11 \\
Treated & 0 & 0 & 0 & 0 & 0 & 2 & 0 \\
\end{tabular}

\vspace{1em}
\begin{tabular}{l|*{7}{c}}
Label & 39a1 & 40a1 & 42a1 & 43a1 & 44a1 & 45a1 & 46a1 \\
\hline
Untreated & 14 & 8 & 13 & 9 & 11 & 13 & 11 \\
Treated & 0 & 0 & 0 & 0 & 0 & 0 & 1 \\
\end{tabular}

\vspace{1em}
\begin{tabular}{l|*{6}{c}}
Label & 48a1 & 49a1 & 50a1 & 50b1 & 51a1 & 53a1 \\
\hline
Untreated & 13 & 255 & 11 & 11 & 12 & 12 \\
Treated & 0 & 0 & 2 & 0 & 2 & 0 \\
\end{tabular}
\caption{Number of zeros of treated and untreated $h$-sequences among the first $500$ coefficients of index $p_n, n = 1, 2, ..., 500$ in the
Fourier expansions of weight two modular forms associated to elliptic curves with the given labels in virtue of the Modularity Theorem. See \texttt{https://zenodo.org/records/21831972}.}
\end{table}
\begin{table}[htbp]
\setlength{\parindent}{0pt}
\raggedright
\begin{tabular}{l|*{7}{c}}
Label & 11a1 & 14a1 & 15a1 & 17a1 & 19a1 & 20a1 & 21a1 \\
\hline
Untreated & 108 & 162 & 114 & 224 & 342 & 502 & 62 \\
Treated & 1 & 0 & 1 & 0 & 0 & 0 & 0 \\
\end{tabular}

\vspace{1em}
\begin{tabular}{l|*{7}{c}}
Label & 24a1 & 26a1 & 26b1 & 27a1 & 30a1 & 32a1 & 33a1 \\
\hline
Untreated & 502 & 53 & 22 & 502 & 96 & 503 & 59 \\
Treated & 1 & 0 & 0 & 0 & 0 & 0 & 1 \\
\end{tabular}

\vspace{1em}
\begin{tabular}{l|*{7}{c}}
Label & 34a1 & 35a1 & 36a1 & 37a1 & 37b1 & 38a1 & 38b1 \\
\hline
Untreated & 155 & 338 & 503 & 120 & 374 & 132 & 12 \\
Treated & 0 & 0 & 0 & 0 & 0 & 0 & 5 \\
\end{tabular}

\vspace{1em}
\begin{tabular}{l|*{7}{c}}
Label & 39a1 & 40a1 & 42a1 & 43a1 & 44a1 & 45a1 & 46a1 \\
\hline
Untreated & 71 & 503 & 45 & 139 & 502 & 316 & 214 \\
Treated & 1 & 0 & 5 & 0 & 0 & 0 & 0 \\
\end{tabular}

\vspace{1em}
\begin{tabular}{l|*{6}{c}}
Label & 48a1 & 49a1 & 50a1 & 50b1 & 51a1 & 53a1 \\
\hline
Untreated & 502 & 427 & 147 & 147 & 334 & 141 \\
Treated & 0 & 0 & 0 & 5 & 0 & 0 \\
\end{tabular}
\caption{Number of zeros of treated and untreated $h$-sequences among the first $500$ coefficients of index $p_n + 1, n = 1, 2, ..., 500$ in the
Fourier expansions of weight two modular forms associated to elliptic curves with the given labels in virtue of the Modularity Theorem. See \texttt{https://zenodo.org/records/21831972}.}
\end{table}
\clearpage
\clearpage
\begingroup
\small
\renewcommand{\arraystretch}{0.88}
\setlength{\tabcolsep}{4.7pt}
\setlength{\LTcapwidth}{\textwidth}
\begin{longtable}{|c|c||c|c||c|}
\caption[]{
Classification of the plots for Figures 8 - 41 for 
sufficiently large $n \in \{1, 2, ..., 500\}$ 
for the sequence $h_n = a(n)$.
\thinspace The spectrum $E_n^{(1)}$
depends on both the sequence $\{h\}$ and the value of $h_n$. 
The descriptor ``character of plots'' 
is listed as L or O, depending on whether a plot 
appears (to us!) linear or oscillatory.
The parameter $F$ is (in terms of the table's claims) the largest exceptional $n$. If \emph{SageMath} detects no exceptions in a given
row, $F$ gives an X. If it detects no pattern in that row, Column 3 gives an N.
In the following legend: \newline \newline
$\mu_n = \min \{|v|$ for $v \in  E_n^{(1)}\}$;
\newline
$R =\#\{ v \in E_n^{(1)} \text{ s.t. } |v|=\mu_n, \Im(v) = 0\}$;
\newline
$C =\#\{ v \in E_n^{(1)} \text{ s.t. } |v|=\mu_n, \Im(v) \ne 0\}$; 
\newline
\newline
The columns are as follows.
\newline
Column 1: Figure.\newline
Column 2: Curve label in \cite{LMFDB}.\newline
Column 3: The ordered triple $(R, C, F)$. \newline
Column 4: Character of plots.\newline
Column 5: Structure of the Mordell--Weil group $E(\mathbb{Q})$ of the curve, displayed as $\mathbb{Z}^{r} \times T$ where $r$ is the rank and $T$ is the torsion subgroup, with $T = 0$ when the torsion is trivial.
}
\label{tab:classificationCurvesOnce} \\
\hline
1 & 2 & 3 & 4 & 5 \\
\hline
\endfirsthead
\hline
1 & 2 & 3 & 4 & 5 \\
\hline
\endhead
\endfoot
\endlastfoot
8A & 11a1 & (1, 0, 41) & O & $\mathbb{Z}^{0} \times \mathbb{Z}/5\mathbb{Z}$ \\
\hline
9A & 14a1 & (1, 0, 24) & O & $\mathbb{Z}^{0} \times \mathbb{Z}/6\mathbb{Z}$ \\
\hline
10A & 15a1 & (1, 0, 194) & O & $\mathbb{Z}^{0} \times \mathbb{Z}/2\mathbb{Z} \times \mathbb{Z}/4\mathbb{Z}$ \\
\hline
11A & 17a1 & (0, 2, 42) & L & $\mathbb{Z}^{0} \times \mathbb{Z}/4\mathbb{Z}$ \\
\hline
12A & 19a1 & (1, 0, 12) & O & $\mathbb{Z}^{0} \times \mathbb{Z}/3\mathbb{Z}$ \\
\hline
13A & 20a1 & (1, 0, 18) & O & $\mathbb{Z}^{0} \times \mathbb{Z}/6\mathbb{Z}$ \\
\hline
14A & 21a1 & (0, 2, 62) & L & $\mathbb{Z}^{0} \times \mathbb{Z}/2\mathbb{Z} \times \mathbb{Z}/4\mathbb{Z}$ \\
\hline
15A & 24a1 & (1, 0, 28) & O & $\mathbb{Z}^{0} \times \mathbb{Z}/2\mathbb{Z} \times \mathbb{Z}/4\mathbb{Z}$ \\
\hline
\end{longtable}
\endgroup
\clearpage
\begingroup
\small
\renewcommand{\arraystretch}{0.88}
\setlength{\tabcolsep}{4.7pt}
\setlength{\LTcapwidth}{\textwidth}
\begin{longtable}{|c|c||c|c|c||c|}
\caption[]{
Classification of the plots for Figures 8 - 41 for 
sufficiently large $n \in \{1, 2, ..., 500\}$ 
for the sequence $h_n = a(p_n)$.
\thinspace The spectrum $E_n^{(1)}$
depends on both the sequence $\{h\}$ and the value of $h_n$. 
The descriptor ``character of plots'' 
is listed as L or O, depending on whether a plot 
appears (to us!) linear or oscillatory.
The parameter $F$ is (in terms of the table's claims) the largest exceptional $n$. If \emph{SageMath} detects no exceptions in a given
row, $F$ gives an X. If it detects no pattern in that row, Column 4 gives an N.
In the following legend: \newline \newline
$\mu_n = \min \{|v|$ for $v \in  E_n^{(1)}\}$;
\newline
$R =\#\{ v \in E_n^{(1)} \text{ s.t. } |v|=\mu_n, \Im(v) = 0\}$;
\newline
$C =\#\{ v \in E_n^{(1)} \text{ s.t. } |v|=\mu_n, \Im(v) \ne 0\}$; 
\newline
\newline
The columns are as follows.
\newline
Column 1: Figure.\newline
Column 2: Curve.\newline
Column 3: Residue class of index $n$ modulo 2.\newline
Column 4: The ordered triple $(R, C, F)$. \newline
Column 5: Character of plots.\newline
Column 6: Structure of the Mordell--Weil group $E(\mathbb{Q})$ of the curve, displayed as $\mathbb{Z}^{r} \times T$ where $r$ is the rank and $T$ is the torsion subgroup, with $T = 0$ when the torsion is trivial.
}
\label{tab:classification2Bmod2} \\
\hline
1 & 2 & 3 & 4 & 5 & 6 \\
\hline
\endfirsthead
\hline
1 & 2 & 3 & 4 & 5 & 6 \\
\hline
\endhead
\endfoot
\endlastfoot
8B & 11a1 & $n \equiv 0$ & (0, 2, 30) & L & $\mathbb{Z}^{0} \times \mathbb{Z}/5\mathbb{Z}$ \\
\hline
8B & 11a1 & $n \equiv 1$ & (0, 2, 17) & L & $\mathbb{Z}^{0} \times \mathbb{Z}/5\mathbb{Z}$ \\
\hline
9B & 14a1 & $n \equiv 0$ & (1, 0, 6) & L & $\mathbb{Z}^{0} \times \mathbb{Z}/6\mathbb{Z}$ \\
\hline
9B & 14a1 & $n \equiv 1$ & (0, 2, 3) & L & $\mathbb{Z}^{0} \times \mathbb{Z}/6\mathbb{Z}$ \\
\hline
10B & 15a1 & $n \equiv 0$ & (1, 0, 6) & L & $\mathbb{Z}^{0} \times \mathbb{Z}/2\mathbb{Z} \times \mathbb{Z}/4\mathbb{Z}$ \\
\hline
10B & 15a1 & $n \equiv 1$ & (0, 2, 3) & L & $\mathbb{Z}^{0} \times \mathbb{Z}/2\mathbb{Z} \times \mathbb{Z}/4\mathbb{Z}$ \\
\hline
11B & 17a1 & $n \equiv 0$ & (1, 0, 6) & O & $\mathbb{Z}^{0} \times \mathbb{Z}/4\mathbb{Z}$ \\
\hline
11B & 17a1 & $n \equiv 1$ & (0, 2, 3) & O & $\mathbb{Z}^{0} \times \mathbb{Z}/4\mathbb{Z}$ \\
\hline
\end{longtable}
\endgroup
\clearpage
\begingroup
\small
\renewcommand{\arraystretch}{0.88}
\setlength{\tabcolsep}{4.7pt}
\setlength{\LTcapwidth}{\textwidth}
\begin{longtable}{|c|c||c|c|c||c|}
\caption[]{
Classification of the plots for Figures 8 - 41 for 
sufficiently large $n \in \{1, 2, ..., 500\}$ 
for the sequence $h_n = a(p_n)$.
\thinspace The spectrum $E_n^{(1)}$
depends on both the sequence $\{h\}$ and the value of $h_n$. 
The descriptor ``character of plots'' 
is listed as L or O, depending on whether a plot 
appears (to us!) linear or oscillatory.
The parameter $F$ is (in terms of the table's claims) the largest exceptional $n$. If \emph{SageMath} detects no exceptions in a given
row, $F$ gives an X. If it detects no pattern in that row, Column 4 gives an N.
In the following legend: \newline \newline
$\mu_n = \min \{|v|$ for $v \in  E_n^{(1)}\}$;
\newline
$R =\#\{ v \in E_n^{(1)} \text{ s.t. } |v|=\mu_n, \Im(v) = 0\}$;
\newline
$C =\#\{ v \in E_n^{(1)} \text{ s.t. } |v|=\mu_n, \Im(v) \ne 0\}$; 
\newline
\newline
The columns are as follows.
\newline
Column 1: Figure.\newline
Column 2: Curve.\newline
Column 3: Residue class of index $n$ modulo 6.\newline
Column 4: The ordered triple $(R, C, F)$. \newline
Column 5: Character of plots.\newline
Column 6: Structure of the Mordell--Weil group $E(\mathbb{Q})$ of the curve, displayed as $\mathbb{Z}^{r} \times T$ where $r$ is the rank and $T$ is the torsion subgroup, with $T = 0$ when the torsion is trivial.
}
\label{tab:classification2Bmod6} \\
\hline
1 & 2 & 3 & 4 & 5 & 6 \\
\hline
\endfirsthead
\hline
1 & 2 & 3 & 4 & 5 & 6 \\
\hline
\endhead
\endfoot
\endlastfoot
8B & 11a1 & $n \equiv 0$ & (0, 2, 30) & L & $\mathbb{Z}^{0} \times \mathbb{Z}/5\mathbb{Z}$ \\
\hline
8B & 11a1 & $n \equiv 1$ & (0, 2, 13) & L & $\mathbb{Z}^{0} \times \mathbb{Z}/5\mathbb{Z}$ \\
\hline
8B & 11a1 & $n \equiv 2$ & (0, 2, 26) & L & $\mathbb{Z}^{0} \times \mathbb{Z}/5\mathbb{Z}$ \\
\hline
8B & 11a1 & $n \equiv 3$ & (0, 2, 15) & L & $\mathbb{Z}^{0} \times \mathbb{Z}/5\mathbb{Z}$ \\
\hline
8B & 11a1 & $n \equiv 4$ & (0, 2, 28) & L & $\mathbb{Z}^{0} \times \mathbb{Z}/5\mathbb{Z}$ \\
\hline
8B & 11a1 & $n \equiv 5$ & (0, 2, 17) & L & $\mathbb{Z}^{0} \times \mathbb{Z}/5\mathbb{Z}$ \\
\hline
9B & 14a1 & $n \equiv 0$ & (1, 0, 6) & L & $\mathbb{Z}^{0} \times \mathbb{Z}/6\mathbb{Z}$ \\
\hline
9B & 14a1 & $n \equiv 1$ & (0, 2, 1) & L & $\mathbb{Z}^{0} \times \mathbb{Z}/6\mathbb{Z}$ \\
\hline
\end{longtable}
\endgroup
\clearpage
\begingroup
\small
\renewcommand{\arraystretch}{0.88}
\setlength{\tabcolsep}{4.7pt}
\setlength{\LTcapwidth}{\textwidth}
\begin{longtable}{|c|c||c|c|c||c|}
\caption[]{
Classification of the plots for Figures 8 - 41 for 
sufficiently large $n \in \{1, 2, ..., 500\}$ 
for the sequence $h_n = a(p_n + 1)$.
\thinspace The spectrum $E_n^{(1)}$
depends on both the sequence $\{h\}$ and the value of $h_n$. 
The descriptor ``character of plots'' 
is listed as L or O, depending on whether a plot 
appears (to us!) linear or oscillatory.
The parameter $F$ is (in terms of the table's claims) the largest exceptional $n$. If \emph{SageMath} detects no exceptions in a given
row, $F$ gives an X. If it detects no pattern in that row, Column 4 gives an N.
In the following legend: \newline \newline
$\mu_n = \min \{|v|$ for $v \in  E_n^{(1)}\}$;
\newline
$R =\#\{ v \in E_n^{(1)} \text{ s.t. } |v|=\mu_n, \Im(v) = 0\}$;
\newline
$C =\#\{ v \in E_n^{(1)} \text{ s.t. } |v|=\mu_n, \Im(v) \ne 0\}$; 
\newline
\newline
The columns are as follows.
\newline
Column 1: Figure.\newline
Column 2: Curve.\newline
Column 3: Residue class of index $n$ modulo 2.\newline
Column 4: The ordered triple $(R, C, F)$. \newline
Column 5: Character of plots.\newline
Column 6: Structure of the Mordell--Weil group $E(\mathbb{Q})$ of the curve, displayed as $\mathbb{Z}^{r} \times T$ where $r$ is the rank and $T$ is the torsion subgroup, with $T = 0$ when the torsion is trivial.
}
\label{tab:classification2Cmod2} \\
\hline
1 & 2 & 3 & 4 & 5 & 6 \\
\hline
\endfirsthead
\hline
1 & 2 & 3 & 4 & 5 & 6 \\
\hline
\endhead
\endfoot
\endlastfoot
8C & 11a1 & $n \equiv 0$ & (1, 0, 4) & O & $\mathbb{Z}^{0} \times \mathbb{Z}/5\mathbb{Z}$ \\
\hline
8C & 11a1 & $n \equiv 1$ & (1, 0, 3) & O & $\mathbb{Z}^{0} \times \mathbb{Z}/5\mathbb{Z}$ \\
\hline
9C & 14a1 & $n \equiv 0$ & (1, 0, X) & O & $\mathbb{Z}^{0} \times \mathbb{Z}/6\mathbb{Z}$ \\
\hline
9C & 14a1 & $n \equiv 1$ & (1, 0, 3) & O & $\mathbb{Z}^{0} \times \mathbb{Z}/6\mathbb{Z}$ \\
\hline
10C & 15a1 & $n \equiv 0$ & (1, 0, 12) & O & $\mathbb{Z}^{0} \times \mathbb{Z}/2\mathbb{Z} \times \mathbb{Z}/4\mathbb{Z}$ \\
\hline
10C & 15a1 & $n \equiv 1$ & (1, 0, 7) & O & $\mathbb{Z}^{0} \times \mathbb{Z}/2\mathbb{Z} \times \mathbb{Z}/4\mathbb{Z}$ \\
\hline
11C & 17a1 & $n \equiv 0$ & (1, 0, 18) & O & $\mathbb{Z}^{0} \times \mathbb{Z}/4\mathbb{Z}$ \\
\hline
11C & 17a1 & $n \equiv 1$ & (1, 0, 13) & O & $\mathbb{Z}^{0} \times \mathbb{Z}/4\mathbb{Z}$ \\
\hline
\end{longtable}
\endgroup
\clearpage
\begingroup
\small
\renewcommand{\arraystretch}{0.88}
\setlength{\tabcolsep}{4.7pt}
\setlength{\LTcapwidth}{\textwidth}
\begin{longtable}{|c|c||c|c|c||c|}
\caption[]{
Classification of the plots for Figures 8 - 41 for 
sufficiently large $n \in \{1, 2, ..., 500\}$ 
for the sequence $h_n = a(p_n + 1)$.
\thinspace The spectrum $E_n^{(1)}$
depends on both the sequence $\{h\}$ and the value of $h_n$. 
The descriptor ``character of plots'' 
is listed as L or O, depending on whether a plot 
appears (to us!) linear or oscillatory.
The parameter $F$ is (in terms of the table's claims) the largest exceptional $n$. If \emph{SageMath} detects no exceptions in a given
row, $F$ gives an X. If it detects no pattern in that row, Column 4 gives an N.
In the following legend: \newline \newline
$\mu_n = \min \{|v|$ for $v \in  E_n^{(1)}\}$;
\newline
$R =\#\{ v \in E_n^{(1)} \text{ s.t. } |v|=\mu_n, \Im(v) = 0\}$;
\newline
$C =\#\{ v \in E_n^{(1)} \text{ s.t. } |v|=\mu_n, \Im(v) \ne 0\}$; 
\newline
\newline
The columns are as follows.
\newline
Column 1: Figure.\newline
Column 2: Curve.\newline
Column 3: Residue class of index $n$ modulo 3.\newline
Column 4: The ordered triple $(R, C, F)$. \newline
Column 5: Character of plots.\newline
Column 6: Structure of the Mordell--Weil group $E(\mathbb{Q})$ of the curve, displayed as $\mathbb{Z}^{r} \times T$ where $r$ is the rank and $T$ is the torsion subgroup, with $T = 0$ when the torsion is trivial.
}
\label{tab:classification2Cmod3} \\
\hline
1 & 2 & 3 & 4 & 5 & 6 \\
\hline
\endfirsthead
\hline
1 & 2 & 3 & 4 & 5 & 6 \\
\hline
\endhead
\endfoot
\endlastfoot
8C & 11a1 & $n \equiv 0$ & (1, 0, 3) & O & $\mathbb{Z}^{0} \times \mathbb{Z}/5\mathbb{Z}$ \\
\hline
8C & 11a1 & $n \equiv 1$ & (1, 0, 4) & O & $\mathbb{Z}^{0} \times \mathbb{Z}/5\mathbb{Z}$ \\
\hline
8C & 11a1 & $n \equiv 2$ & (1, 0, X) & O & $\mathbb{Z}^{0} \times \mathbb{Z}/5\mathbb{Z}$ \\
\hline
9C & 14a1 & $n \equiv 0$ & (1, 0, 3) & O & $\mathbb{Z}^{0} \times \mathbb{Z}/6\mathbb{Z}$ \\
\hline
9C & 14a1 & $n \equiv 1$ &  & O & $\mathbb{Z}^{0} \times \mathbb{Z}/6\mathbb{Z}$ \\
\hline
9C & 14a1 & $n \equiv 2$ &  & O & $\mathbb{Z}^{0} \times \mathbb{Z}/6\mathbb{Z}$ \\
\hline
10C & 15a1 & $n \equiv 0$ &  & O & $\mathbb{Z}^{0} \times \mathbb{Z}/2\mathbb{Z} \times \mathbb{Z}/4\mathbb{Z}$ \\
\hline
10C & 15a1 & $n \equiv 1$ &  & O & $\mathbb{Z}^{0} \times \mathbb{Z}/2\mathbb{Z} \times \mathbb{Z}/4\mathbb{Z}$ \\
\hline
\end{longtable}
\endgroup

% Table 9: per-curve smallest-zero data (generated file)
% ===== BEGIN machine-generated content (do not edit by hand) =====
% ===== source: make_dominant_zero_section.py; to update, re-run =====
% ===== the script and paste the regenerated file here verbatim  =====
\begingroup\small
\setlength{\tabcolsep}{5pt}
\begin{longtable}{|c|c|c|c|c|c|}
\caption{The smallest zero $t_0=\rho e^{i\theta}$ of
$P_M(t)=\sum_{n\le M}h_nt^n$, $h_0=1$, $h_n=a(n)$, at $M=200$,
paired with the classification of Table~4. Verdicts: R = real
smallest zero; C = complex-pair smallest zero; I = inconclusive
(smallest zero not stably resolved or not isolated).}
\label{tab:dominantZeroPerCurve} \\
\hline
Label & $\rho$ & $\theta/\pi$ & $\rho_2/\rho$ & Verdict & Table 4 \\
\hline
\endfirsthead
\hline
Label & $\rho$ & $\theta/\pi$ & $\rho_2/\rho$ & Verdict & Table 4 \\
\hline
\endhead
\endfoot
\endlastfoot
\texttt{11a1} & 0.6296 & 0.9154 & 1.259 & C & O \\
\hline
\texttt{14a1} & 0.6441 & 0.8210 & 1.270 & C & O \\
\hline
\texttt{15a1} & 0.7074 & 0.9675 & 1.166 & C & O \\
\hline
\texttt{17a1} & 0.6279 & 1 & 1.236 & R & L \\
\hline
\texttt{19a1} & 0.6547 & 0.6470 & 1.160 & C & O \\
\hline
\texttt{20a1} & 0.6555 & 0.7600 & 1.302 & C & O \\
\hline
\texttt{21a1} & 0.5494 & 1 & 1.495 & R & L \\
\hline
\texttt{24a1} & 0.6845 & 0.8104 & 1.271 & C & O \\
\hline
\end{longtable}
\endgroup
% ===== END machine-generated content =====

\clearpage
\clearpage
\section{Figures}
Except for plots referred to as visualizations, supporting data is in SageMath notebooks listed in the index of figures after the bibliography. The plots referred to as visualizations were generated in the notebooks stored at \texttt{https://zenodo.org/records/21766174}. Points in the complex upper half-plane that are mapped by the underlying modular form to the closure of the unit disk are color-coded according to the quadrants where the images of the colored points lie, so that the junction of four brightly-colored regions at a point indicates that the point is a zero of the form. Points that land in the exterior of the unit disk are colored white, black, or one of two shades of gray according to the
quadrants of their images.
\begin{figure}[H] % Fig 1
    \centering
    \begin{subfigure}[t]{0.48\textwidth}
        \centering
        \includegraphics[width=\textwidth]{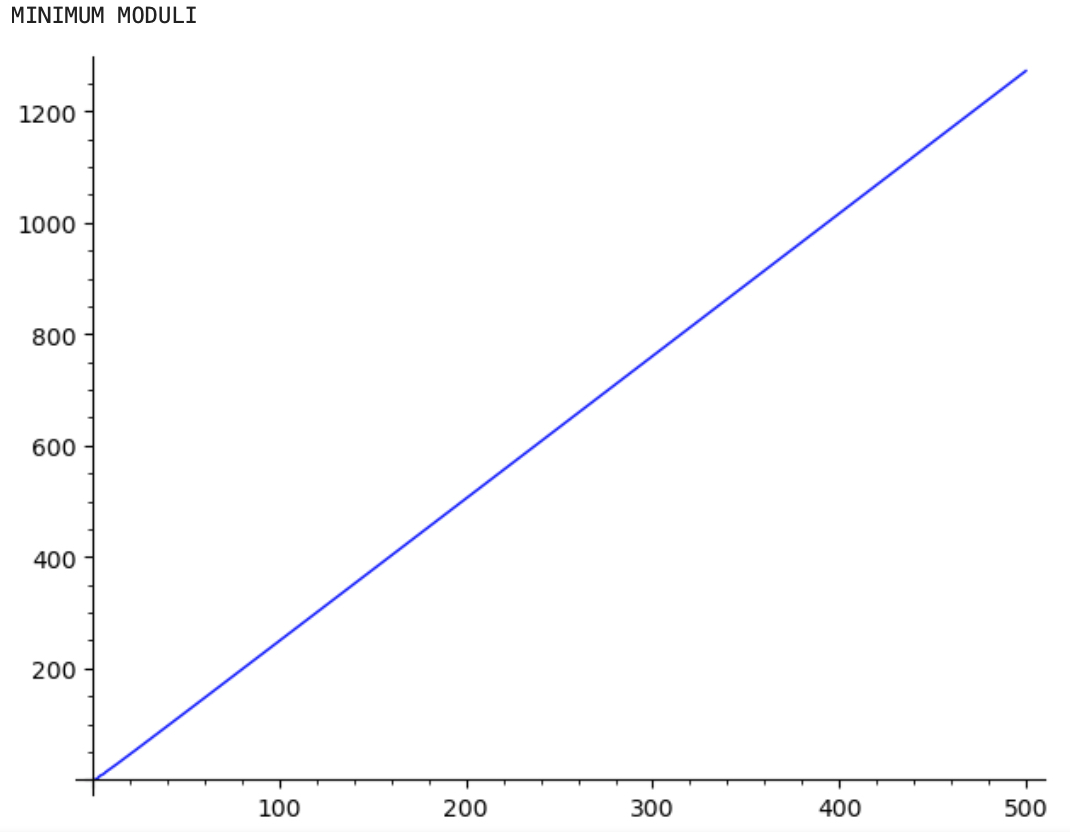}
        \caption{$c=0$.}
        \label{fig:def_tau}
    \end{subfigure}
     \begin{subfigure}[t]{0.48\textwidth}
        \centering
        \includegraphics[width=\textwidth]{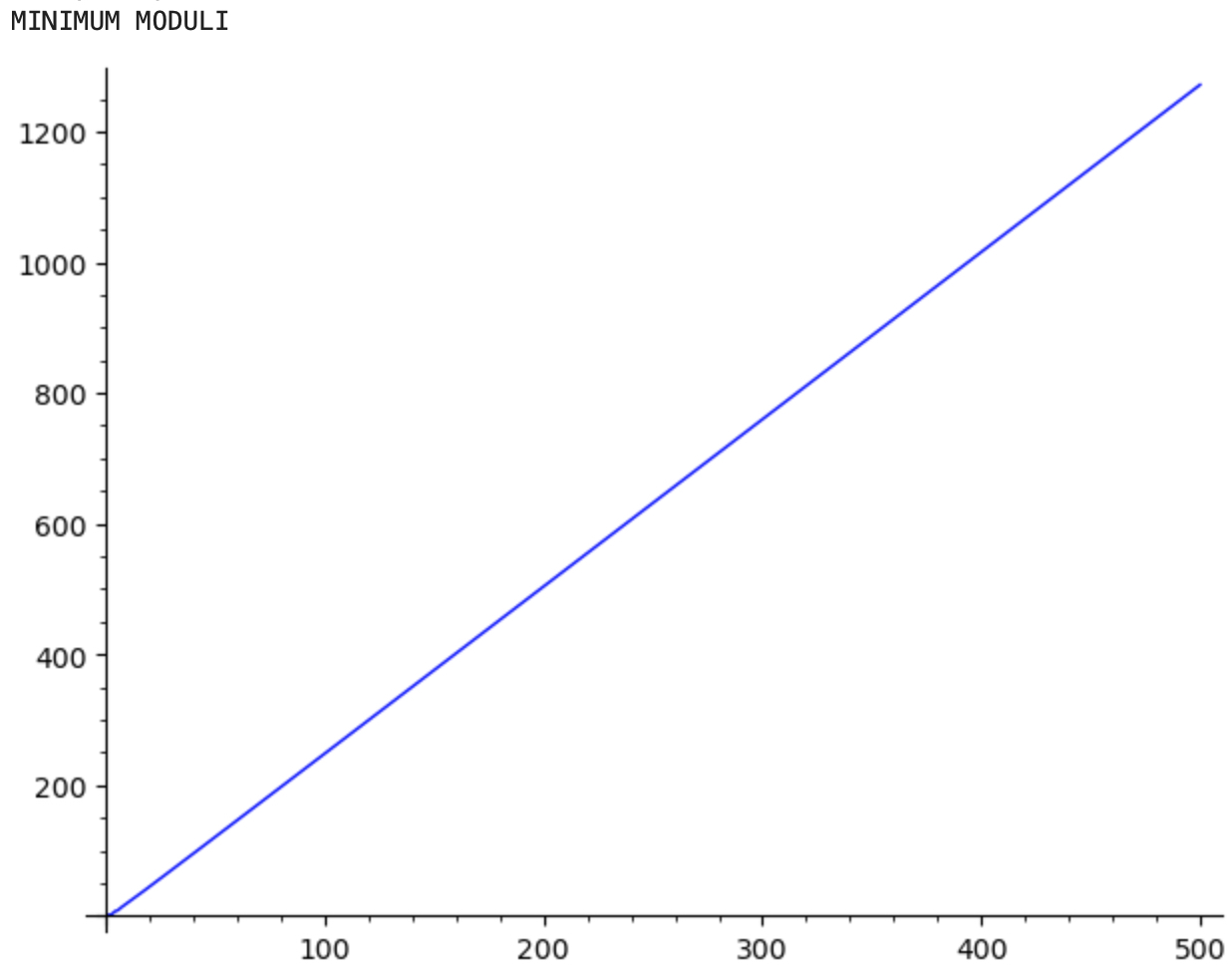}
        \caption{$c=1$.}
        \label{fig:c1_tau}
    \end{subfigure}
    \begin{subfigure}[t]{0.48\textwidth}
        \centering
        \includegraphics[width=\textwidth]{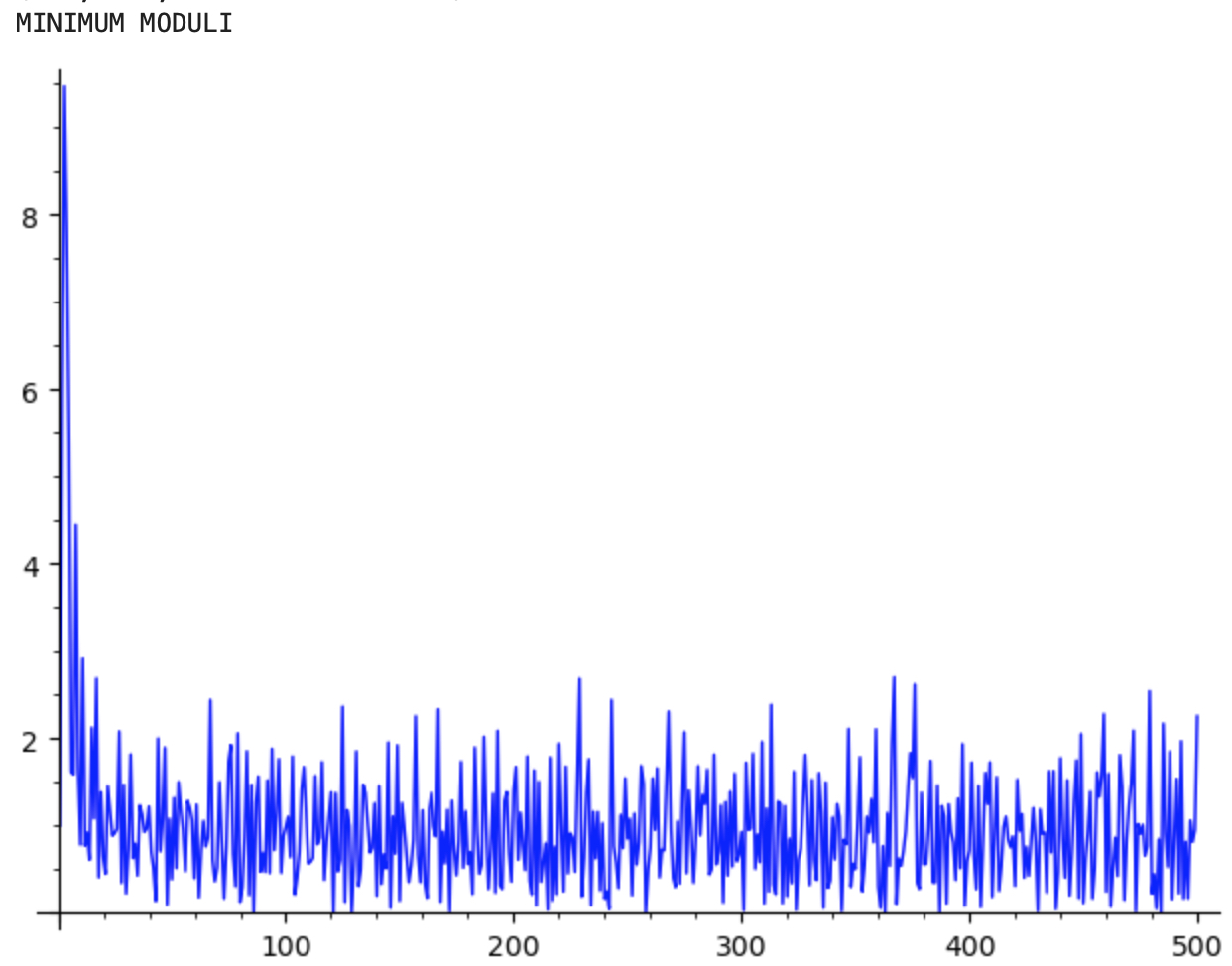}
        \caption{Untreated.}
        \label{fig:no_c}
    \end{subfigure}
    \caption{Minimum moduli for $h_n = \tau(n)$.}
\end{figure}
\begin{figure}[H]%fig 2
    \centering
\includegraphics[width=1\textwidth]{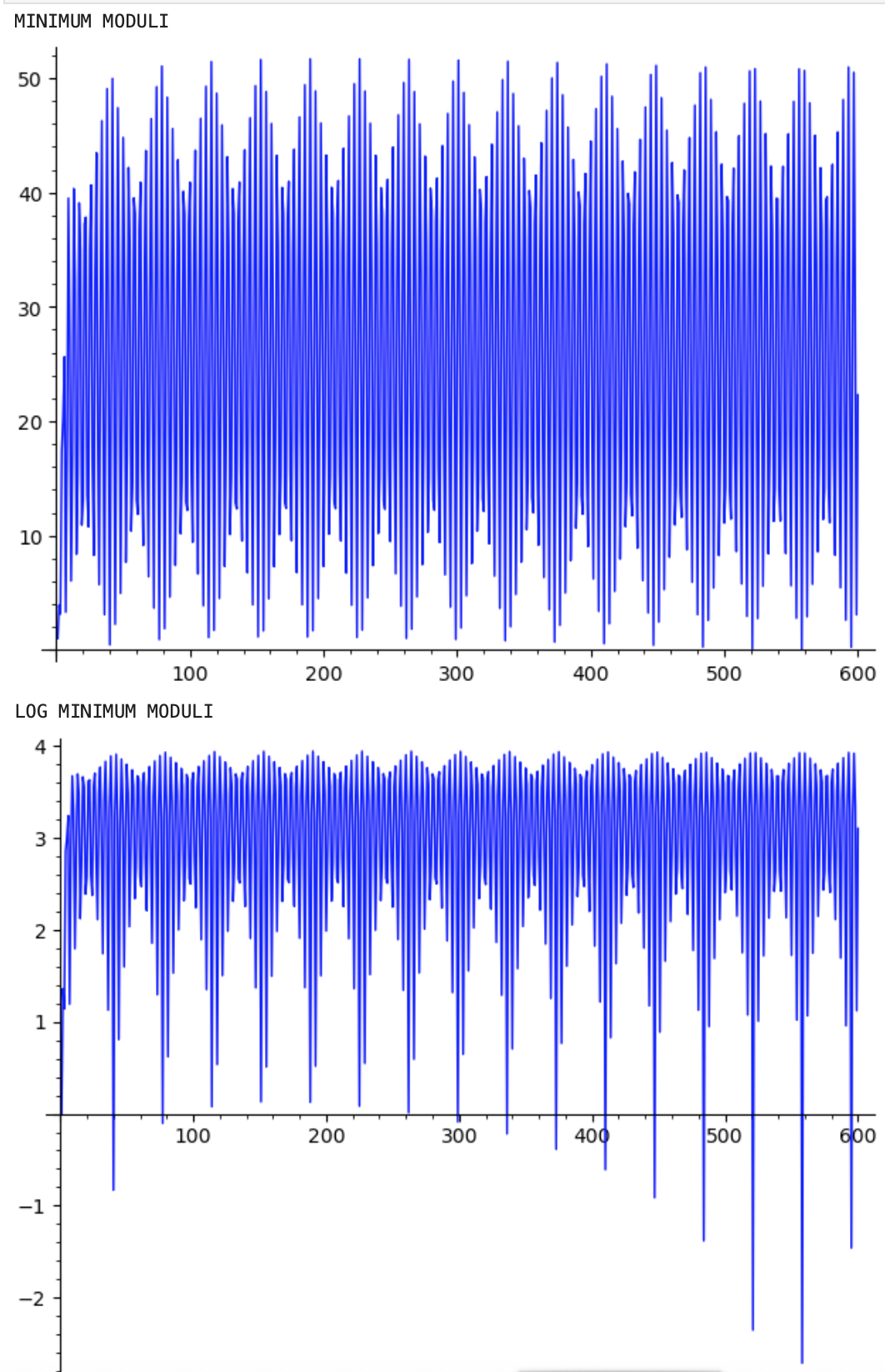}
    \caption{$c = 1, h_n = \lambda_n$}
\end{figure}
\begin{figure}[H] %fig 3
    \centering
    %\hspace*{0em}
    \includegraphics[width=1\textwidth]{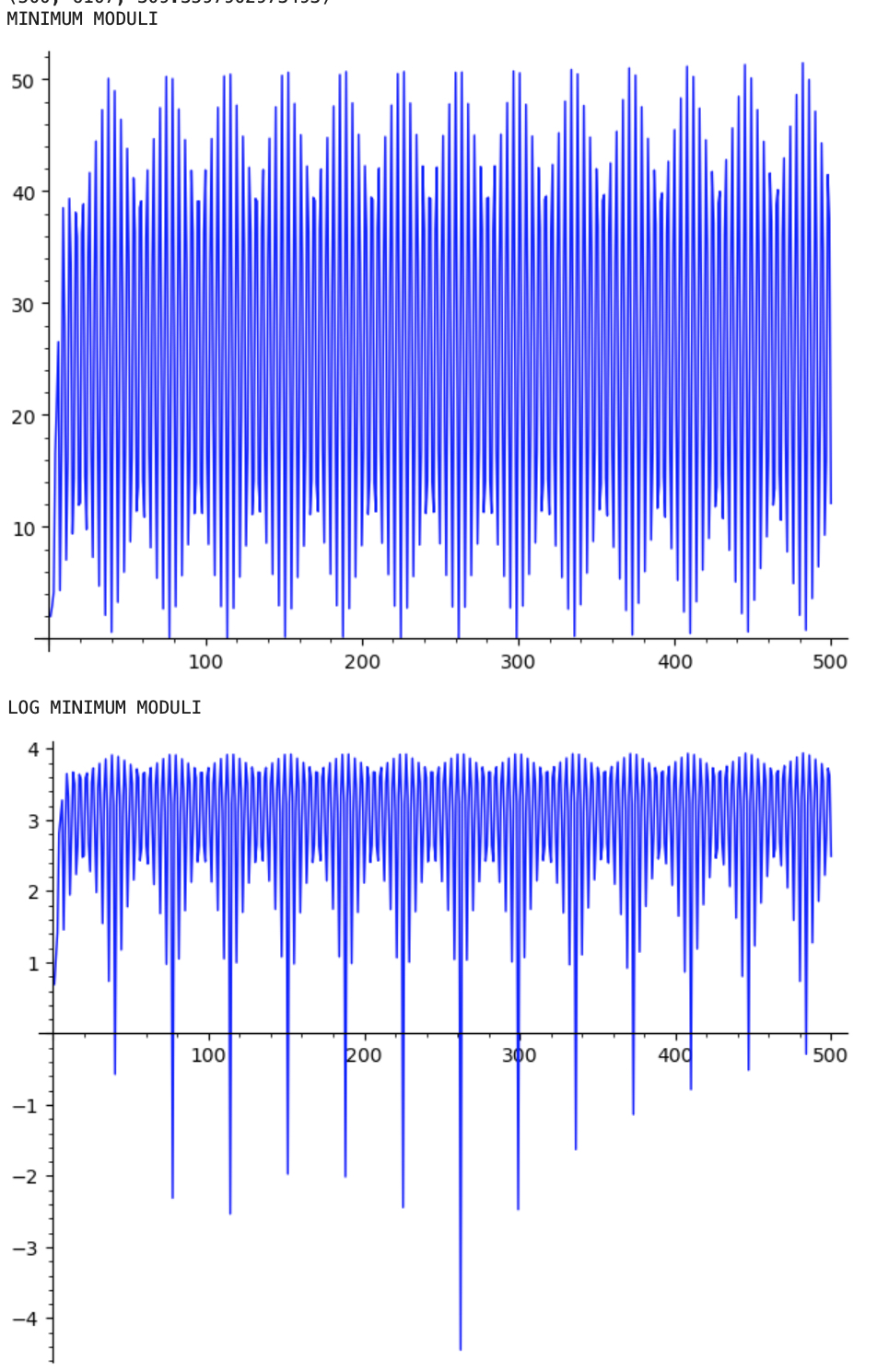}
    \caption{Minimum moduli for $h_n = \lambda_n$
   with $c = 2$.} 
    \label{fig:tauprime_c_2}
\end{figure}
\begin{figure}[H] % fig 4
    \centering
    %\hspace*{0em}
    \includegraphics[width=1\textwidth]{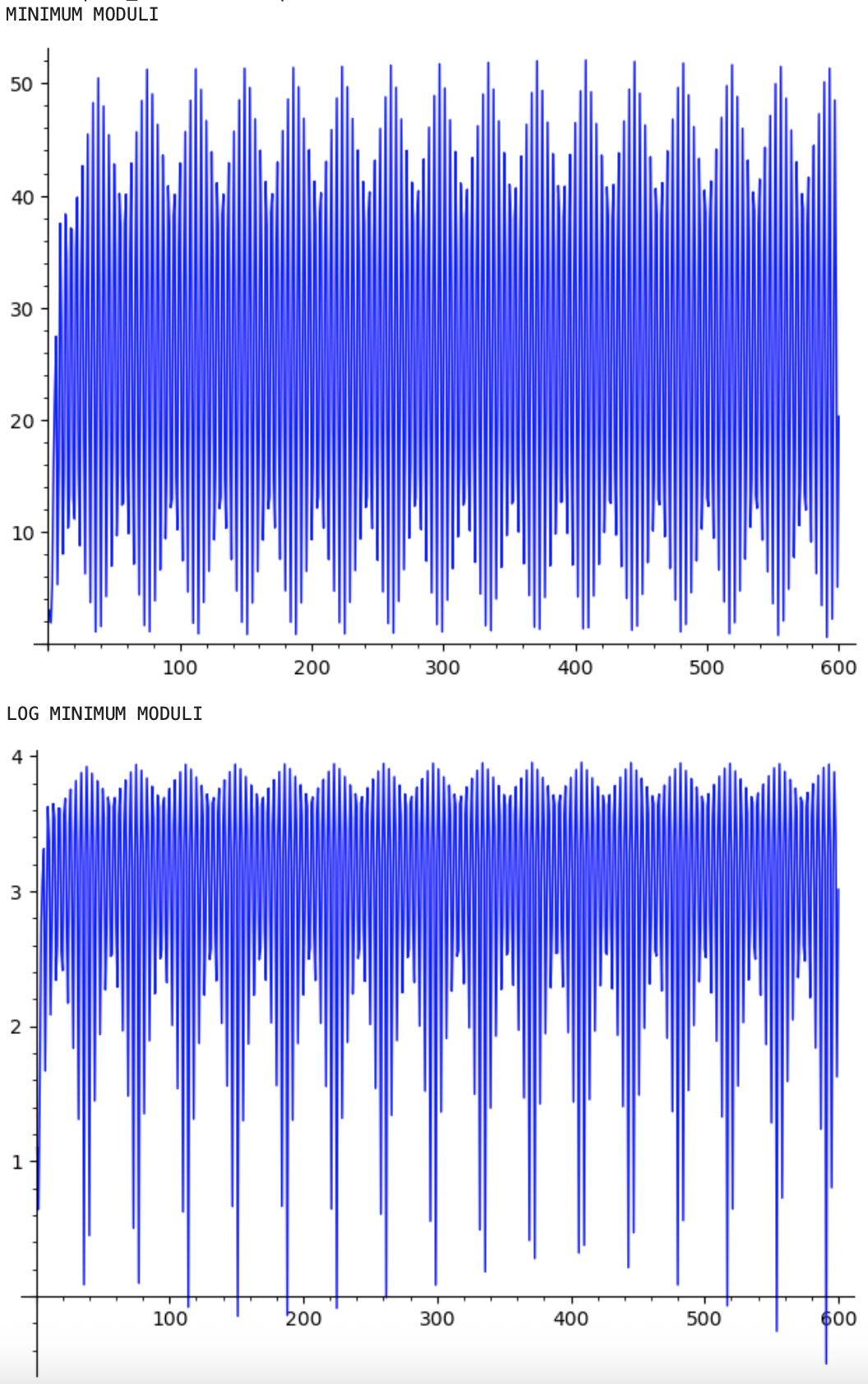}
    \caption{Minimum moduli and their logarithms for $h_n = \lambda_n$
   with $c = 3$.} 
    \label{fig:tauprime_c_3}
\end{figure}
\begin{figure}[H] % fig 5
    \centering
    %\hspace*{0em}
    \includegraphics[width=1\textwidth]{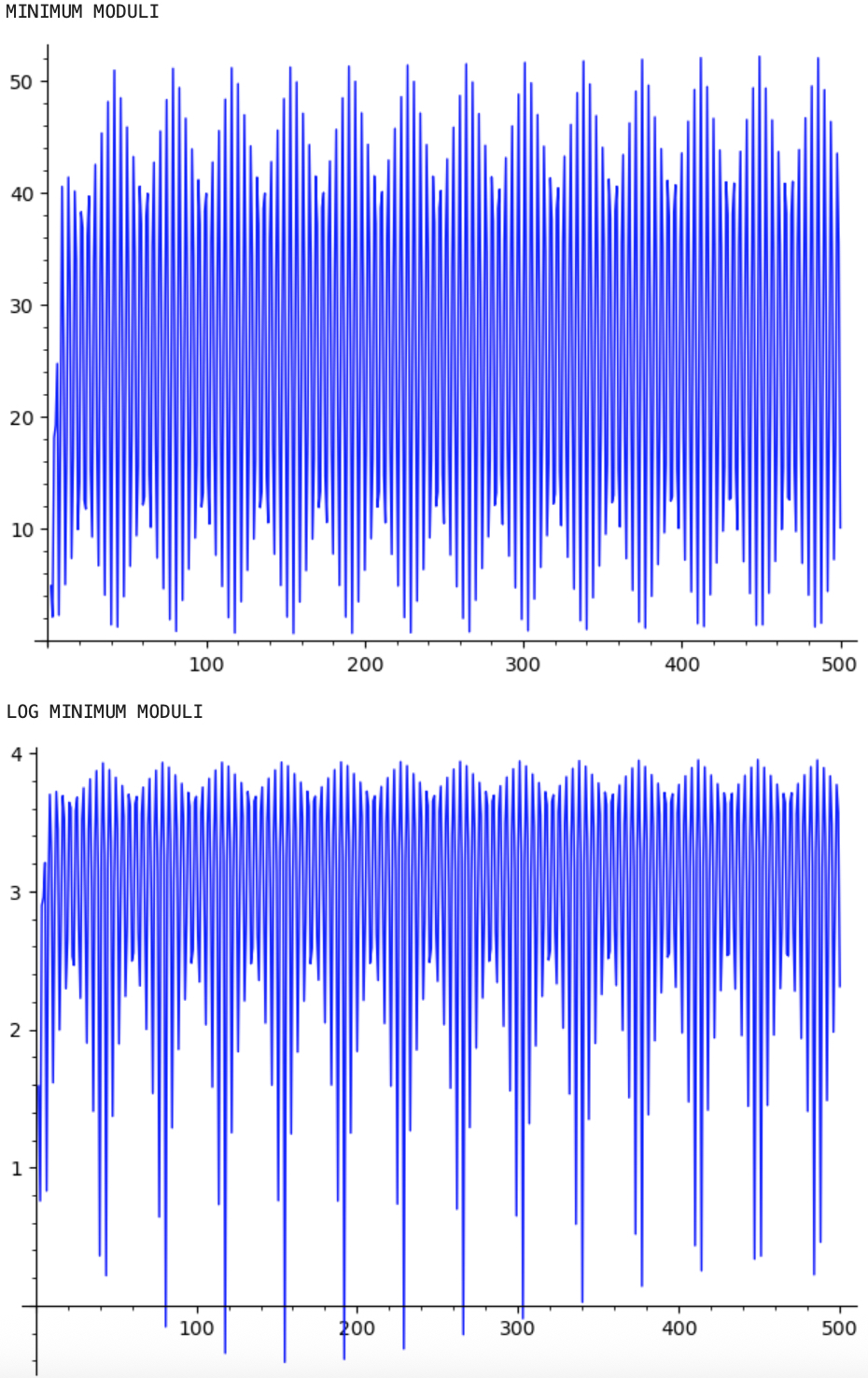}
    \caption{Minimum moduli over many values of $c$ for $h_n = \lambda_n$.}
    \label{fig:tauprime_many_c}
\end{figure}
\begin{figure}[H] % fig 6
    \centering
    %\hspace*{0em}
    \includegraphics[width=1\textwidth]{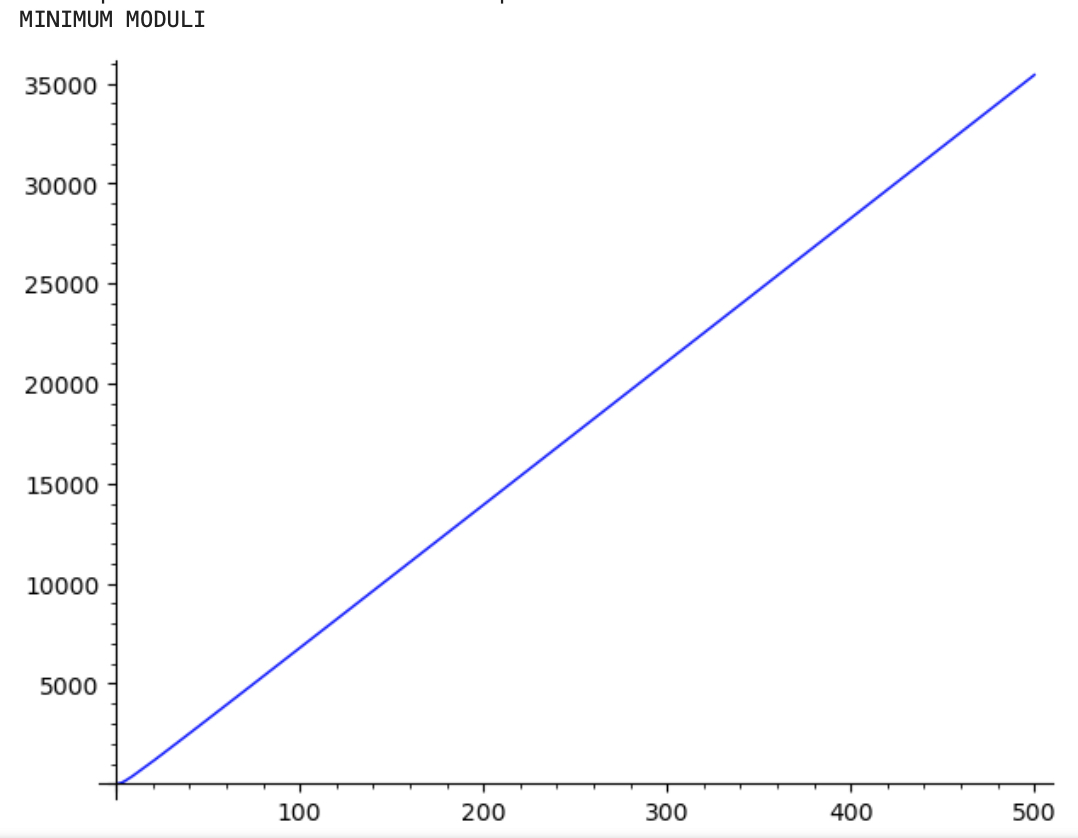}
    \caption{Minimum moduli for $h_n = \tau(p_n+ 1)$ with $c = 1$.} 
    \label{fig:tauprime_control}
\end{figure}
\begin{figure}[H]
    \centering
    %\hspace*{0em}
    \includegraphics[width=1\textwidth]{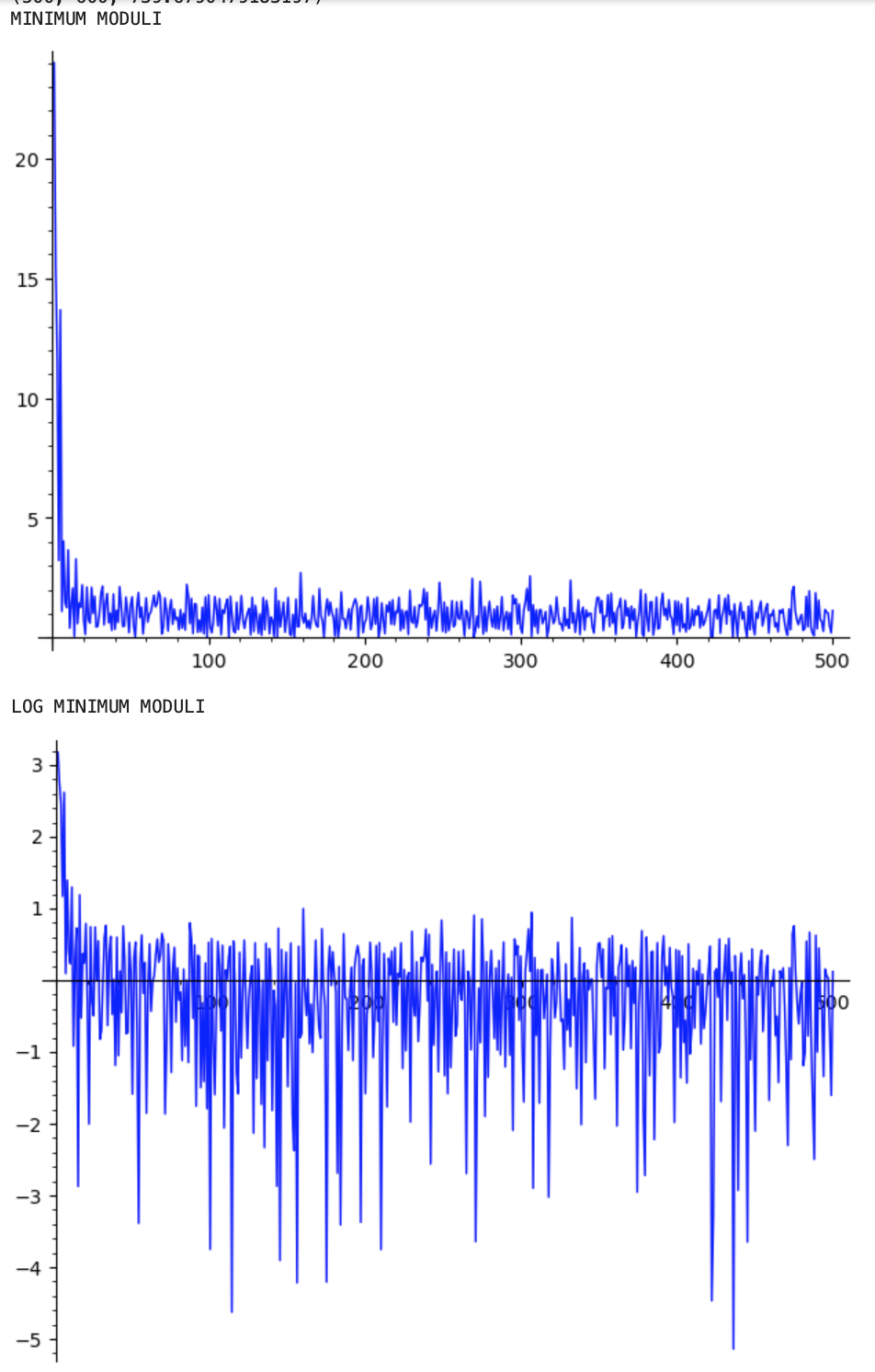}
    \caption{Minimum moduli for $h_n = \lambda_n$ before treatment.}
    \label{fig:undeformed_primeTau_min_moduli}
\end{figure}
\begin{figure}[H]
    \centering
    \begin{subfigure}[t]{0.48\textwidth}
        \centering
        \includegraphics[width=\textwidth]{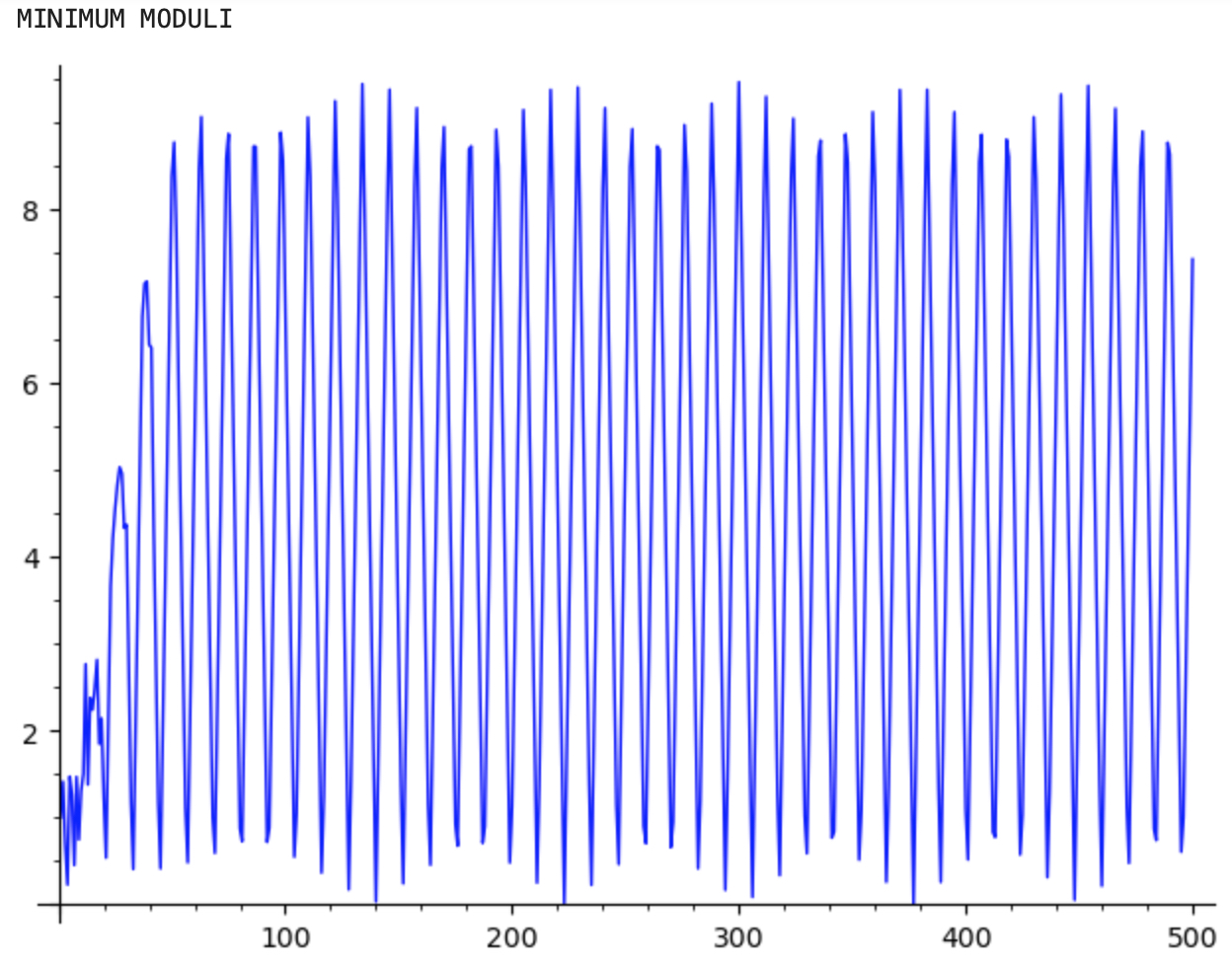}
        \caption{$c=1, h_n=a(n)$.}
        \label{fig:curve_11a1_all}
    \end{subfigure}
     \begin{subfigure}[t]{0.48\textwidth}
        \centering
        \includegraphics[width=\textwidth]{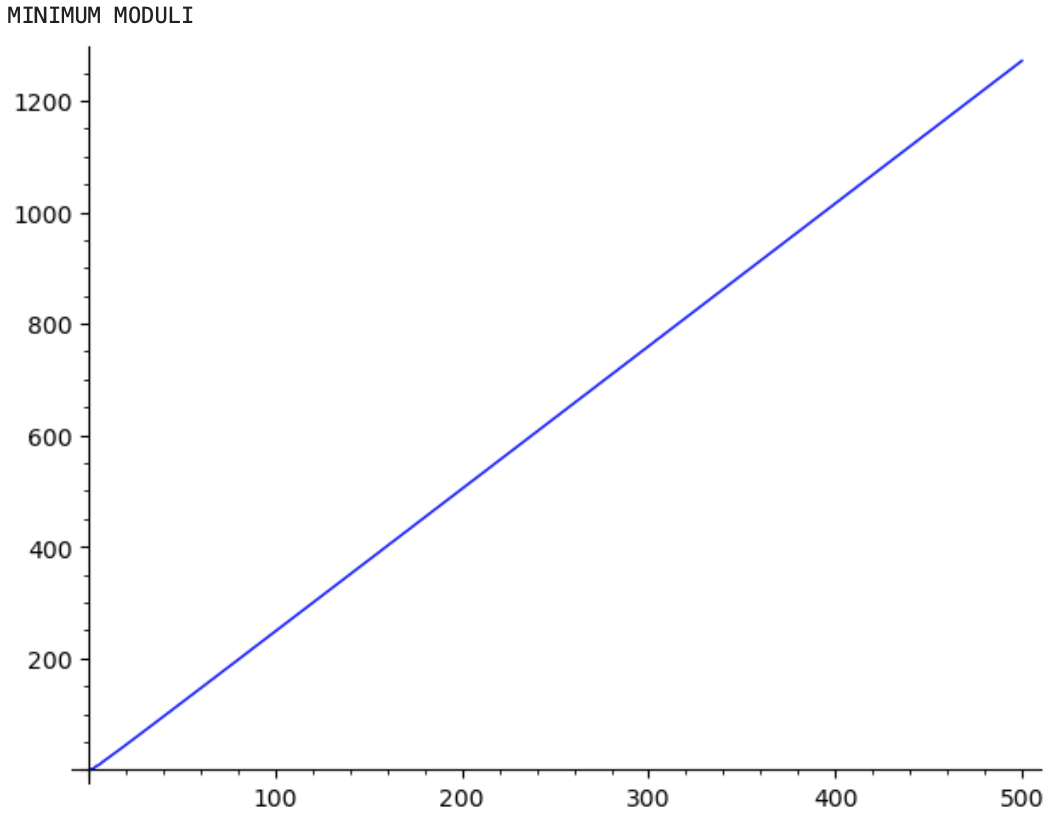}
        \caption{$c=1, h_n = a(p_n)$.}
        \label{fig:curve11a1_primes}
    \end{subfigure}
    \hfill
    \begin{subfigure}[t]{0.48\textwidth}
        \centering
        \includegraphics[width=\textwidth]{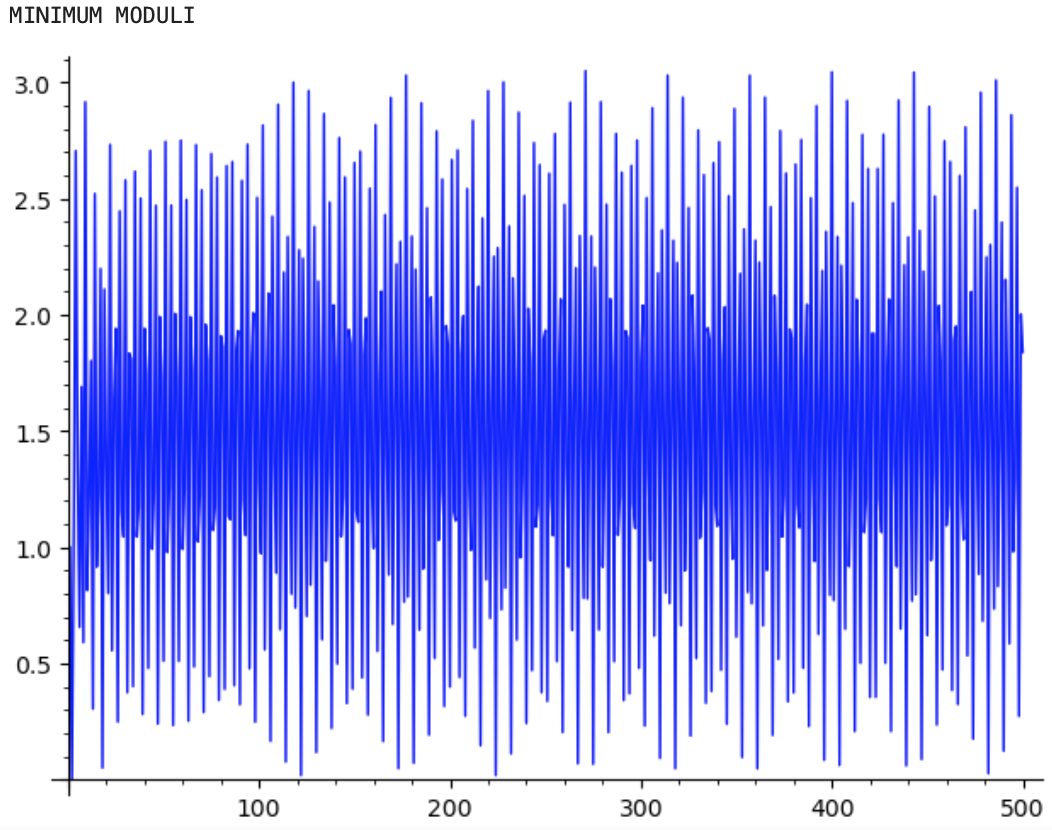}
        \caption{$c=1, h_n = a(p_n + 1)$.}
        \label{fig:curve11a1_primes_plus1}
    \end{subfigure}
    \hfill
    \caption{Cremona curve 11a1. }
    \label{fig:comparison_crv11a1}
\end{figure}
\begin{figure}[H]
    \centering
    \begin{subfigure}[t]{0.48\textwidth}
        \centering
        \includegraphics[width=\textwidth]{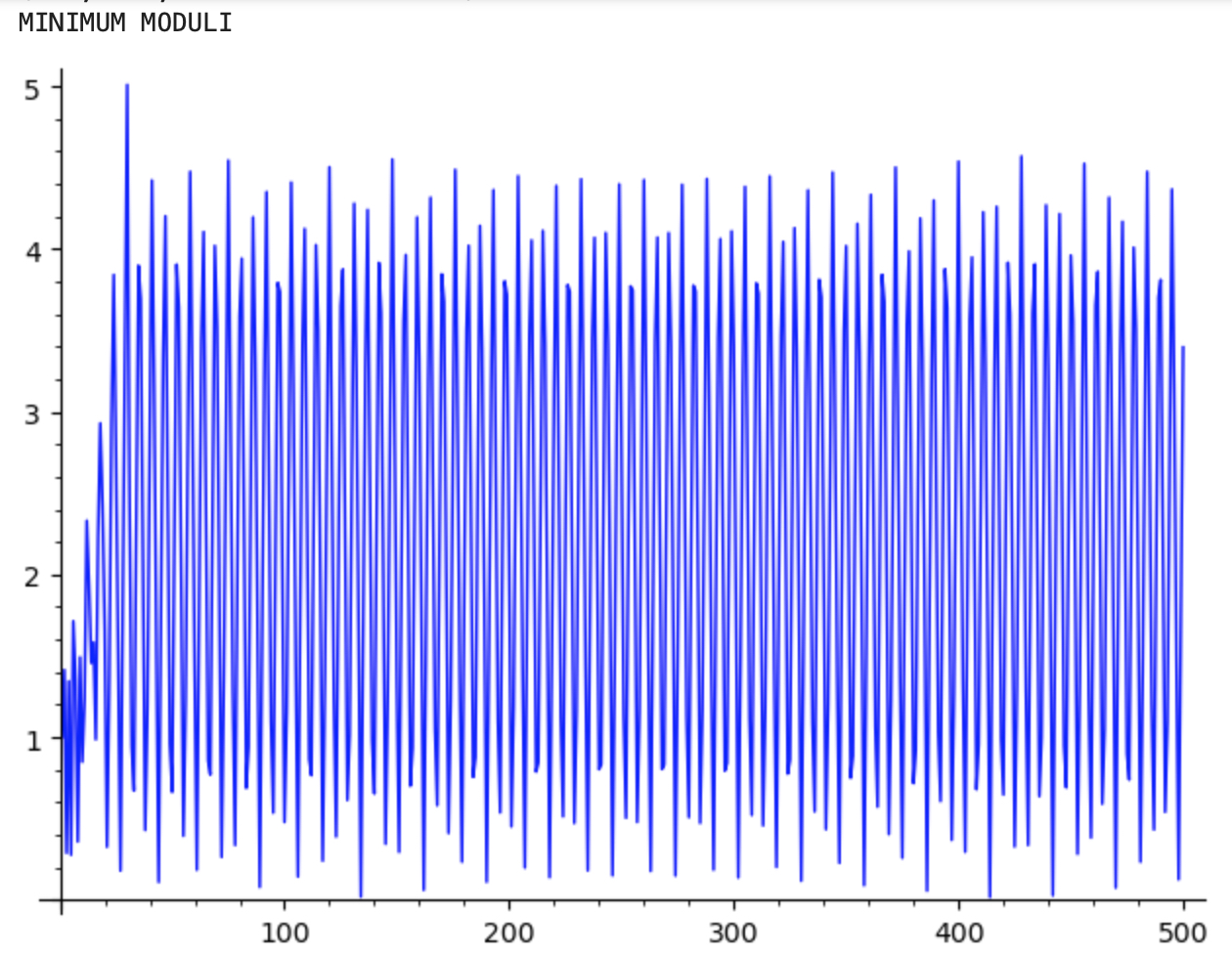}
        \caption{$h_n=a(n)$.}
        \label{fig:curve_14a1}
    \end{subfigure}
    \hfill
    \begin{subfigure}[t]{0.48\textwidth}
        \centering
        \includegraphics[width=\textwidth]{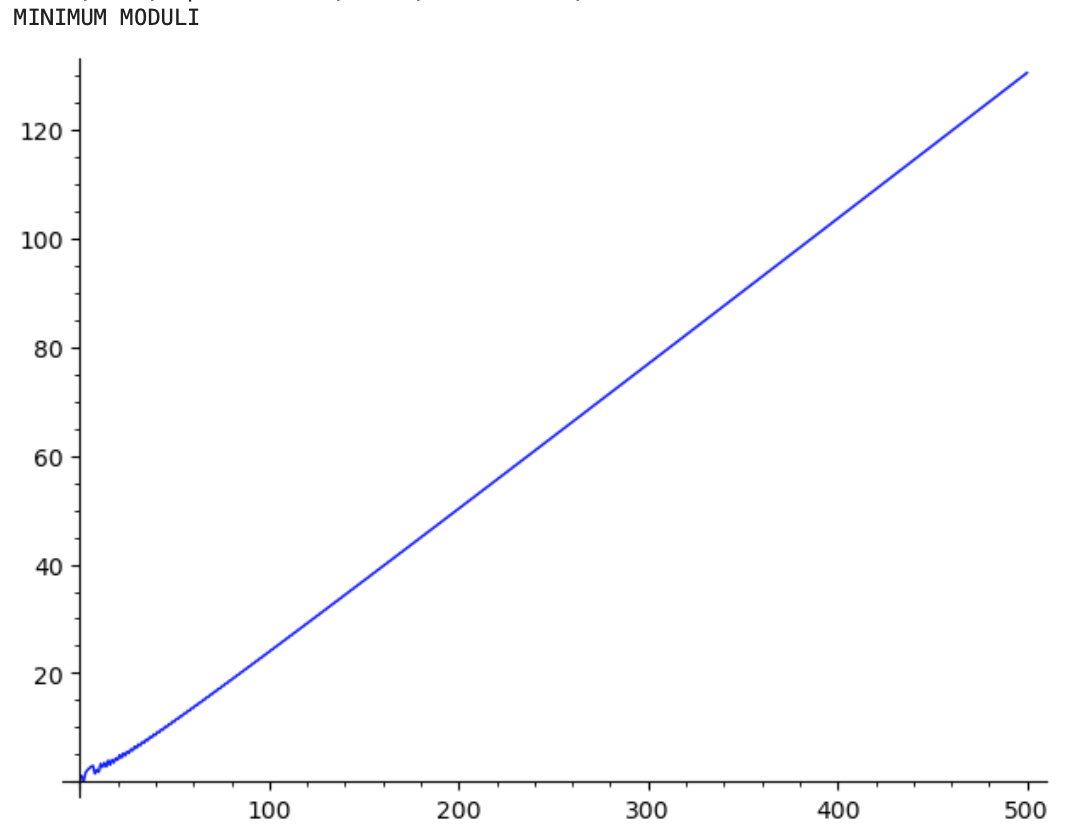}
        \caption{$h_n = a(p_n)$.}
        \label{fig:14a1_primes}
    \end{subfigure}
    \begin{subfigure}[t]{0.48\textwidth}
        \centering
        \includegraphics[width=\textwidth]{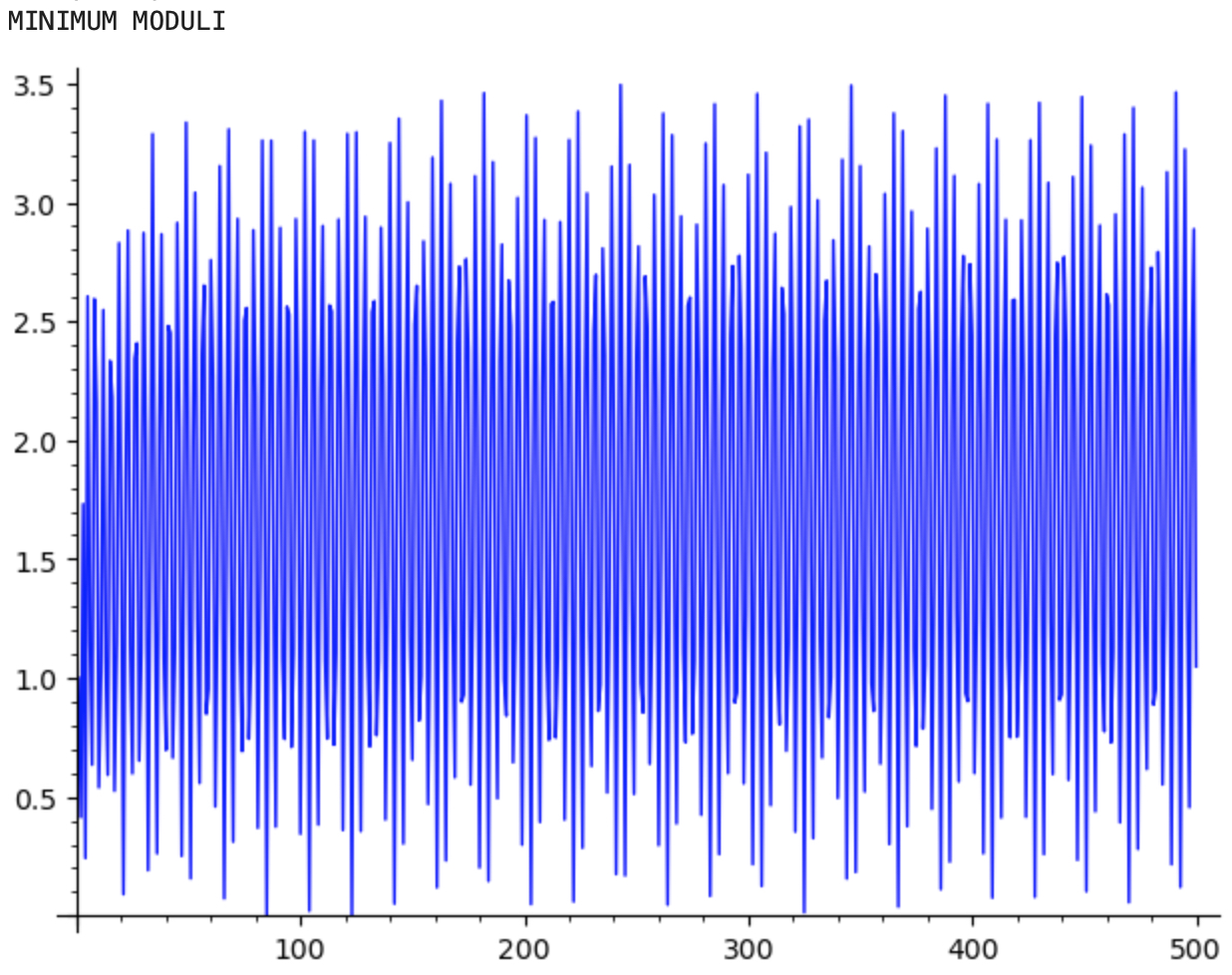}
        \caption{$h_n = a(p_n + 1)$.}
        \label{fig:curve14a1_primes_plus1}
    \end{subfigure}
    \caption{Cremona curve 14a1 with $c = 1$.}
    \label{fig:crv_14a1_comparison}
\end{figure}
\begin{figure}[H]
    \centering
    \begin{subfigure}[t]{0.48\textwidth}
        \centering
       \includegraphics[width=\textwidth]{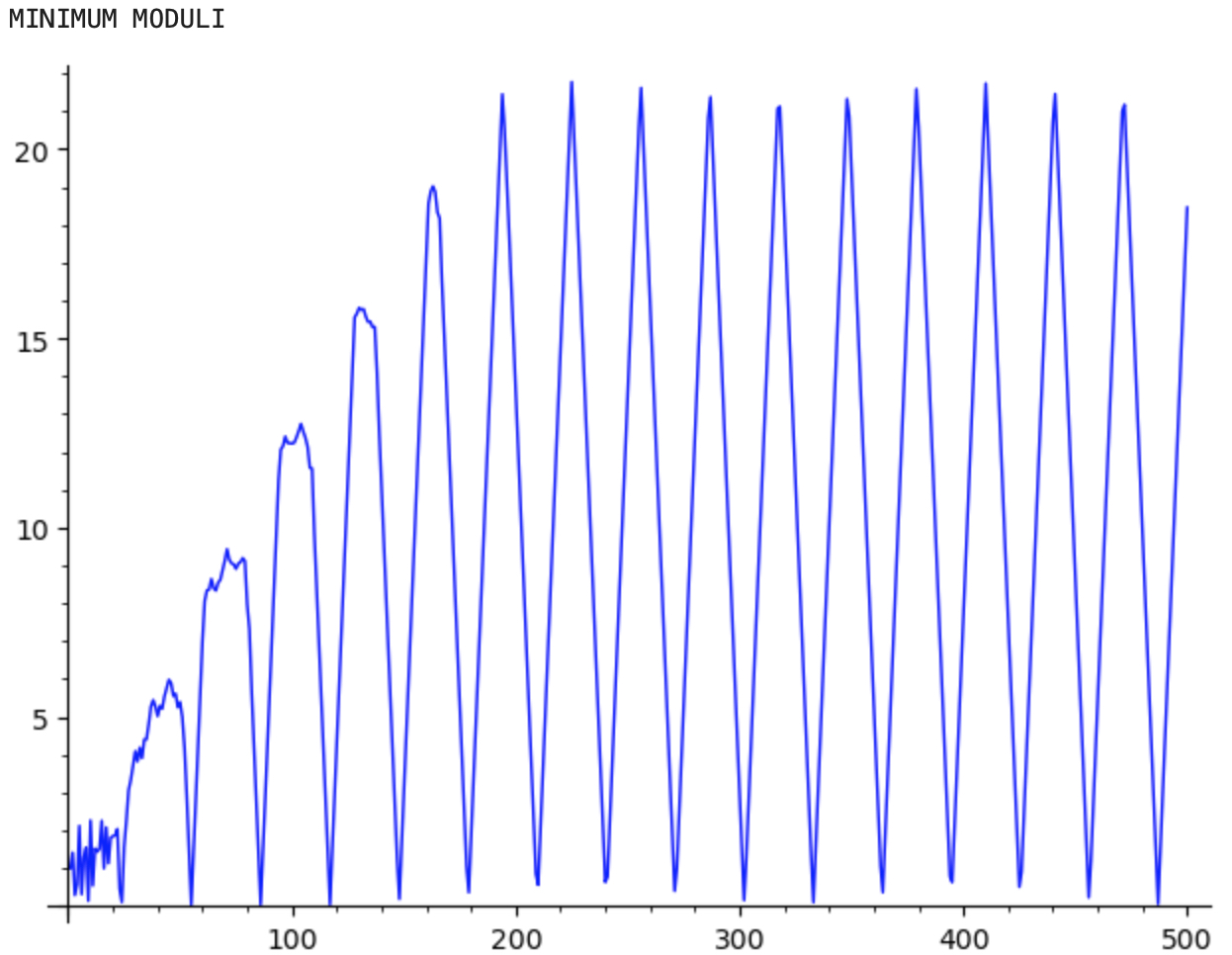}
        \caption{$h_n=a(n)$.}
        \label{fig:curve_15a_1_all}
    \end{subfigure}
    \hfill
    \begin{subfigure}[t]{0.48\textwidth}
        \centering
        \includegraphics[width=\textwidth]{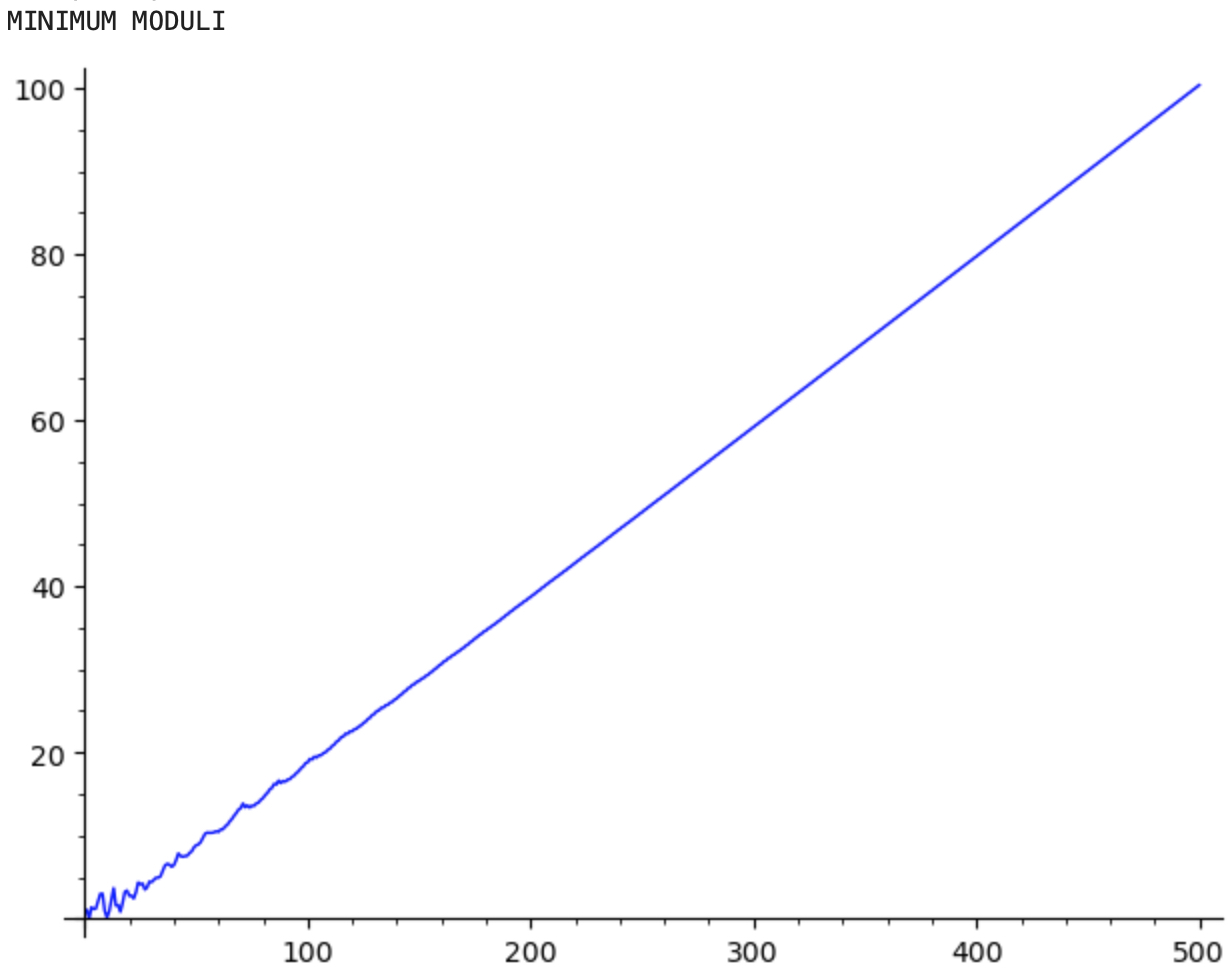}
        \caption{$h_n = a(p_n)$.}
        \label{fig:crv15a1_primes}
    \end{subfigure}
    \begin{subfigure}[t]{0.48\textwidth}
        \centering
        \includegraphics[width=\textwidth]{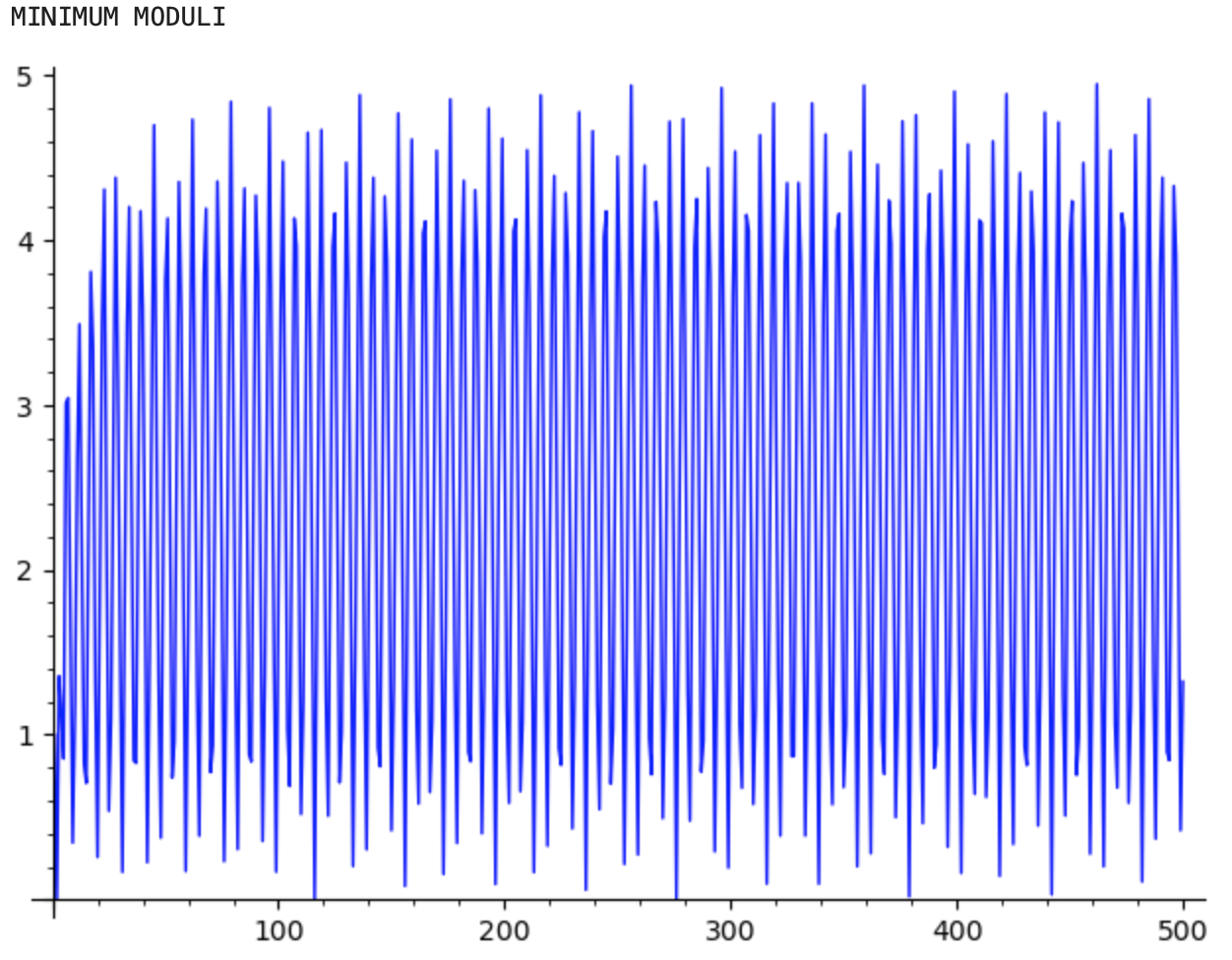}
        \caption{$h_n = a(p_n + 1)$.}
        \label{fig:crv_15a1_primes_plus_1}
    \end{subfigure}
    \caption{Cremona curve 15a1 with $c = 1$.}
    \label{fig:crv_15a1_comparison}
\end{figure}
\begin{figure}[H]
    \centering
    \begin{subfigure}[t]{0.48\textwidth}
        \centering
        \includegraphics[width=\textwidth]{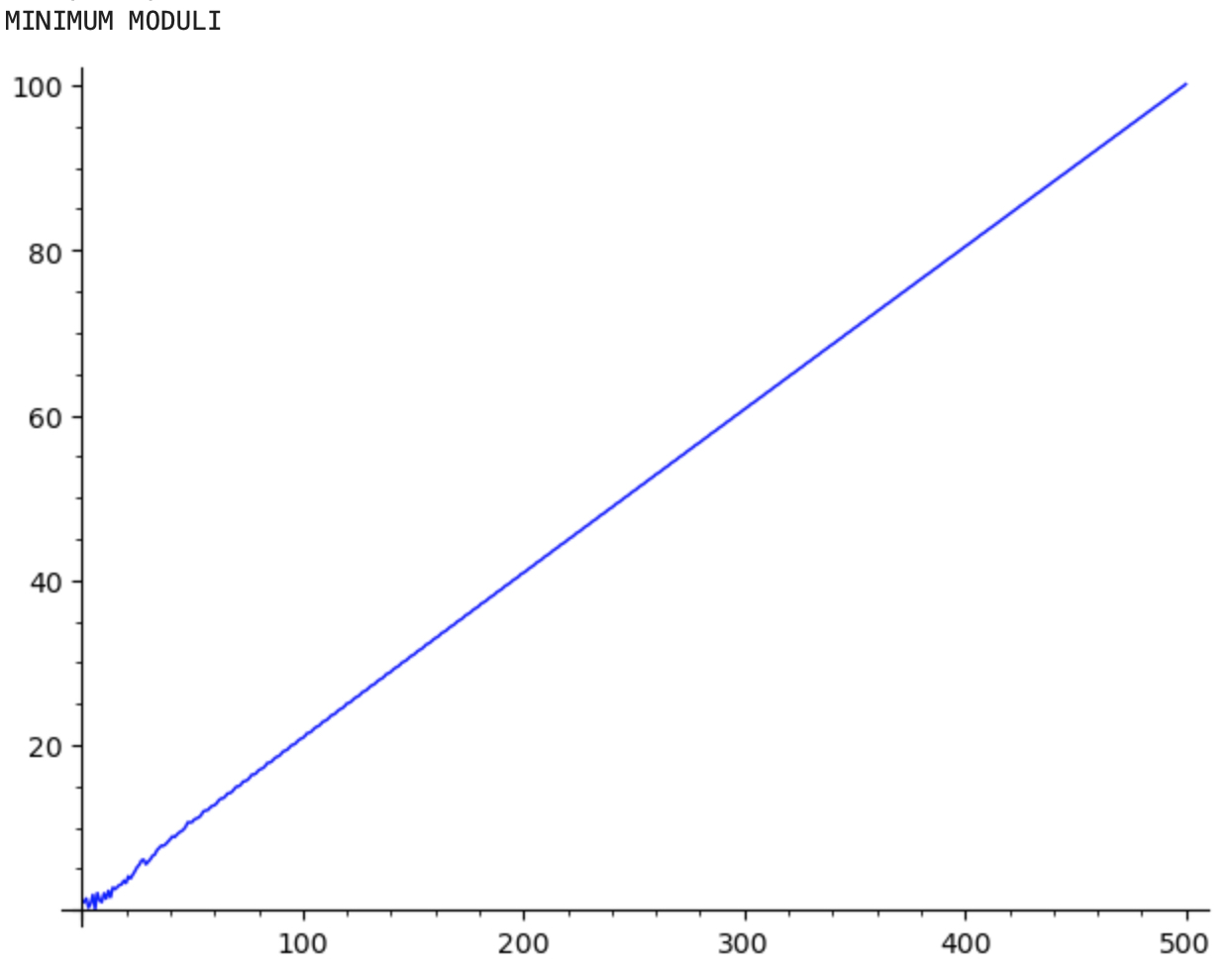}
        \caption{$h_n=a(n)$.}
        \label{fig:crv17a1_all_n}
    \end{subfigure}
    \hfill
    \begin{subfigure}[t]{0.48\textwidth}
        \centering
        \includegraphics[width=\textwidth]{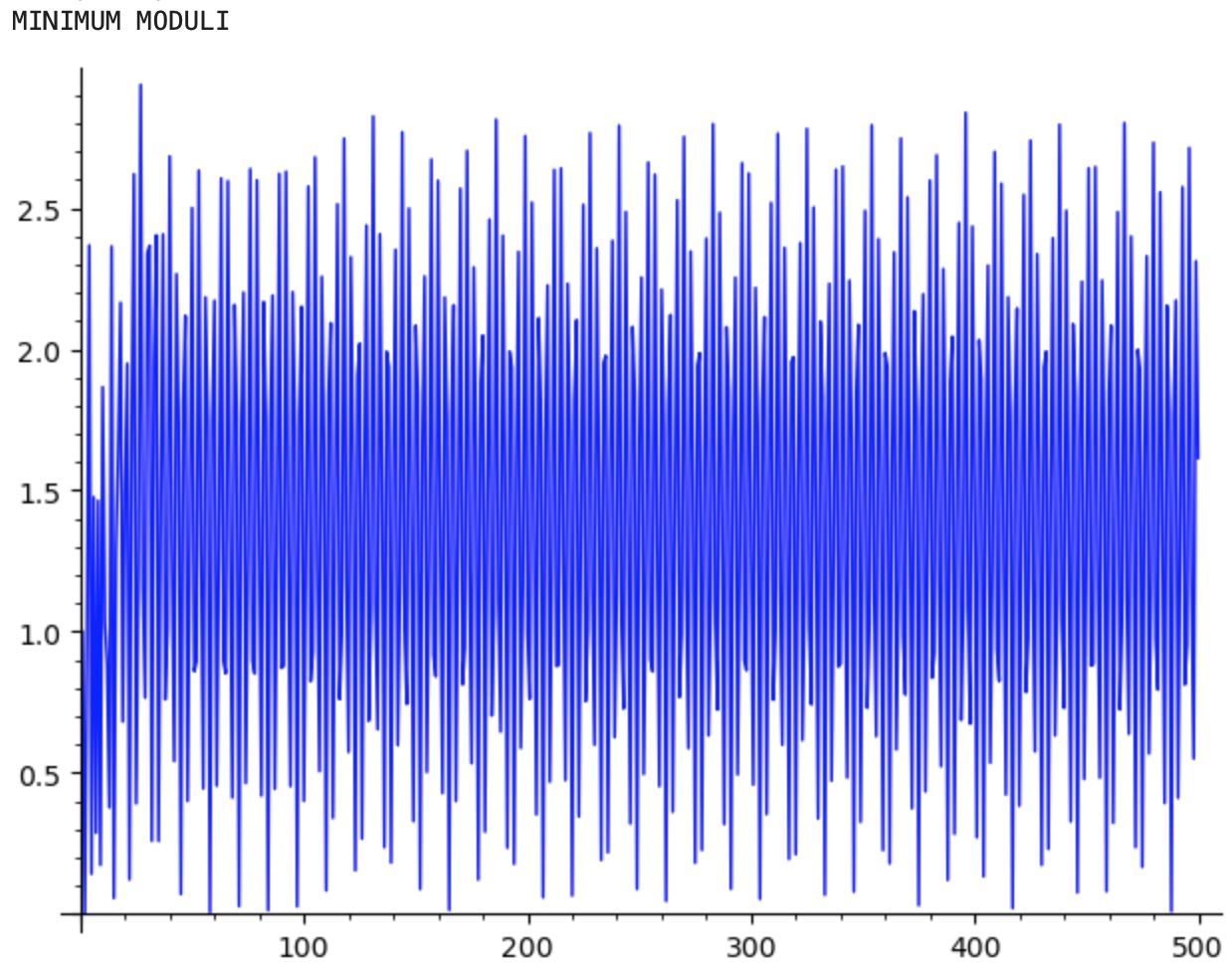}
        \caption{$h_n = a(p_n)$.}
        \label{fig:crv17a1_primes}
    \end{subfigure}
    \begin{subfigure}[t]{0.48\textwidth}
        \centering
        \includegraphics[width=\textwidth]{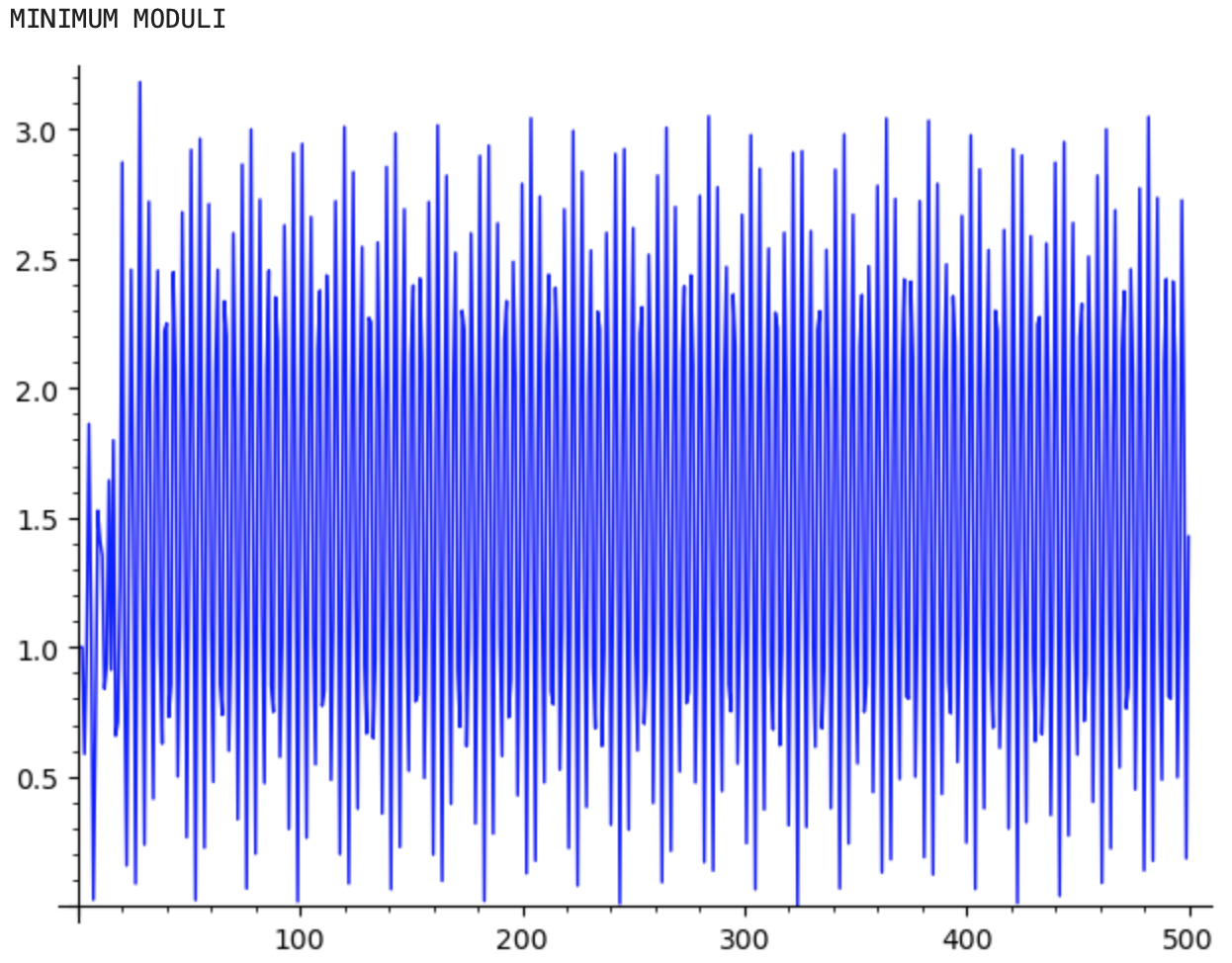}
        \caption{$h_n= a(p_n + 1)$.}
        \label{fig:curve_17a1_on_primes_plus_one}
    \end{subfigure}
    \caption{Cremona curve 17a1 with $c = 1$.}
    \label{fig:crv_17a1_comparison}
\end{figure}
\begin{figure}[H]
    \centering
    \begin{subfigure}[t]{0.48\textwidth}
        \centering
        \includegraphics[width=\textwidth]{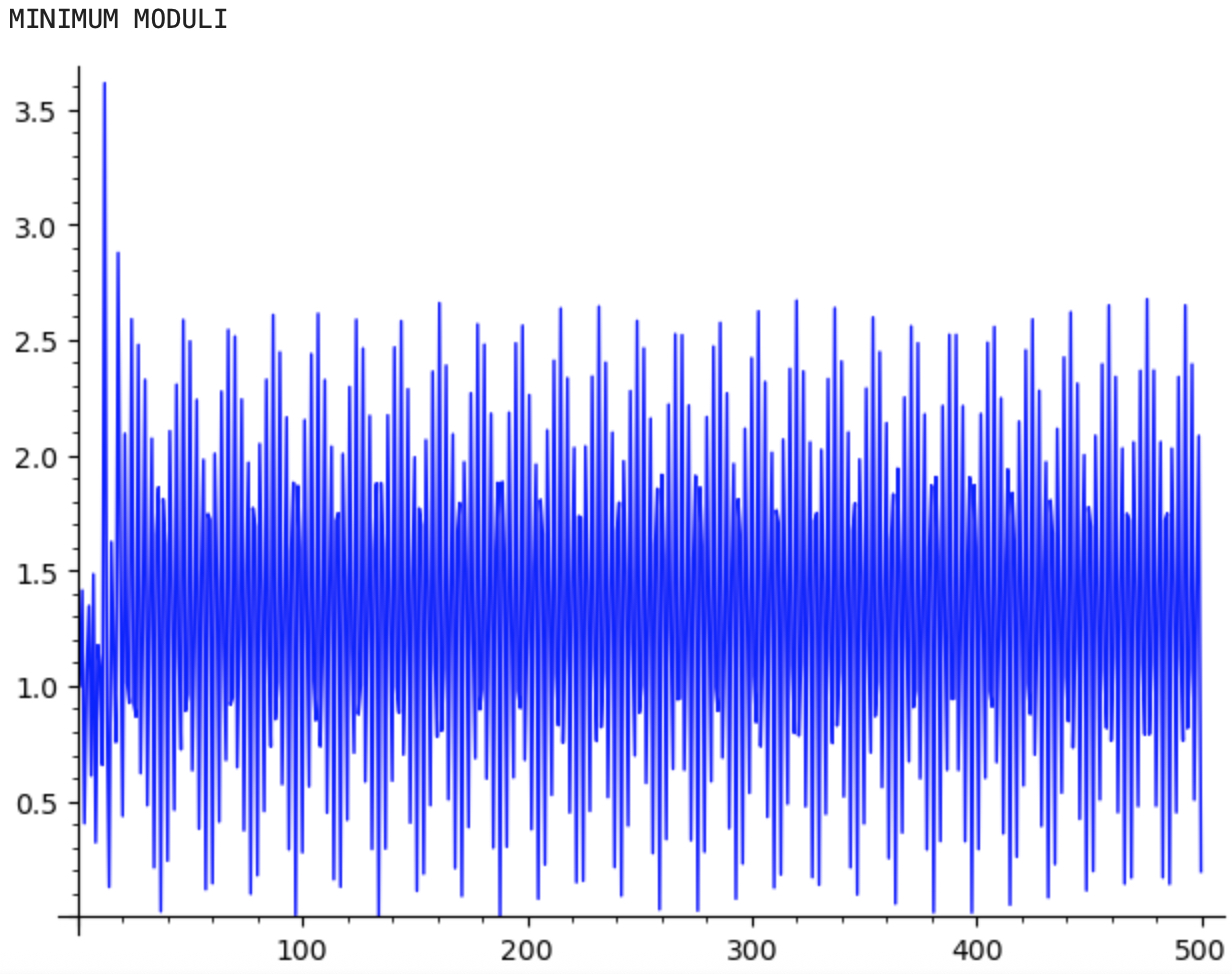}
        \caption{$h_n=a(n)$.}
        \label{fig:curve_19a1}
    \end{subfigure}
    \hfill
    \begin{subfigure}[t]{0.48\textwidth}
        \centering
        \includegraphics[width=\textwidth]{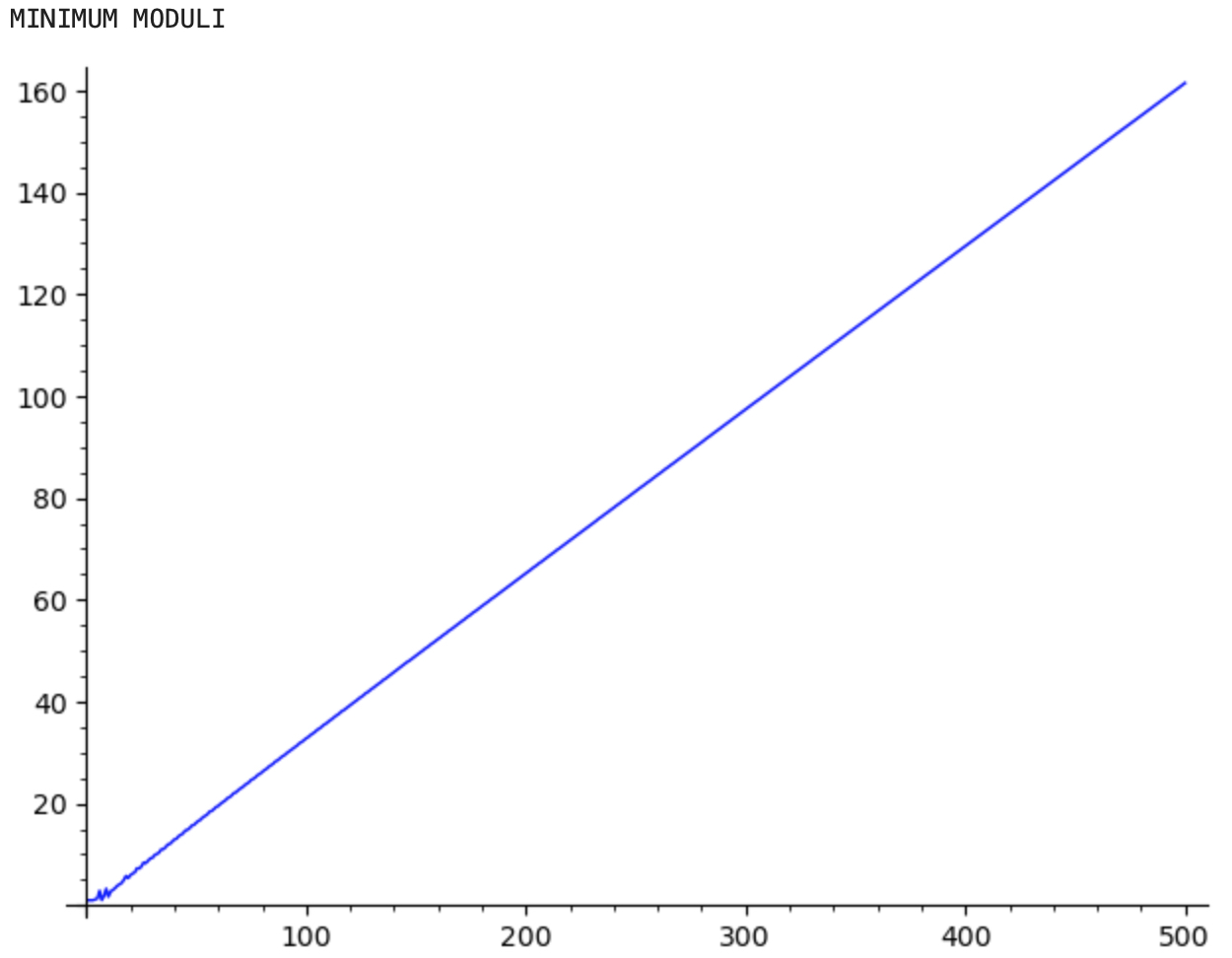}
        \caption{$h_n = a(p_n)$.}
        \label{fig:crv19a1_primes}
    \end{subfigure}
    \begin{subfigure}[t]{0.48\textwidth}
        \centering
        \includegraphics[width=\textwidth]{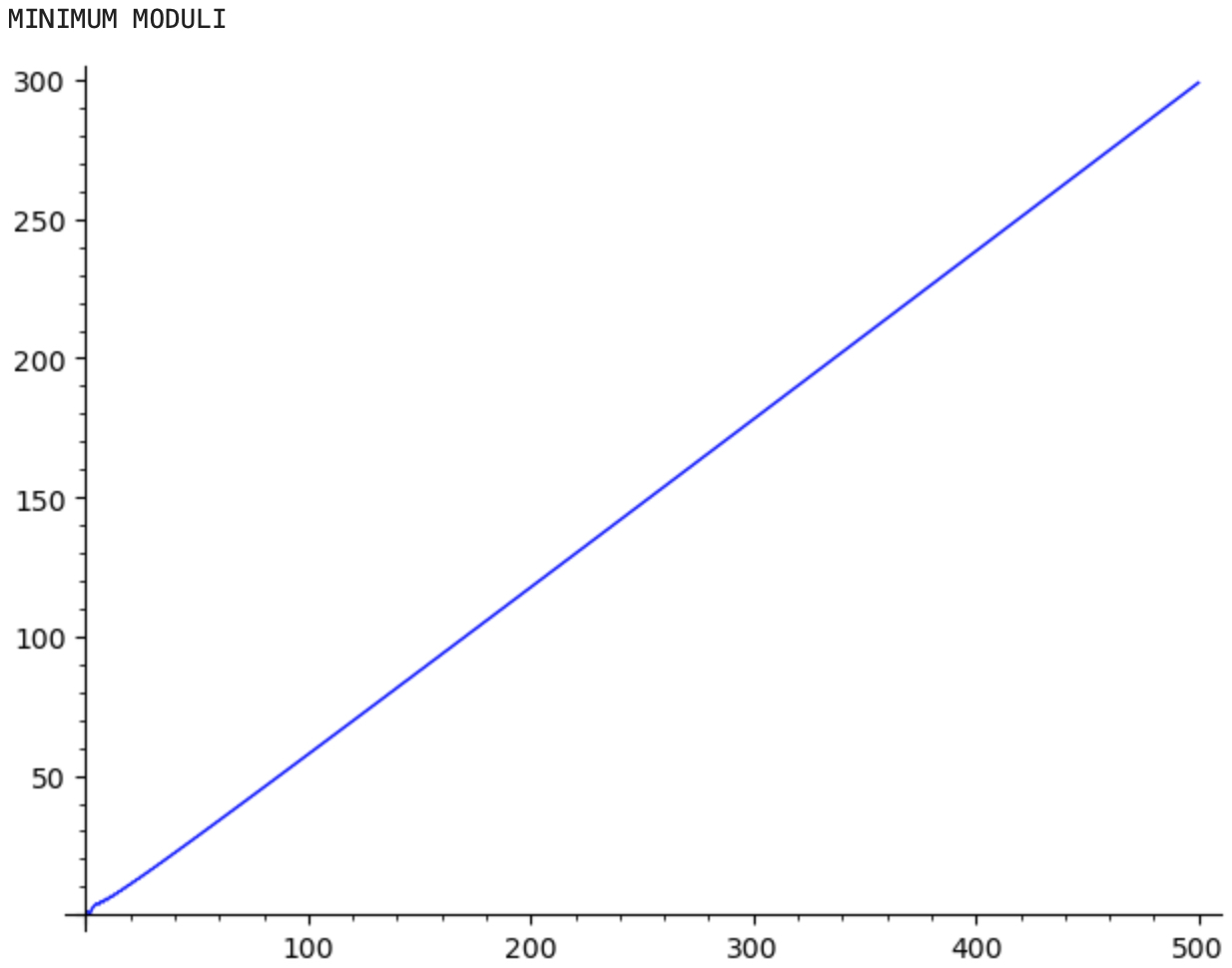}
        \caption{$h_n = a(p_n +1)$.}
        \label{fig:crv19a1_primes_plus1}
    \end{subfigure}
    \caption{Cremona curve 19a1 with $c = 1$.}
    \label{fig:crv_19a1_comparison}
\end{figure}
\begin{figure}[H]
    \centering
    \begin{subfigure}[t]{0.48\textwidth}
        \centering
        \includegraphics[width=\textwidth]{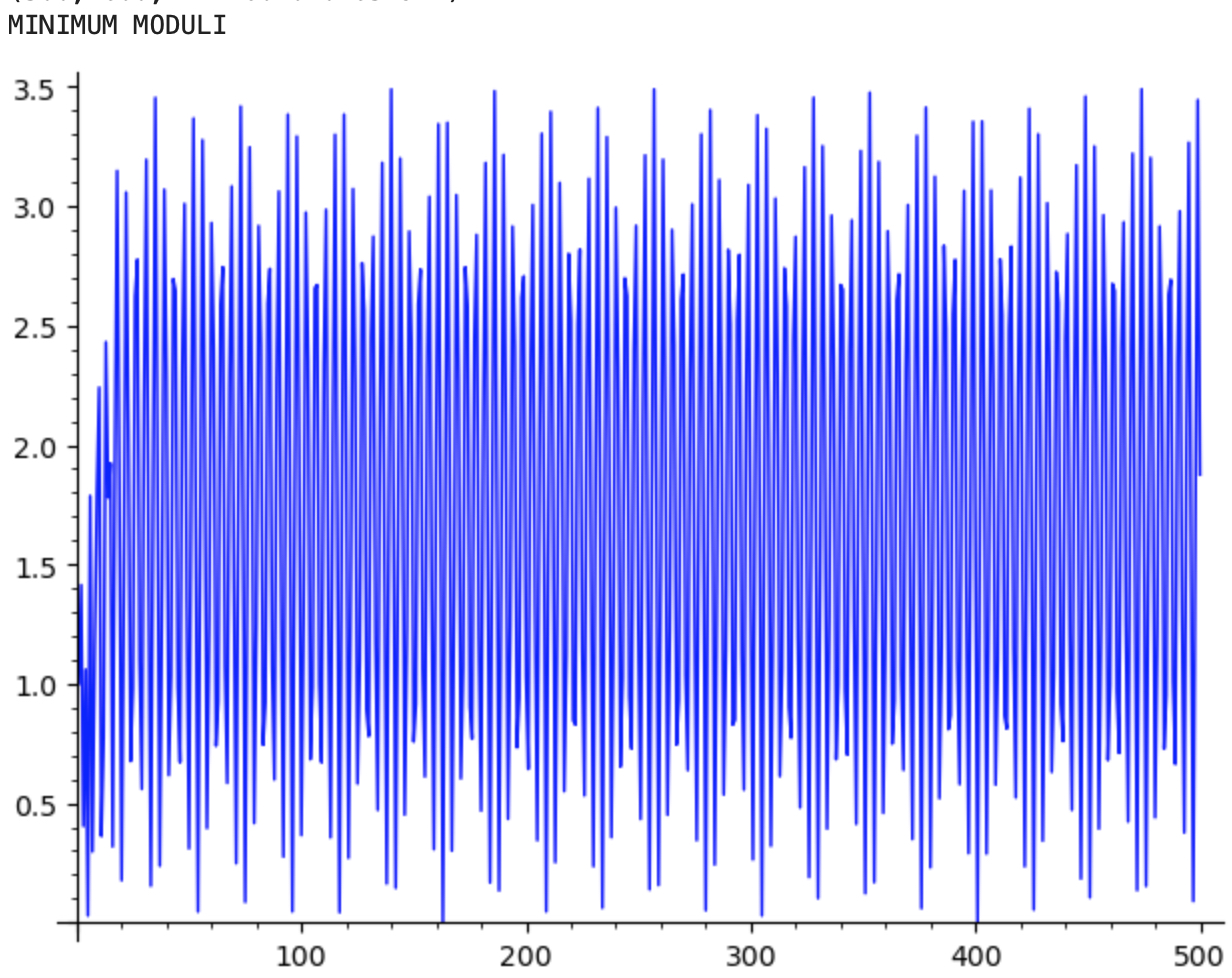}
        \caption{$h_n=a(n)$.}
        \label{fig:crv_20a1_all}
    \end{subfigure}
    \hfill
    \begin{subfigure}[t]{0.48\textwidth}
        \centering
        \includegraphics[width=\textwidth]{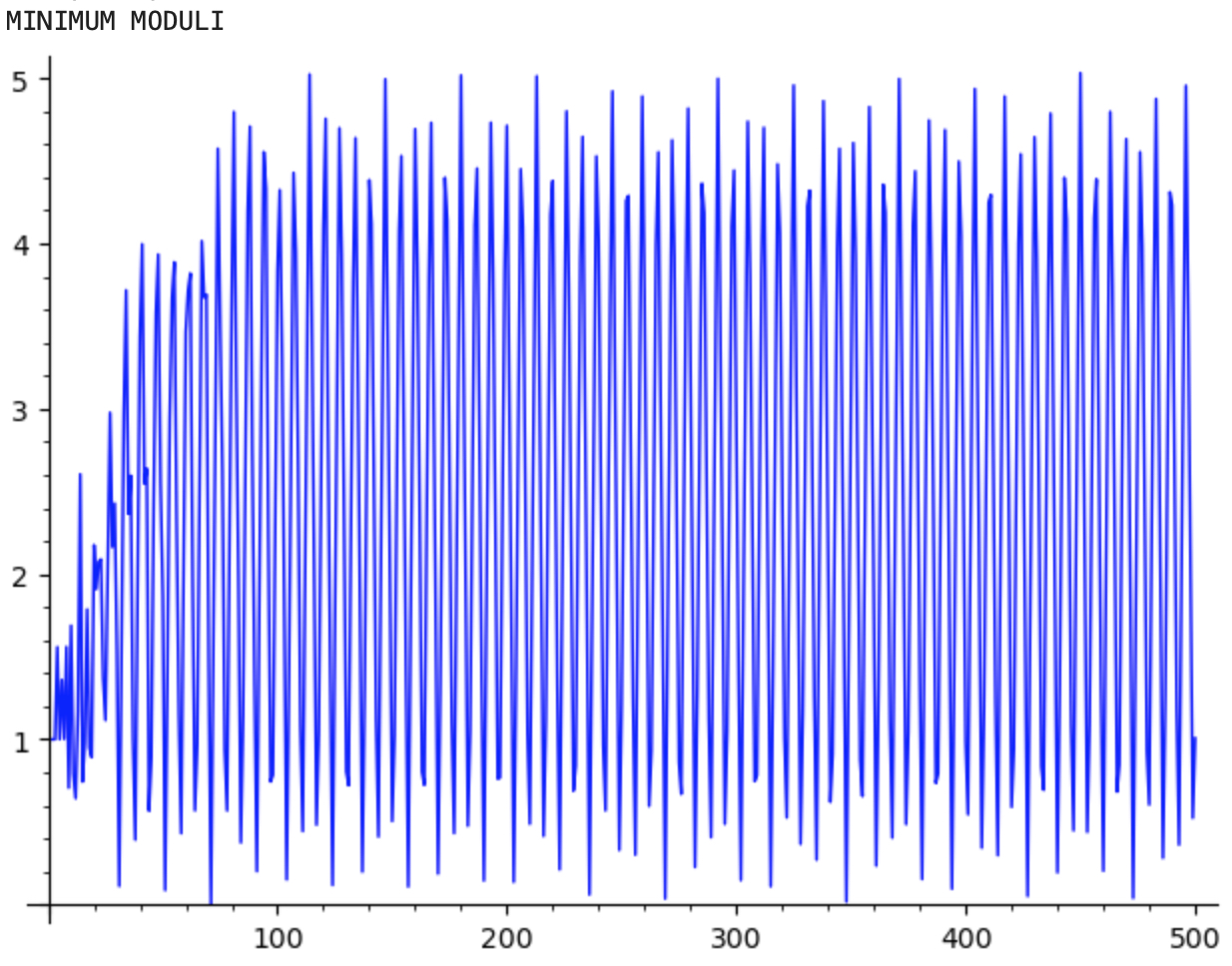}
        \caption{$h_n = a(p_n)$.}
        \label{fig:crv20a1_primes}
    \end{subfigure}
    \begin{subfigure}[t]{0.48\textwidth}
        \centering
        \includegraphics[width=\textwidth]{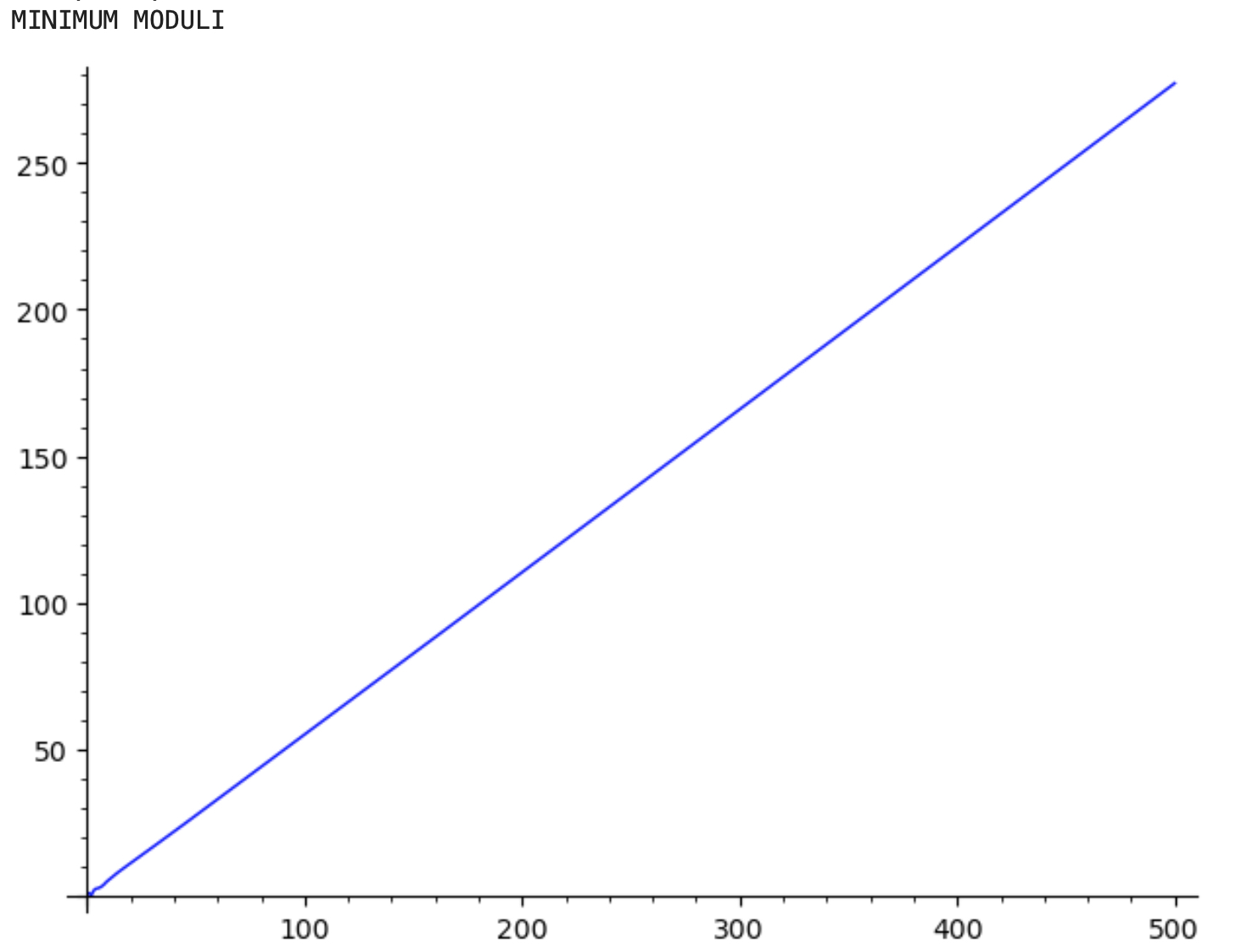}
        \caption{$h_n = a(p_n + 1)$.}
        \label{fig:crv_20a1_primes_plus1}
    \end{subfigure}
    \caption{Cremona curve 20a1 with $c = 1$.}
    \label{fig:crv_20a1_comparison}
\end{figure}
\begin{figure}[H]
    \centering
    \begin{subfigure}[t]{0.48\textwidth}
        \centering
        \includegraphics[width=\textwidth]{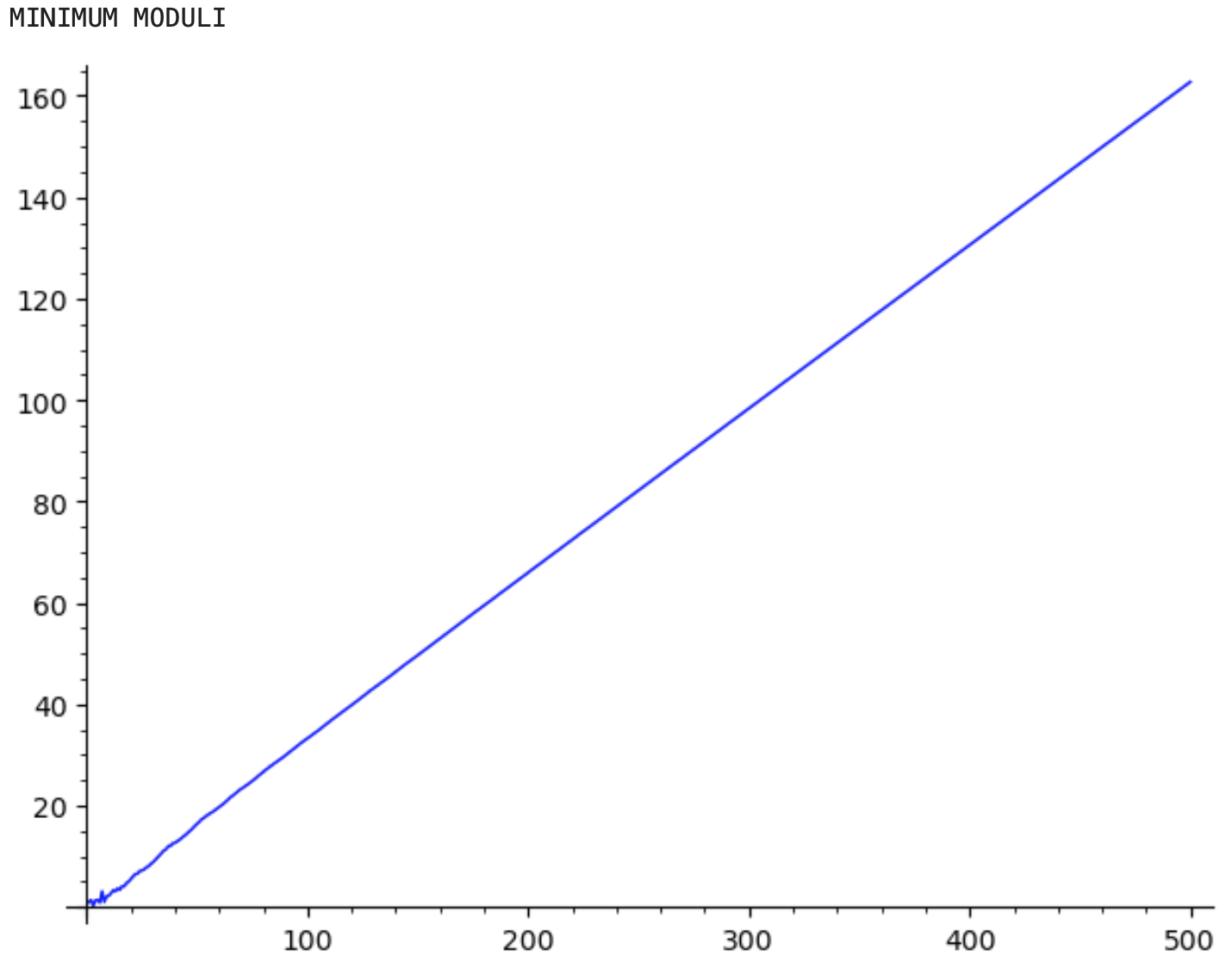}
        \caption{$h_n=a(n)$.}
        \label{fig:crv21a1_all}
    \end{subfigure}
    \hfill
    \begin{subfigure}[t]{0.48\textwidth}
        \centering
        \includegraphics[width=\textwidth]{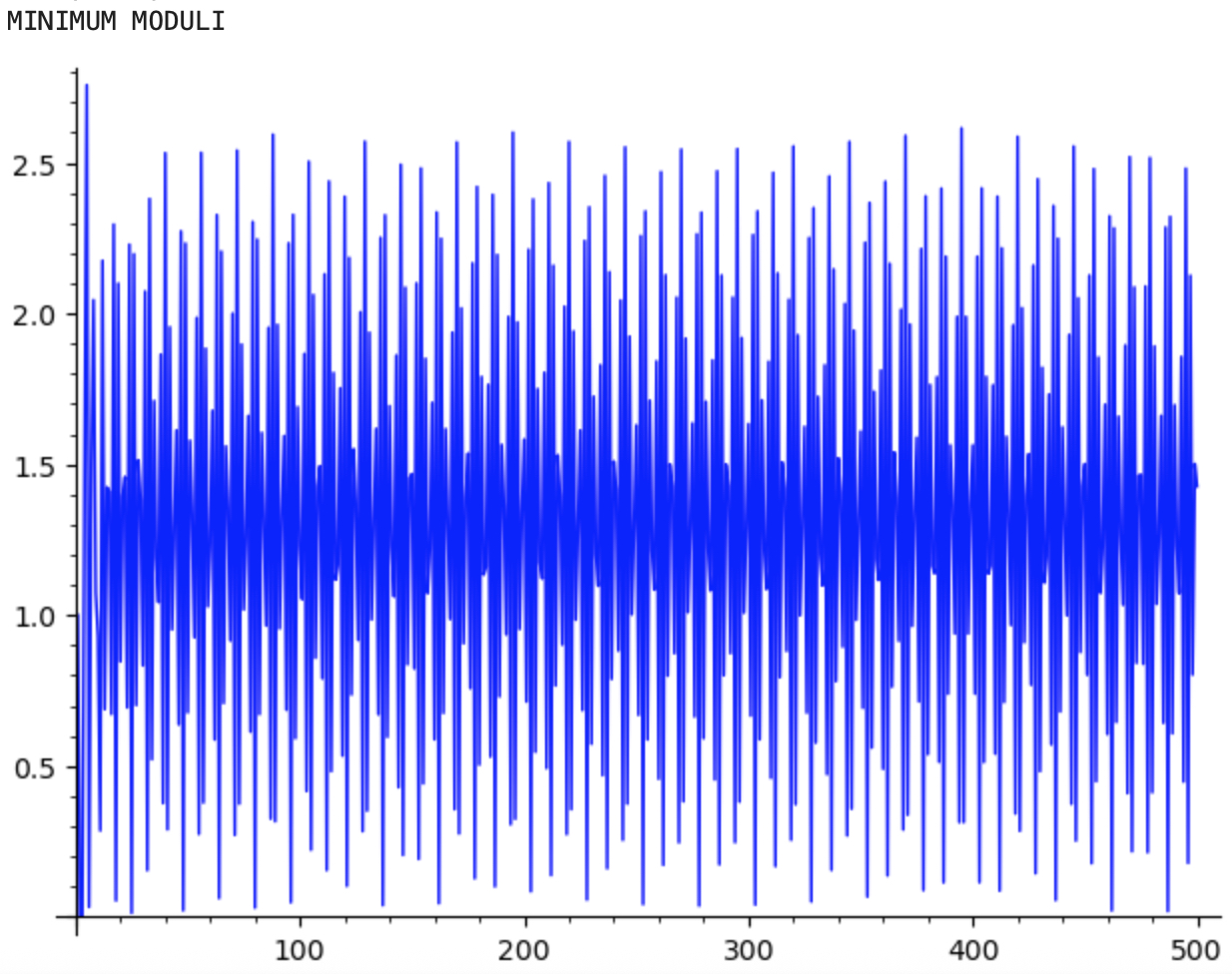}
        \caption{$h_n = a(p_n)$.}
        \label{fig:crv21a1_primes}
    \end{subfigure}
    \begin{subfigure}[t]{0.48\textwidth}
        \centering
        \includegraphics[width=\textwidth]{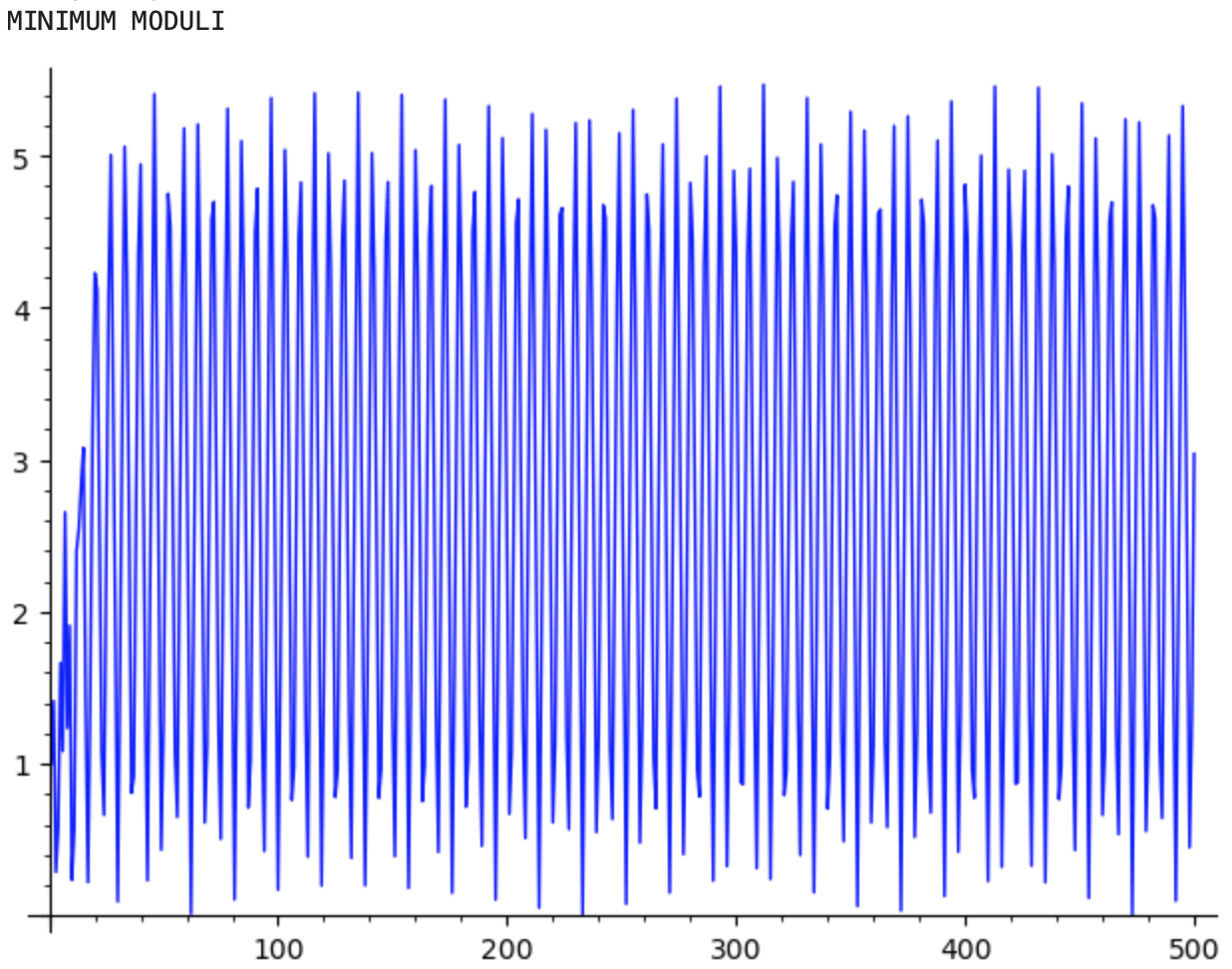}
        \caption{$h_n = a(p_n + 1)$.}
        \label{fig:crv21a1_primes_plus1}
    \end{subfigure}
    \caption{Cremona curve 21a1 with $c = 1$.}
    \label{fig:crv_21a1_primes}
\end{figure}
\begin{figure}[H]
    \centering
    \begin{subfigure}[t]{0.48\textwidth}
        \centering
        \includegraphics[width=\textwidth]{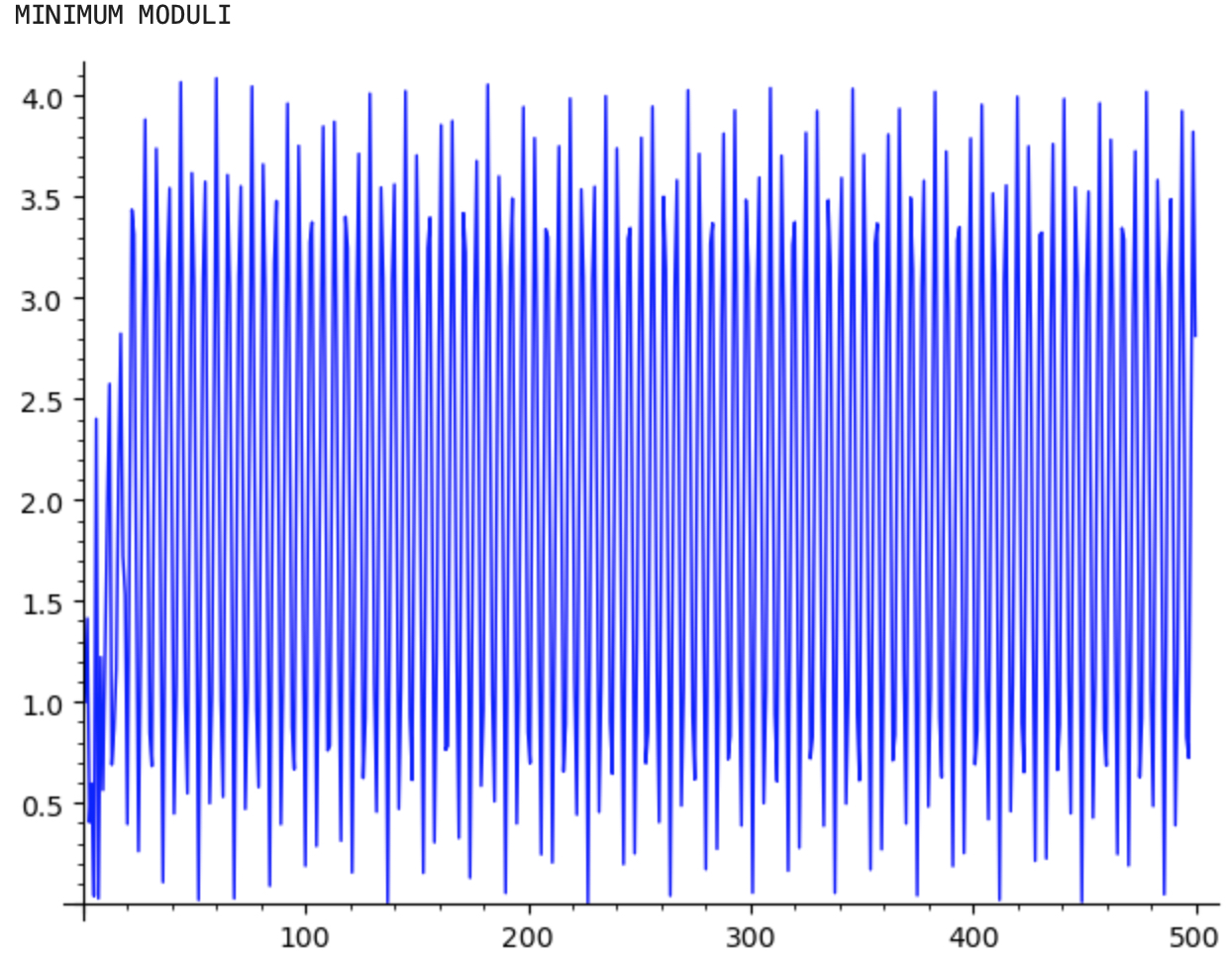}
        \caption{$h_n=a(n)$.}
        \label{fig:crv24a1_all_n}
    \end{subfigure}
    \hfill
    \begin{subfigure}[t]{0.48\textwidth}
        \centering
        \includegraphics[width=\textwidth]{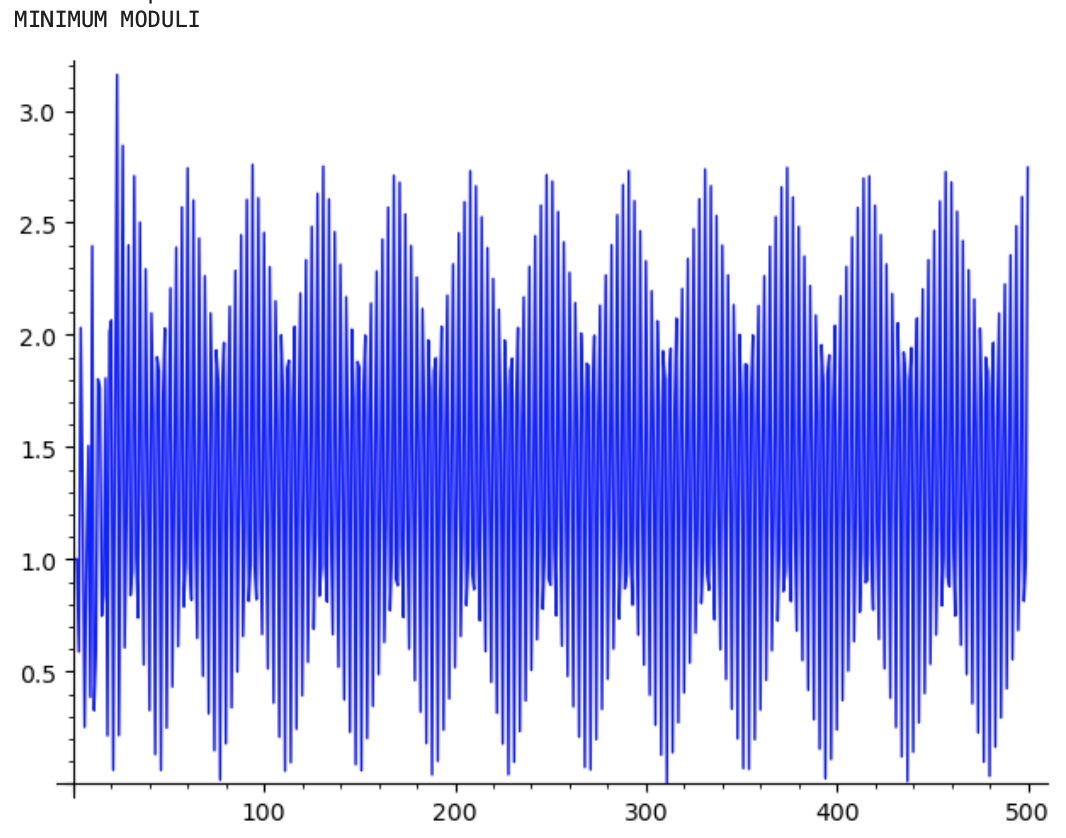}
        \caption{$h_n = a(p_n)$.}
        \label{fig:crv24a1_primes}
    \end{subfigure}
     \begin{subfigure}[t]{0.48\textwidth}
        \centering
        \includegraphics[width=\textwidth]{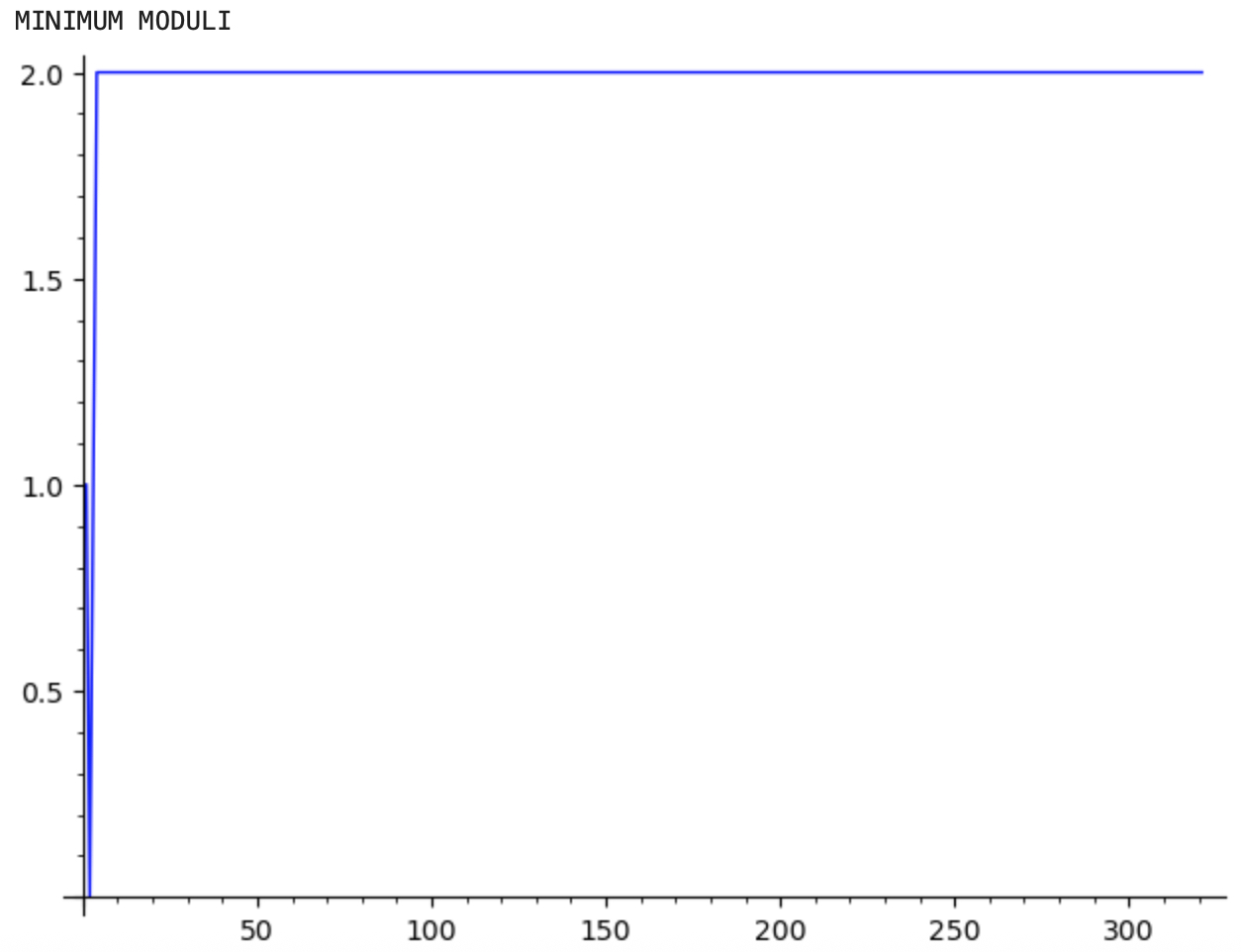}
        \caption{$h_n = a(p_n + 1)$.}
        \label{fig:crv24a1_primes_plus1}
    \end{subfigure}
    \caption{Cremona curve 24a1 with $c = 1$.}
    \label{fig:crv_24a1_comparison}
\end{figure}
\begin{figure}[H]
    \centering
    \begin{subfigure}[t]{0.48\textwidth}
        \centering
        \includegraphics[width=\textwidth]{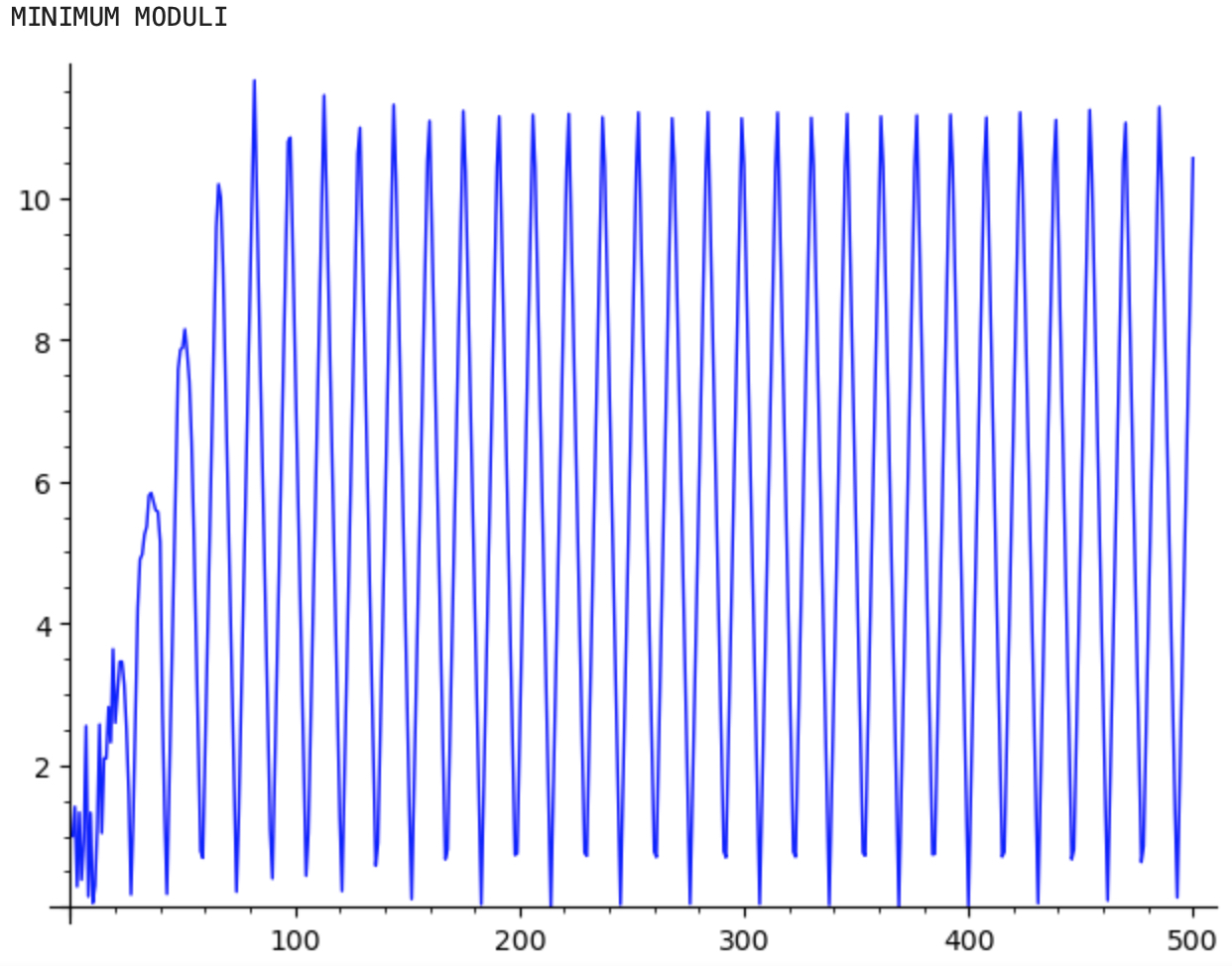}
        \caption{$h_n=a(n)$.}
        \label{fig:crv26a1_all}
    \end{subfigure}
    \hfill
    \begin{subfigure}[t]{0.48\textwidth}
        \centering
        \includegraphics[width=\textwidth]{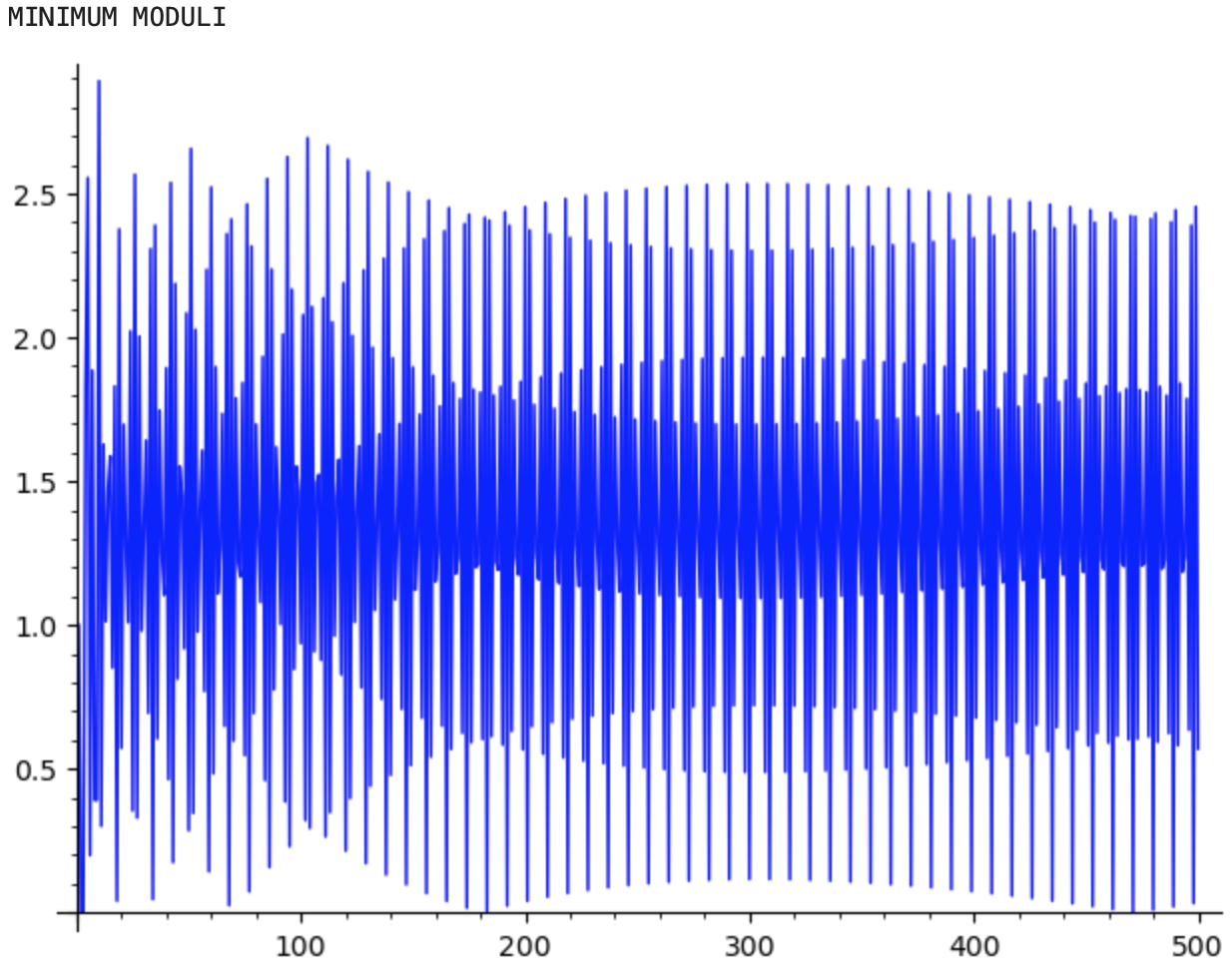}
        \caption{$h_n = a(p_n)$.}
        \label{fig:crv26a1_primes}
    \end{subfigure}
    \begin{subfigure}[t]{0.48\textwidth}
        \centering
        \includegraphics[width=\textwidth]{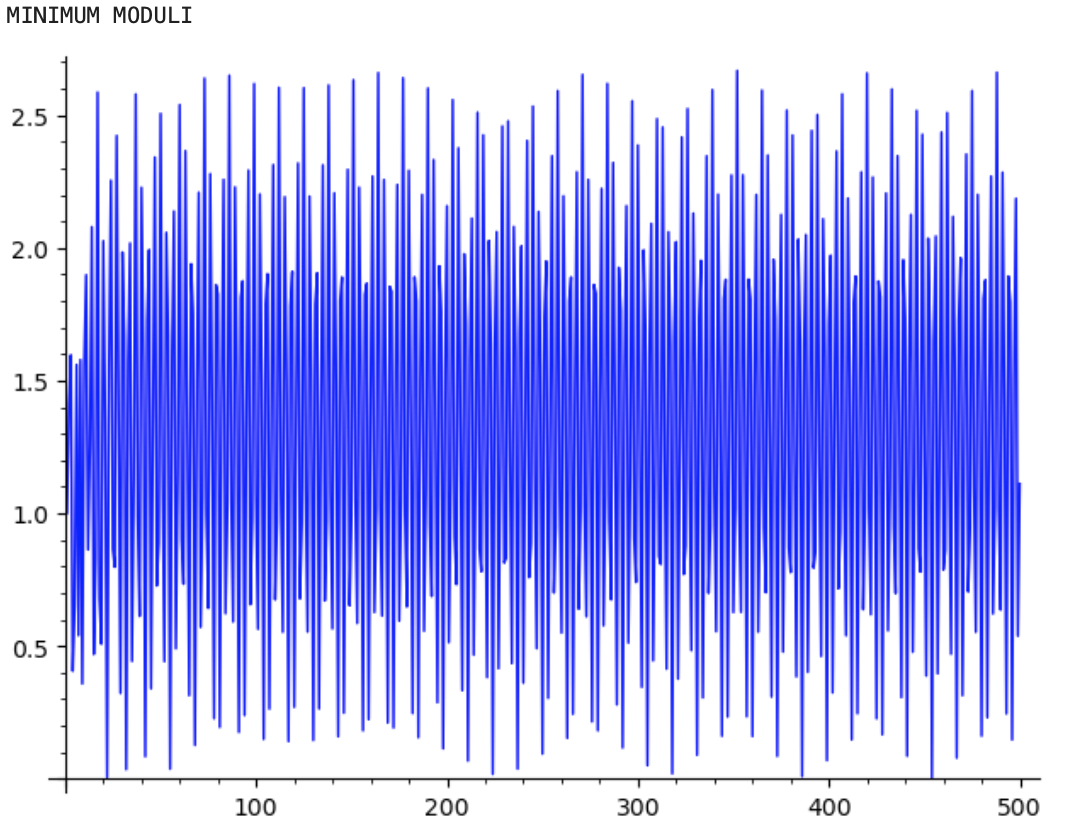}
        \caption{$h_n = a(p_n + 1)$.}
        \label{fig:curve26a1_primes_plus1}
    \end{subfigure}
    \caption{Cremona curve 26a1 with $c = 1$.}
    \label{fig:crv_26a1_comparison}
\end{figure}
\begin{figure}[H]
    \centering
    \begin{subfigure}[t]{0.48\textwidth}
        \centering
        \includegraphics[width=\textwidth]{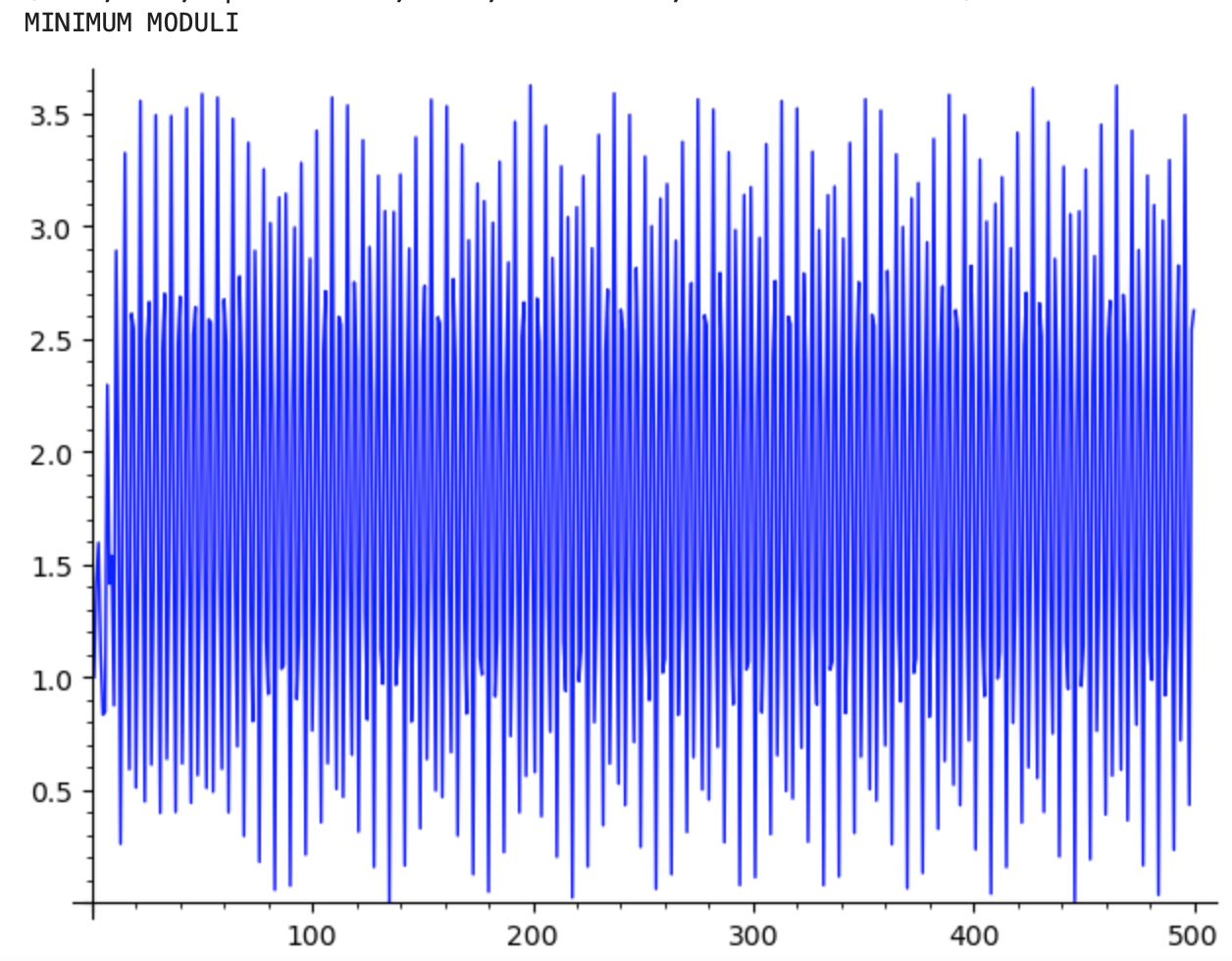}
        \caption{$h_n=a(n)$.}
        \label{fig:crv26b1_all}
    \end{subfigure}
    \hfill
    \begin{subfigure}[t]{0.48\textwidth}
        \centering
        \includegraphics[width=\textwidth]{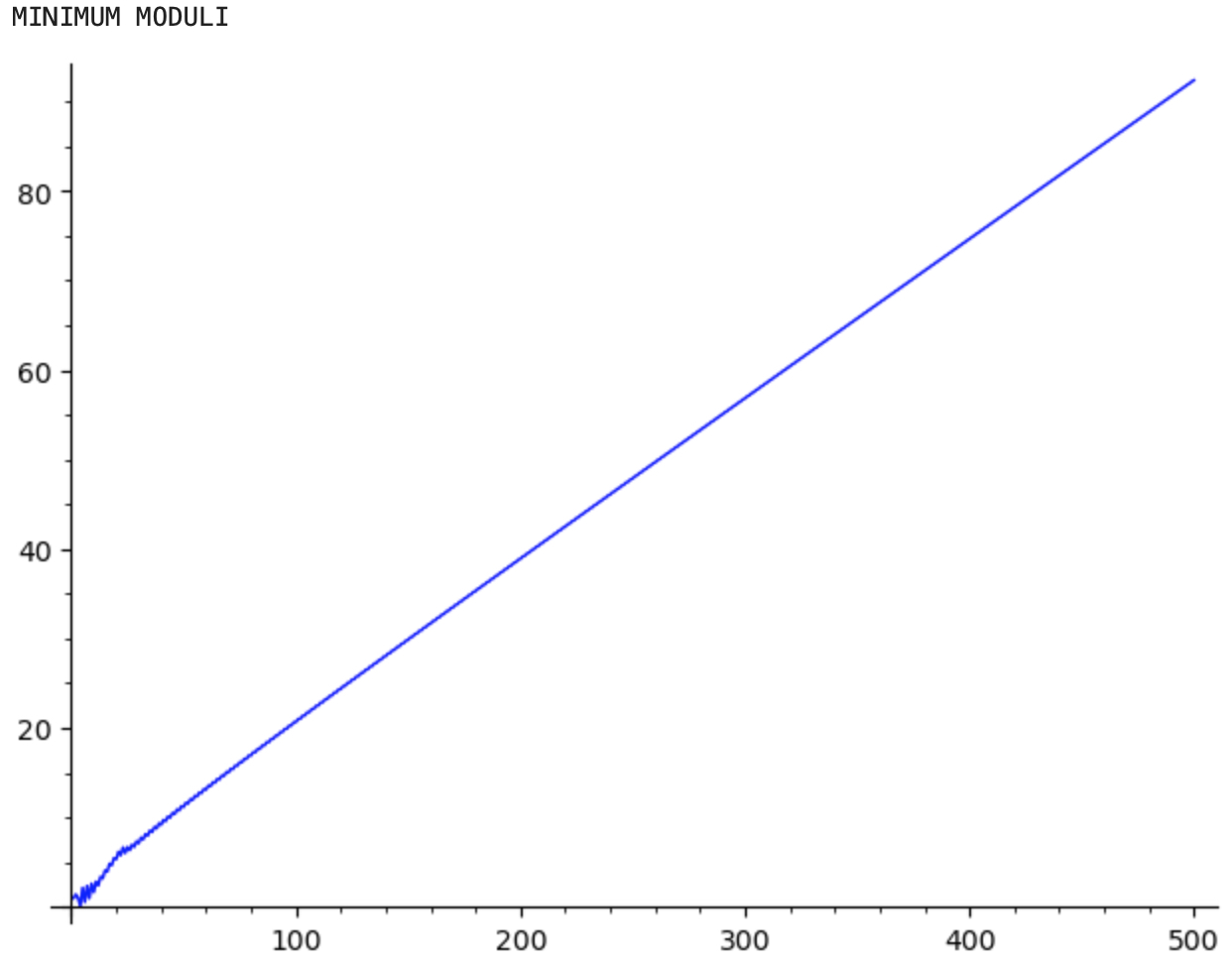}
        \caption{$h_n = a(p_n)$.}
        \label{fig:crv26b1_primes}
    \end{subfigure}
    \begin{subfigure}[t]{0.48\textwidth}
        \centering
        \includegraphics[width=\textwidth]{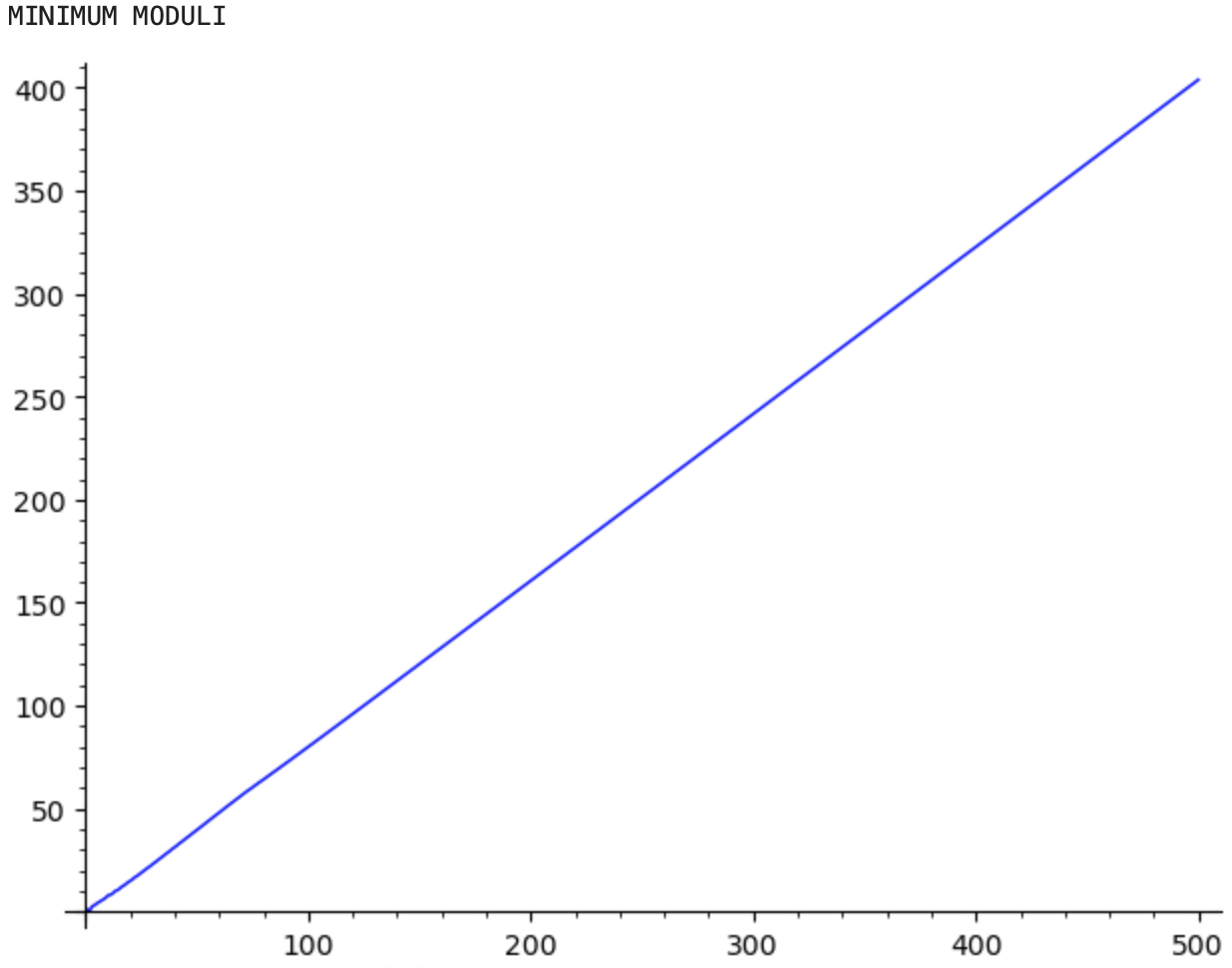}
        \caption{$h_n = a(p_n + 1)$.}
        \label{fig:crv26b1_primes_plus1}
    \end{subfigure}
    \caption{Cremona curve 26b1 with $c = 1$.}
    \label{fig:crv_26b1_comparison}
\end{figure}
\begin{figure}[H]
    \centering
    \begin{subfigure}[t]{0.48\textwidth}
        \centering
        \includegraphics[width=\textwidth]{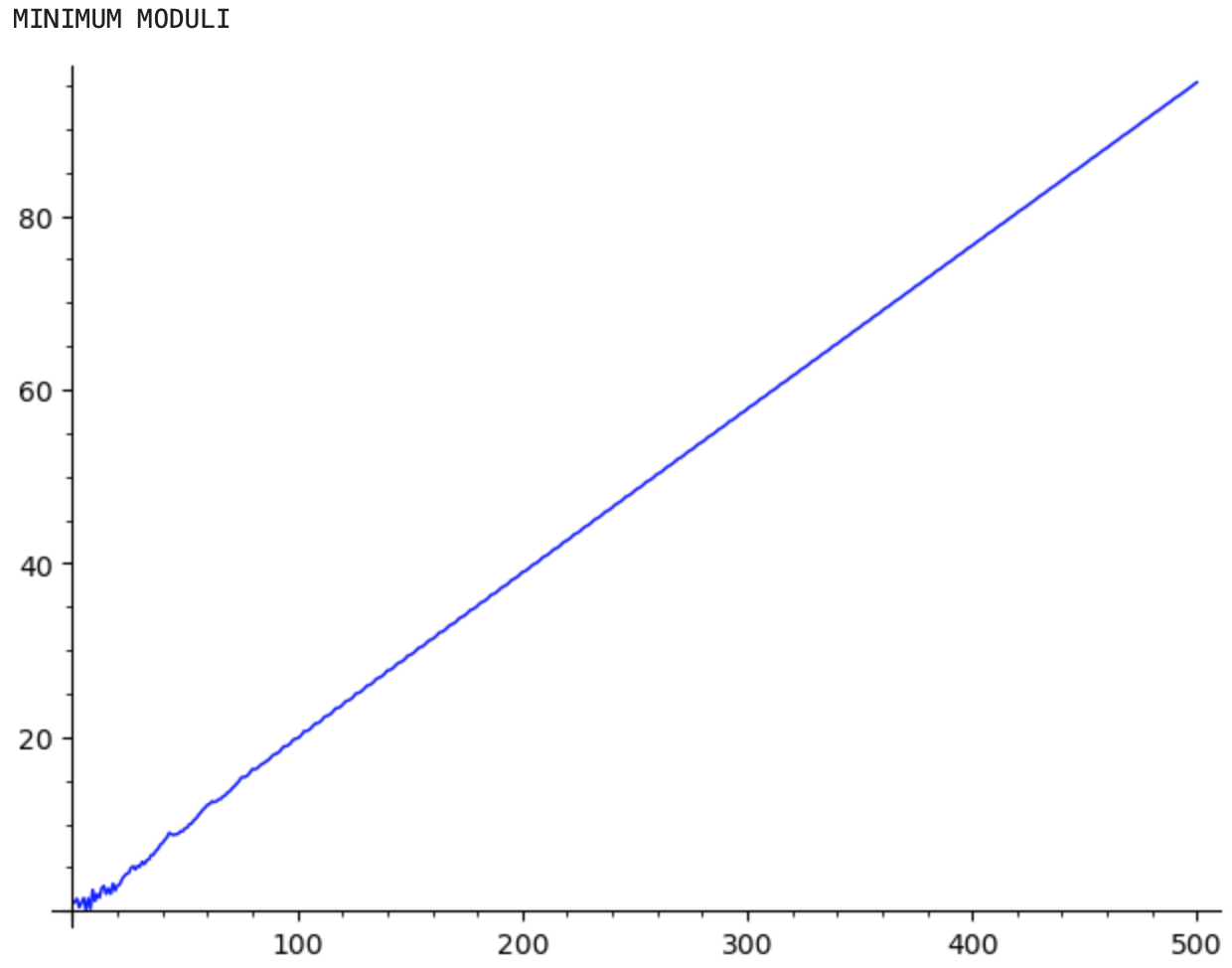}
        \caption{$h_n=a(n)$.}
        \label{fig:crv27a1_all}
    \end{subfigure}
    \hfill
    \begin{subfigure}[t]{0.48\textwidth}
        \centering
        \includegraphics[width=\textwidth]{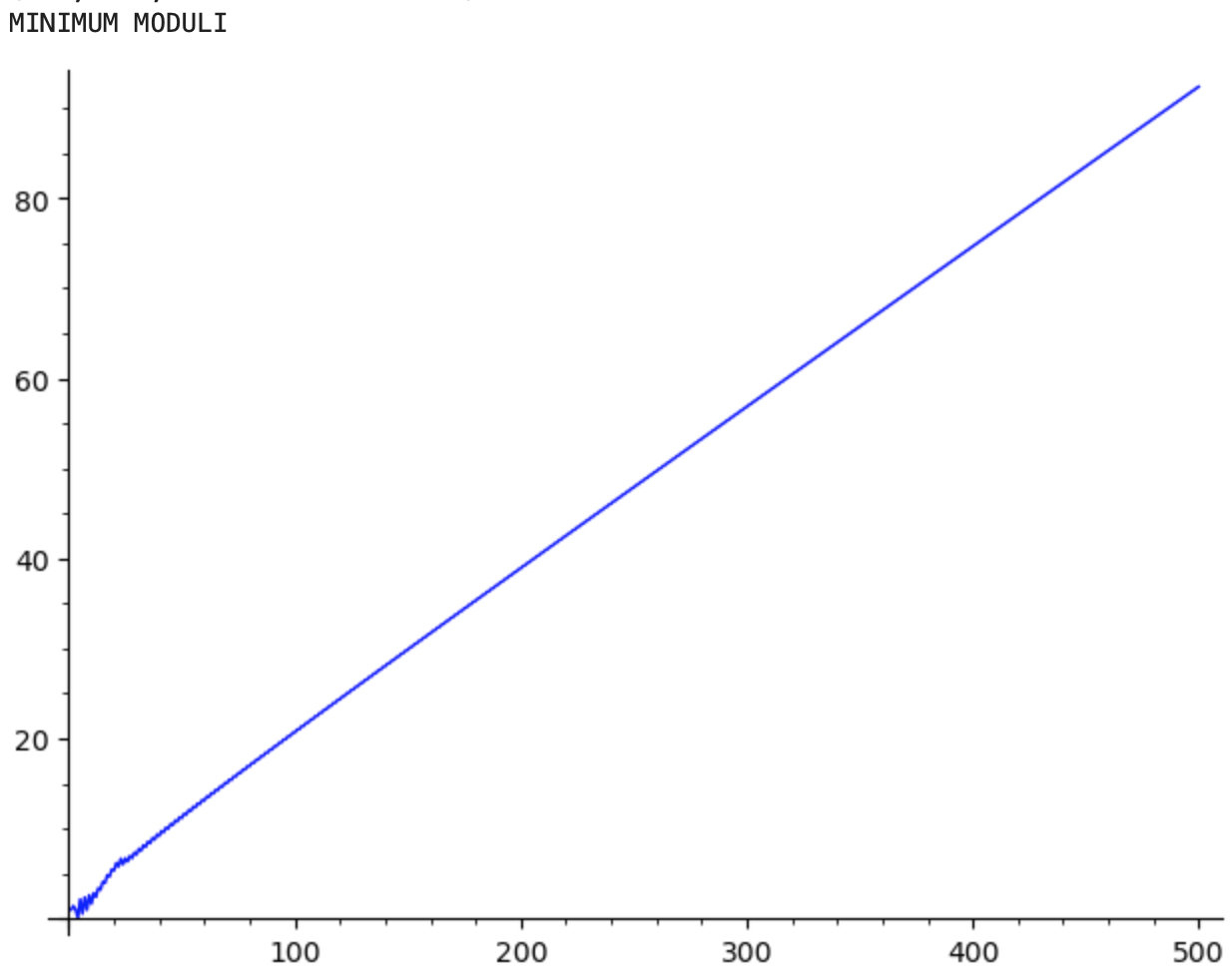}
        \caption{$h_n = a(p_n)$.}
        \label{fig:crv27a1_primes}
    \end{subfigure}
    \begin{subfigure}[t]{0.48\textwidth}
        \centering
        \includegraphics[width=\textwidth]{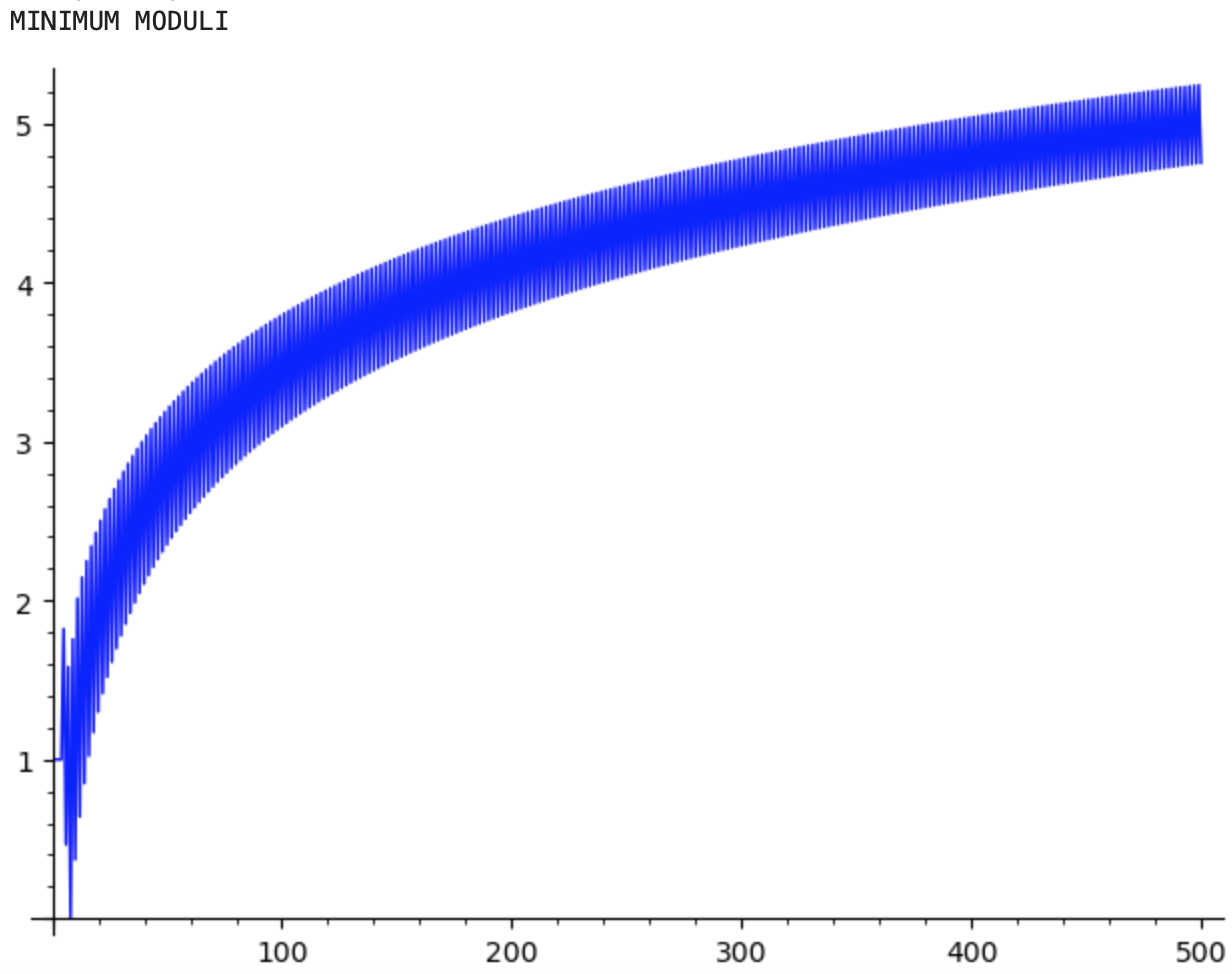}
        \caption{$h_n = a(p_n + 1)$.}
        \label{fig:crv27a1_primes_plus1}
    \end{subfigure}
    \caption{Cremona curve 27a1 with $c = 1$.}
    \label{fig:crv27a1_comparison}
\end{figure}
\begin{figure}[H]
    \centering
    \begin{subfigure}[t]{0.48\textwidth}
        \centering
        \includegraphics[width=1\textwidth]{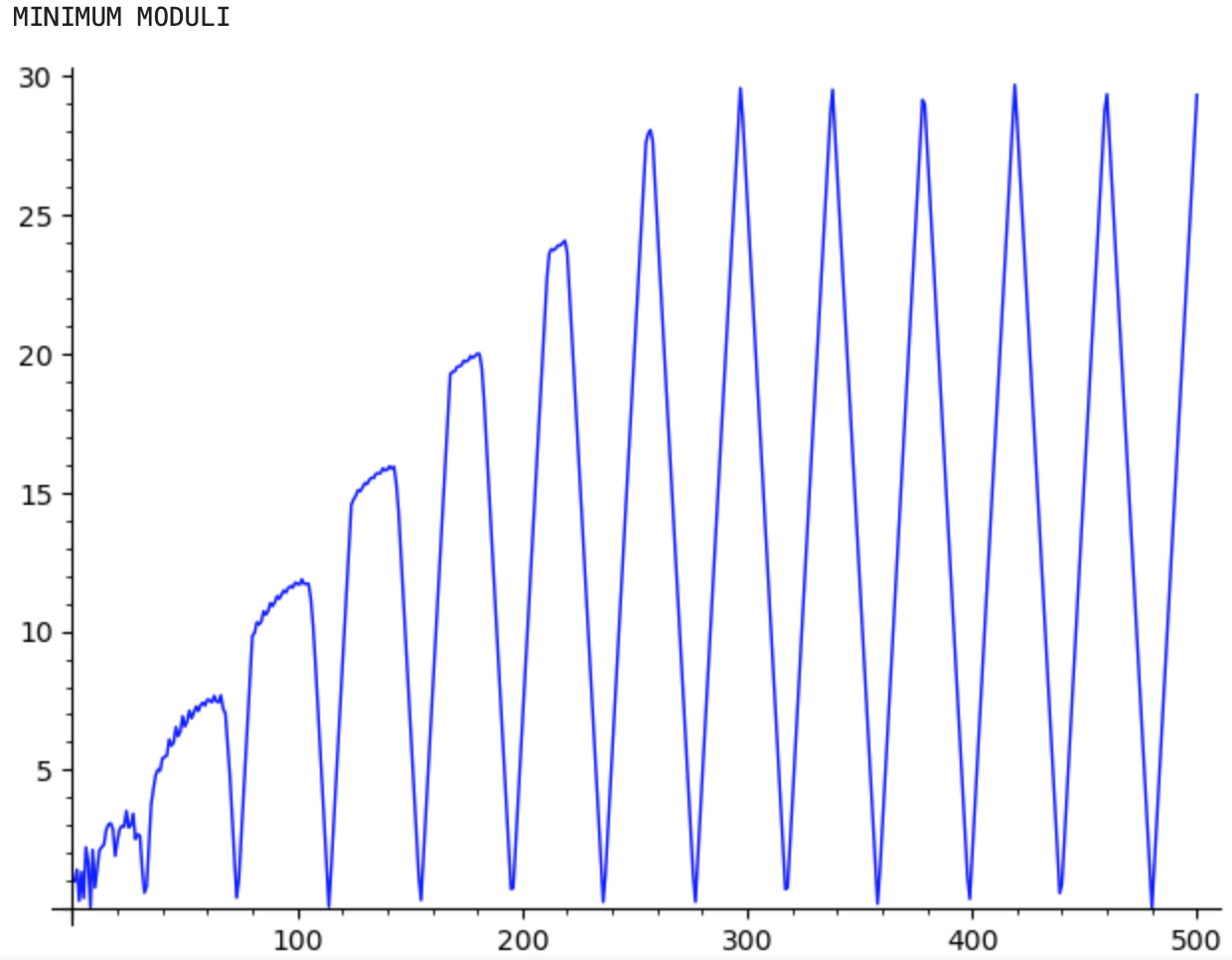}
        \caption{$h_n=a(n)$.}
        \label{fig:crv30a1_all}
    \end{subfigure}
    \begin{subfigure}[t]{0.48\textwidth}
        \centering
        \includegraphics[width=1\textwidth]{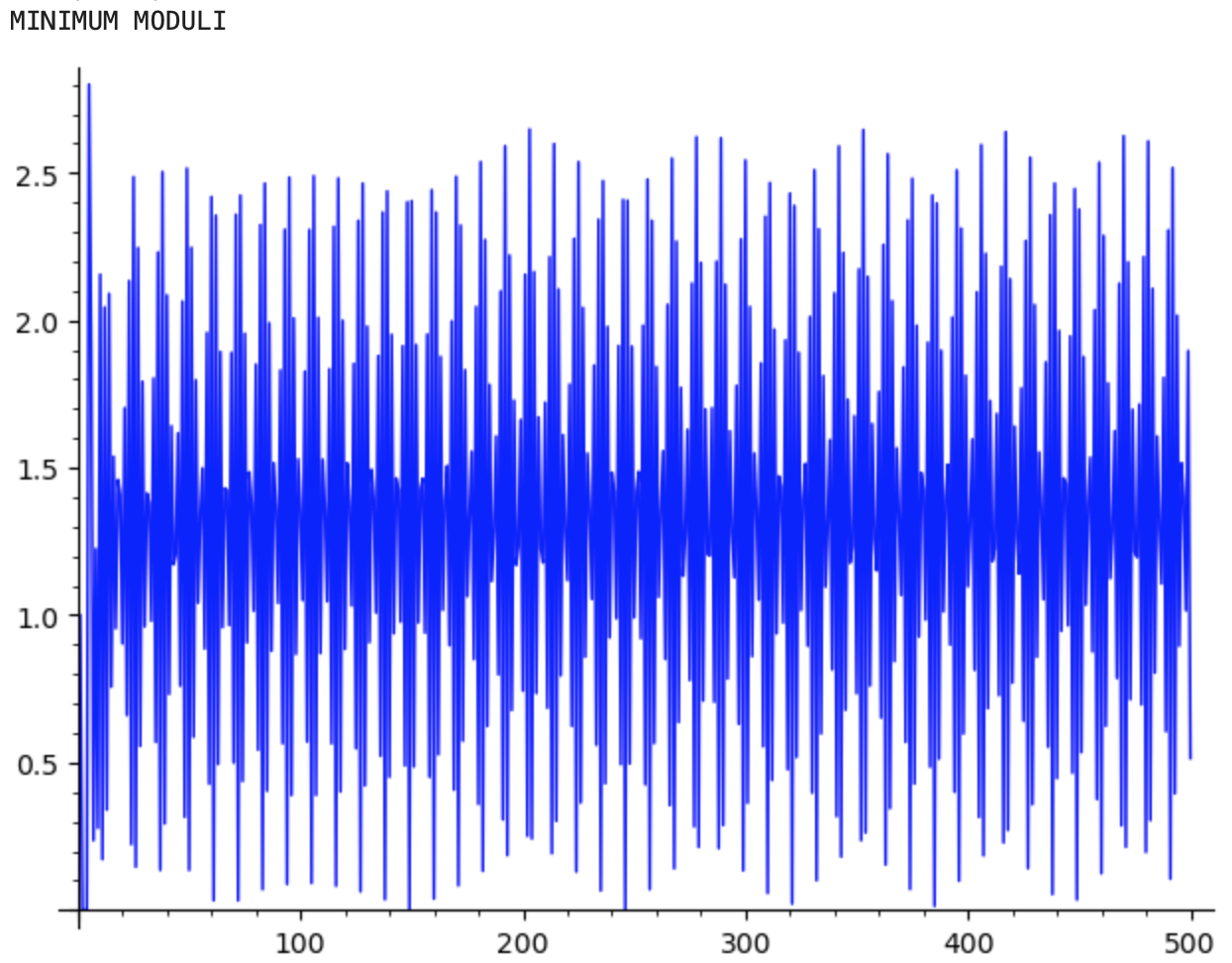}
        \caption{$h_n=a(p_n)$.}
        \label{fig:crv30a1_primes}
    \end{subfigure}
    \begin{subfigure}[t]{0.48\textwidth}
        \centering
        \includegraphics[width=1\textwidth]{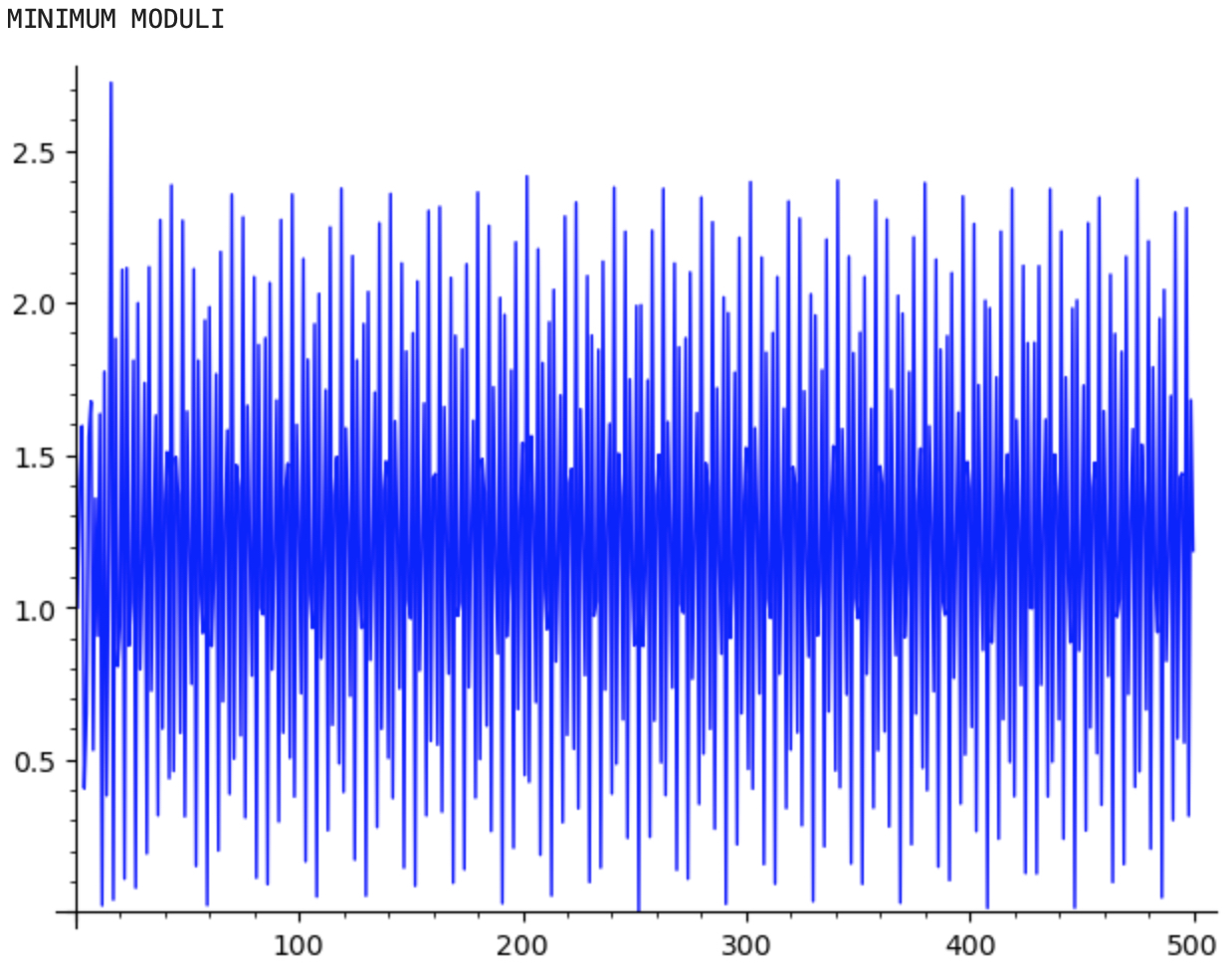}
        \caption{$h_n=a(p_n + 1)$.}
        \label{fig:crv30a1_primes_plus1}
    \end{subfigure}
     \caption{Cremona curve 30a1 with $c = 1$.}
    \label{fig:crv30a1_comparison}
\end{figure}
\begin{figure}[H]
    \centering
    \begin{subfigure}[t]{0.48\textwidth}
     \centering
    \includegraphics[width=1\textwidth]{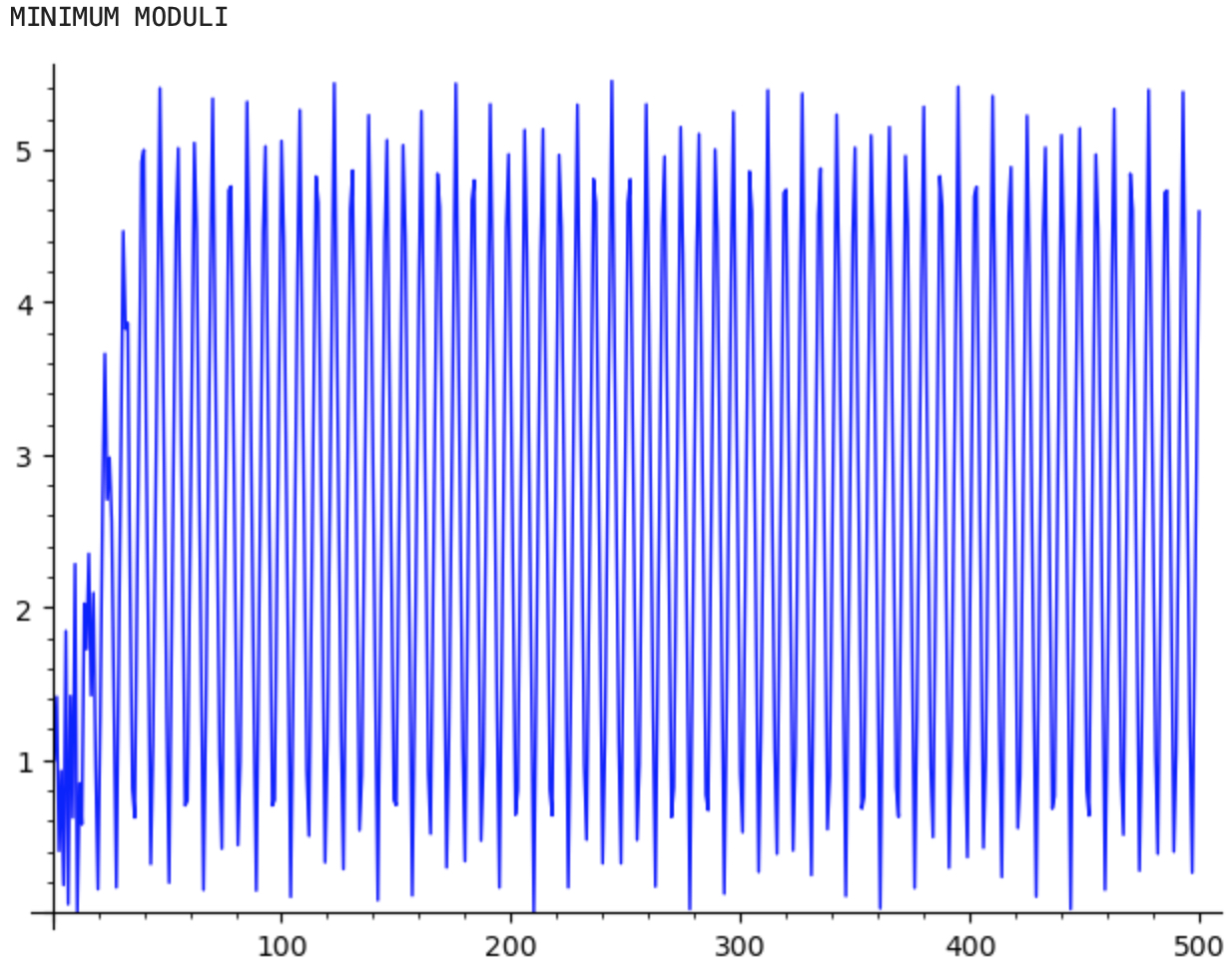}
    \caption{$h_n=a(n)$.}
    \label{fig:crv32a1_all}
     \end{subfigure}    
    \begin{subfigure}[t]{0.48\textwidth}
     \centering
    \includegraphics[width=1\textwidth]{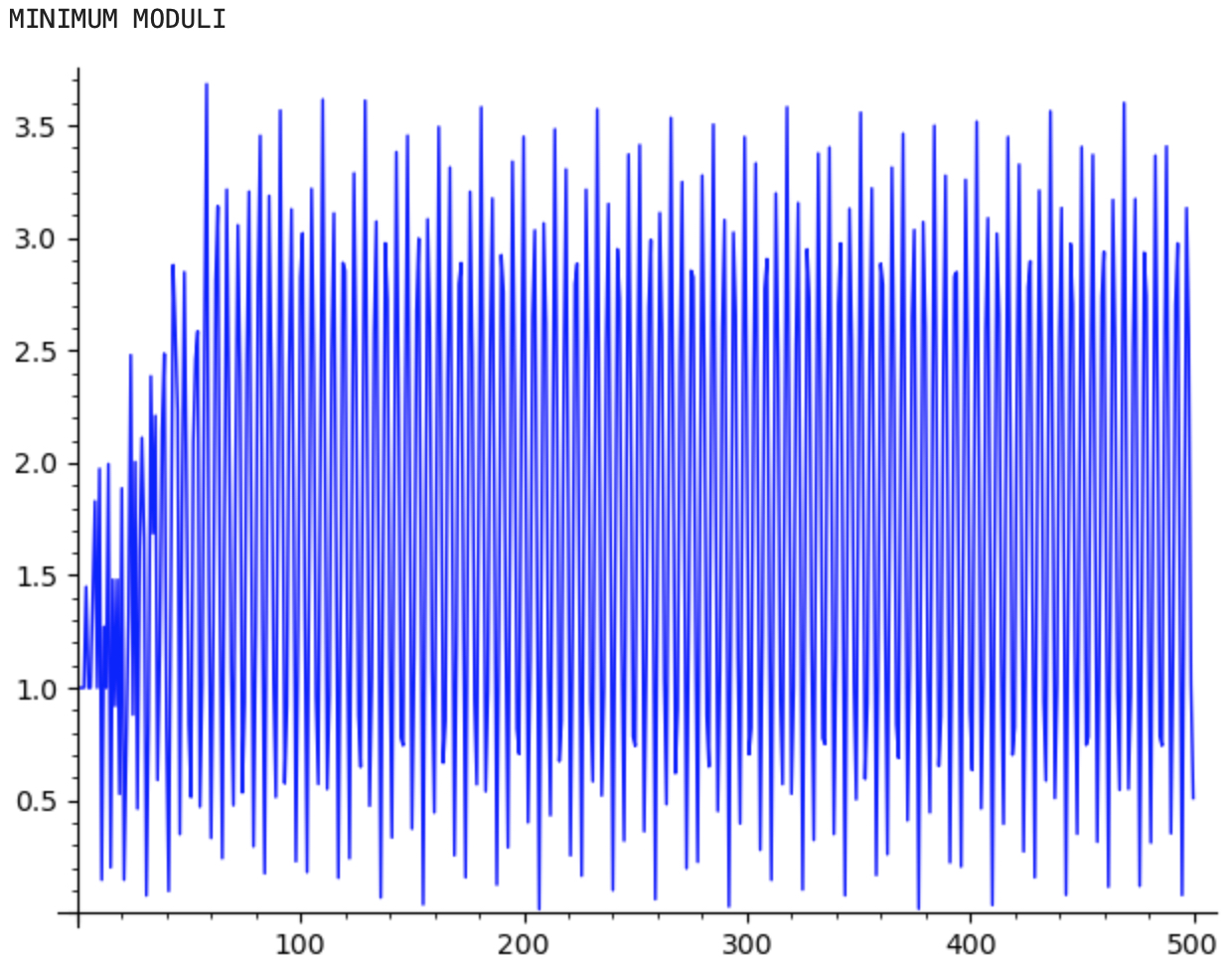}
    \caption{$h_n=a(p_n)$.}
    \label{fig:crv32a1_primes}
     \end{subfigure}  
     \begin{subfigure}[t]{0.48\textwidth}
     \centering
    \includegraphics[width=1\textwidth]{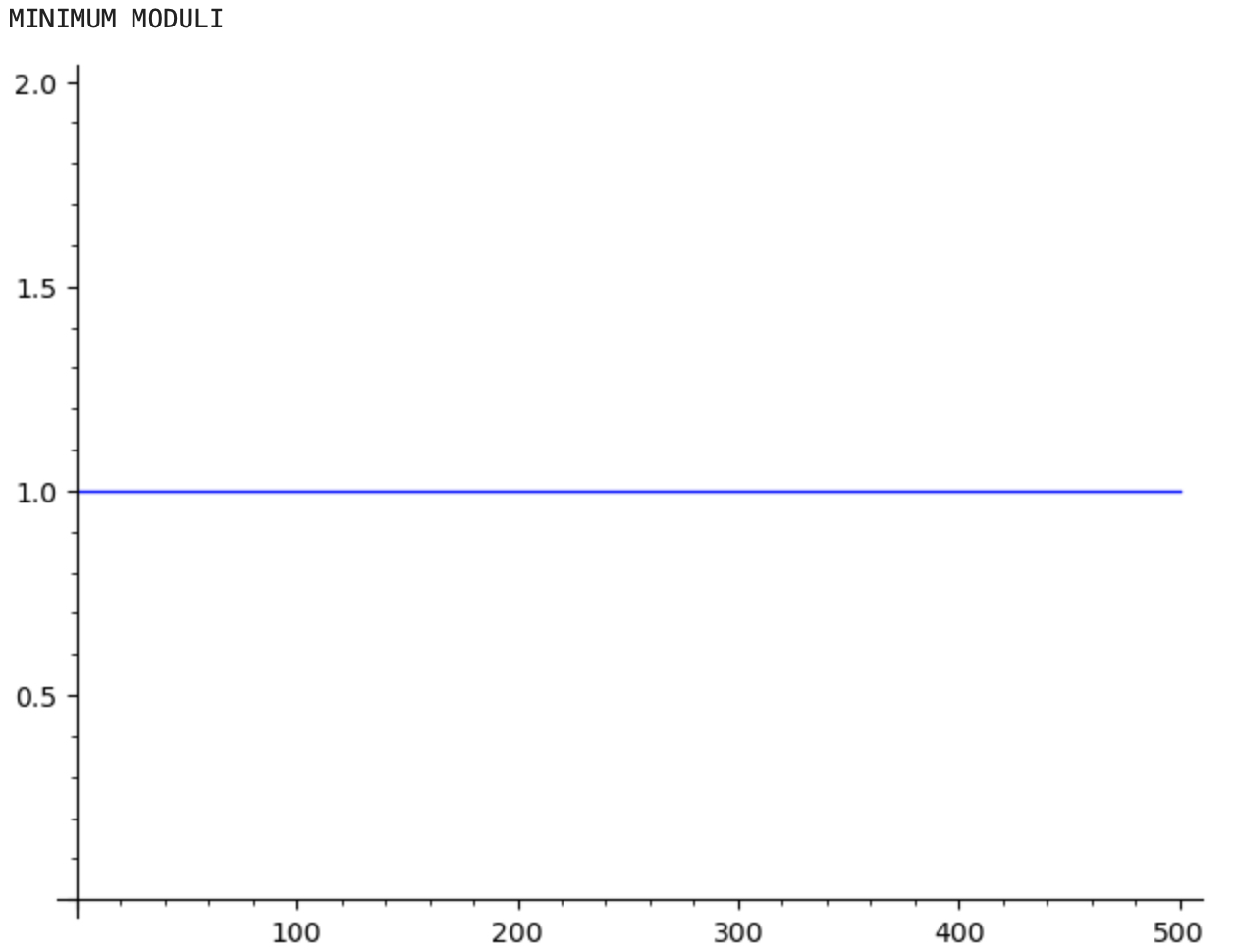}
    \caption{$h_n=a(p_n+1)$.}
    \label{fig:crv32a1_p_plus1}
     \end{subfigure}  
      \caption{Cremona curve 32a1 with $c = 1$.}
\end{figure}
\begin{figure}
\centering
    \hfill
    \begin{subfigure}[t]{0.48\textwidth}
        \centering
\includegraphics[width=\textwidth]
        {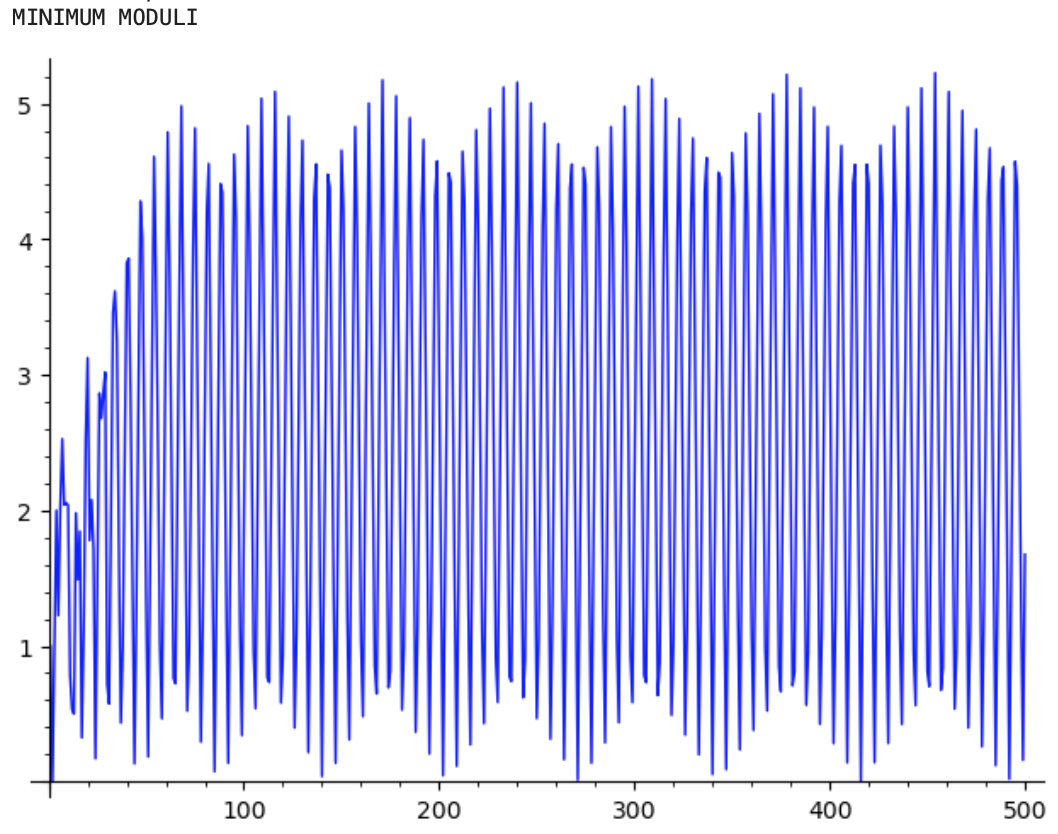}
        \caption{ $h_n = p_{n+1}-p_n-2$.}
        \label{fig:pph}
    \end{subfigure}
     \begin{subfigure}[t]{0.48\textwidth}
        \centering
\includegraphics[width=\textwidth]
        {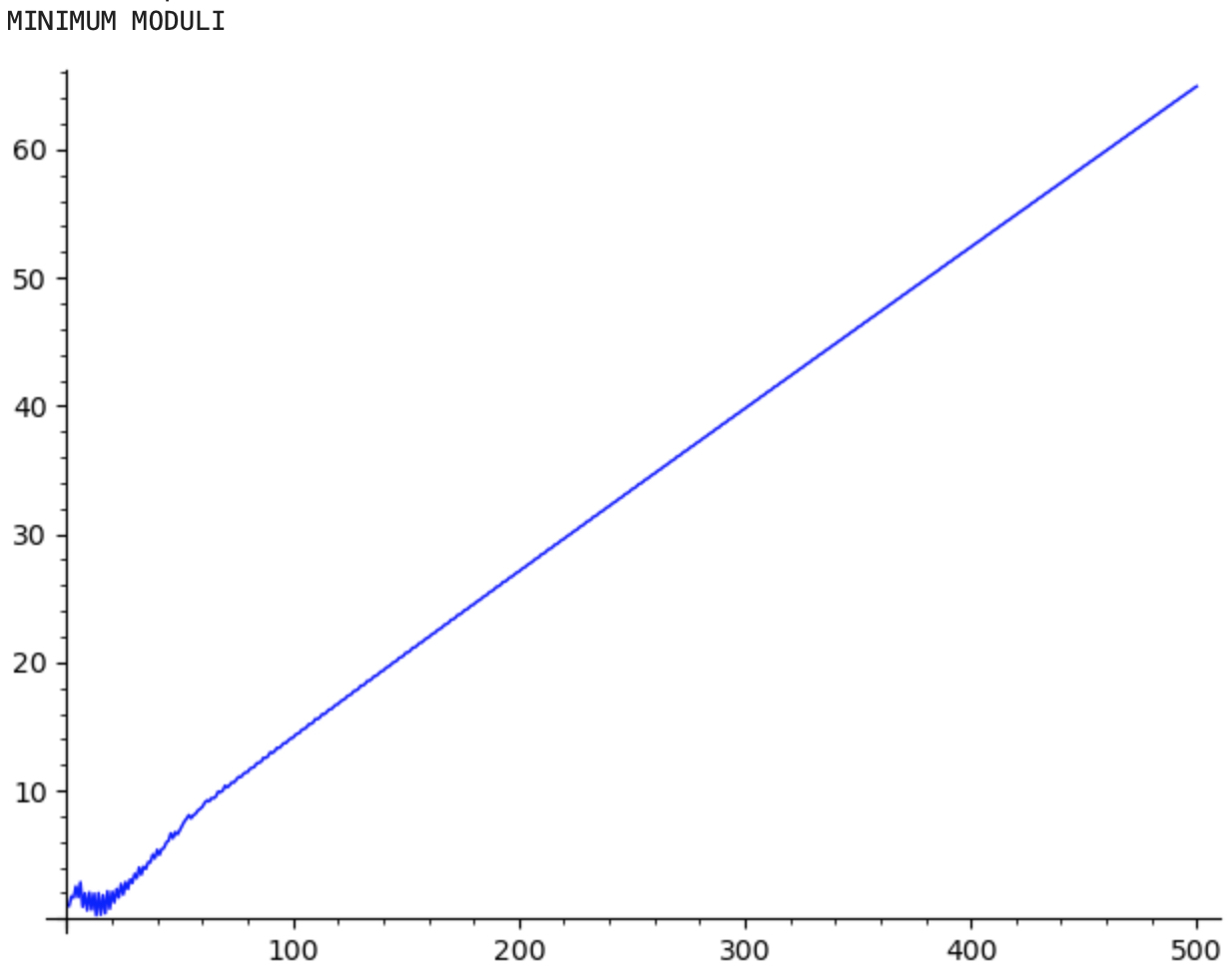}
        \caption{ $h_n = p_n$.}
    \label{fig:48B}
    \end{subfigure}
    \begin{subfigure}[t]{0.48\textwidth}
       \centering
\includegraphics[width=\textwidth]
    {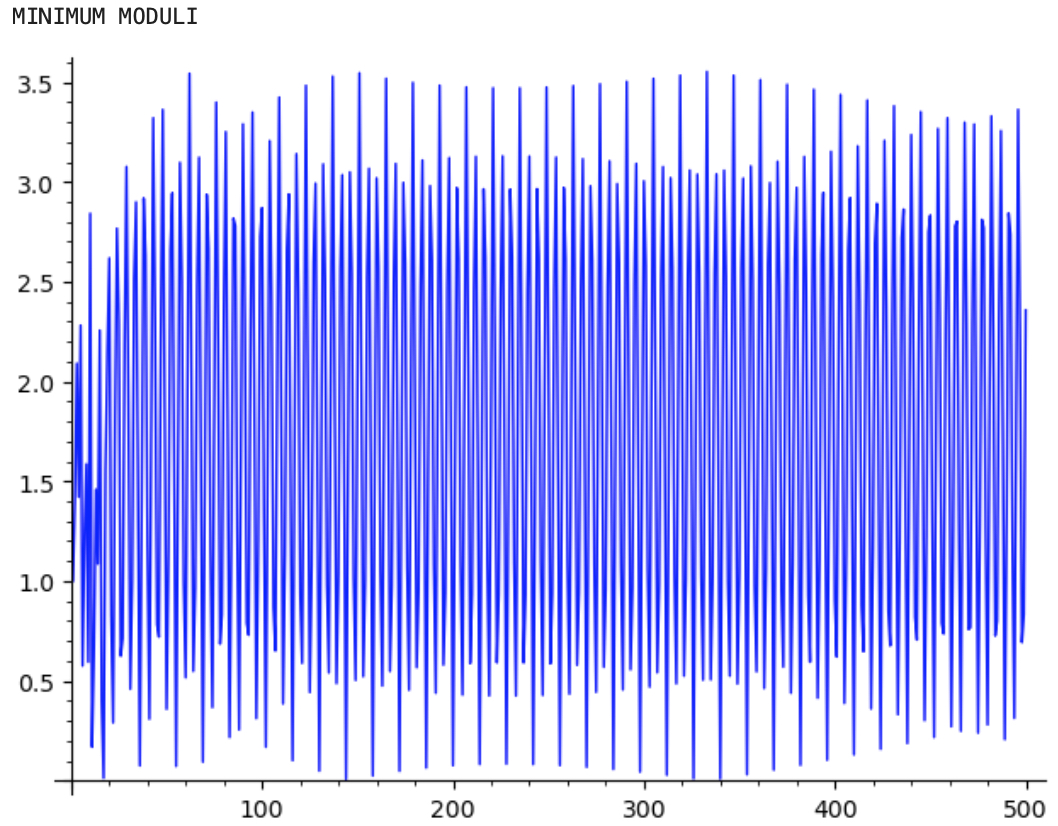}
      \caption{$h_n = p_{n+1}-p_n$.}
       \label{fig:prime_diffs}
    \end{subfigure}
    \caption{Simple $h$ sequences with $c = 1$.}
\end{figure}
\clearpage
\begin{figure}[H]
    \centering
    \begin{subfigure}[t]{0.48\textwidth}
        \centering
\includegraphics[width=\textwidth]{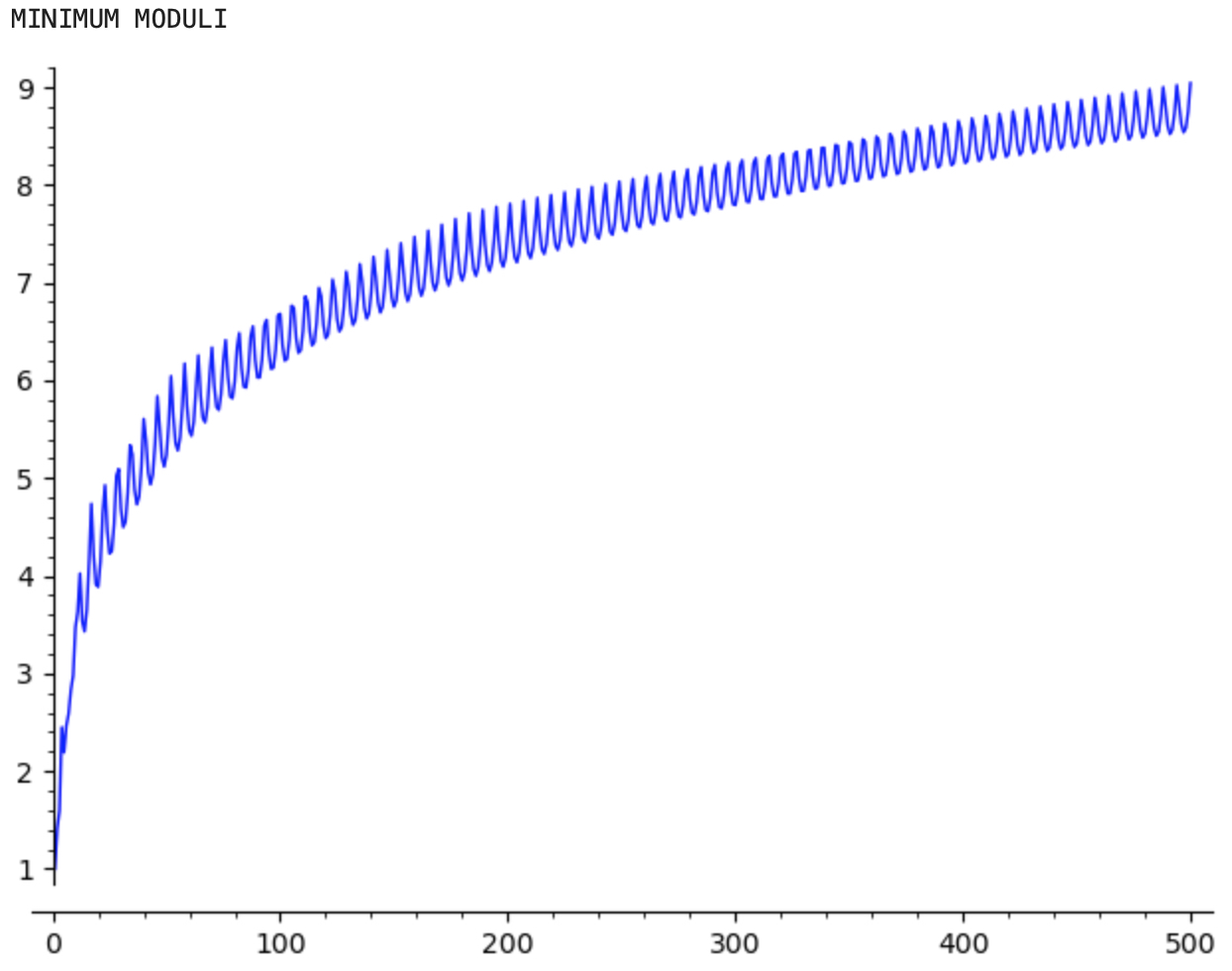}
        \caption{$a=1$.}
        \label{fig:h(n)=n}
    \end{subfigure}
   \begin{subfigure}[t]{0.48\textwidth}
        \centering
\includegraphics[width=\textwidth]{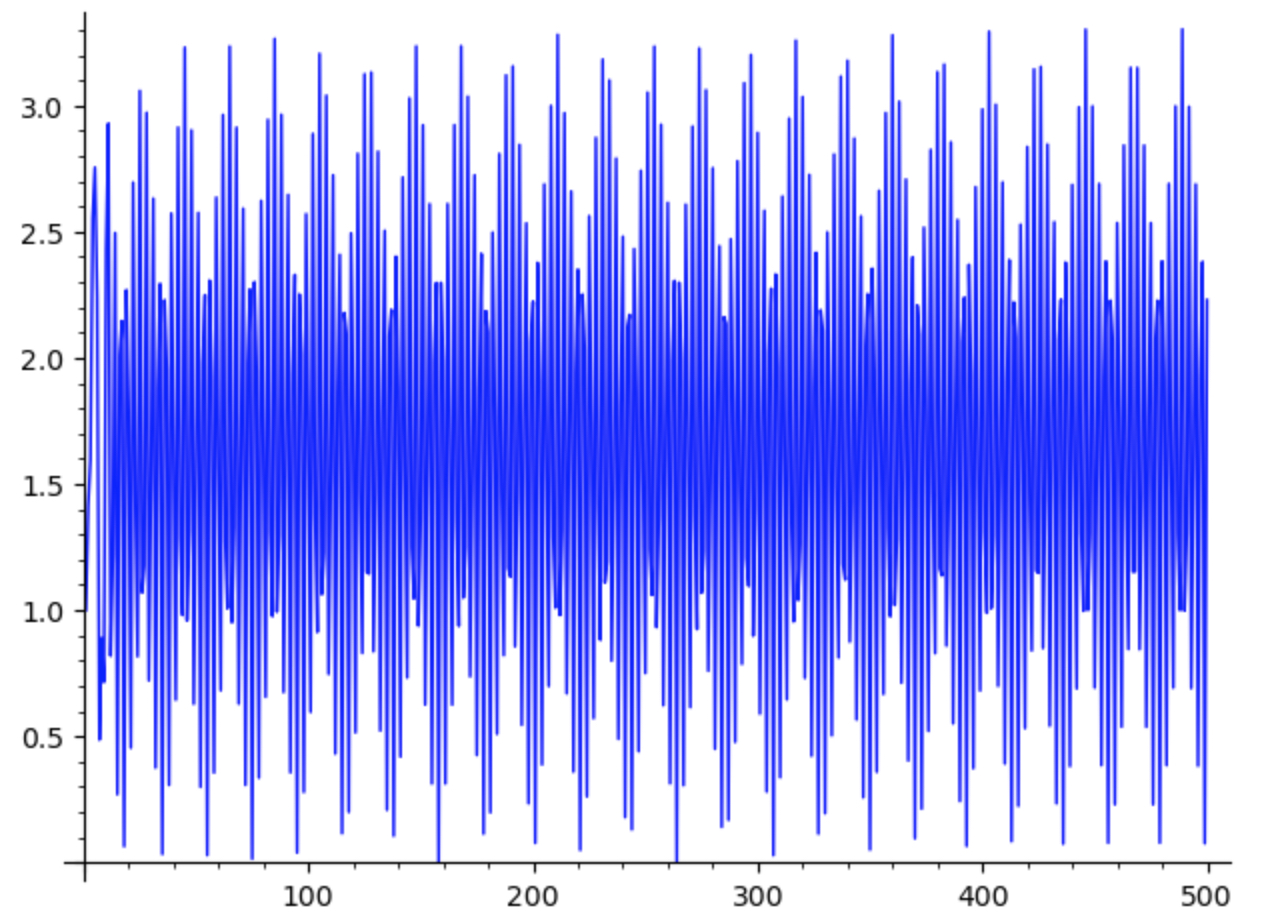}
        \caption{$a=2$.}
        \label{fig:h(n)=n^2}
    \end{subfigure}
    \begin{subfigure}[t]{0.48\textwidth}
        \centering
\includegraphics[width=\textwidth]{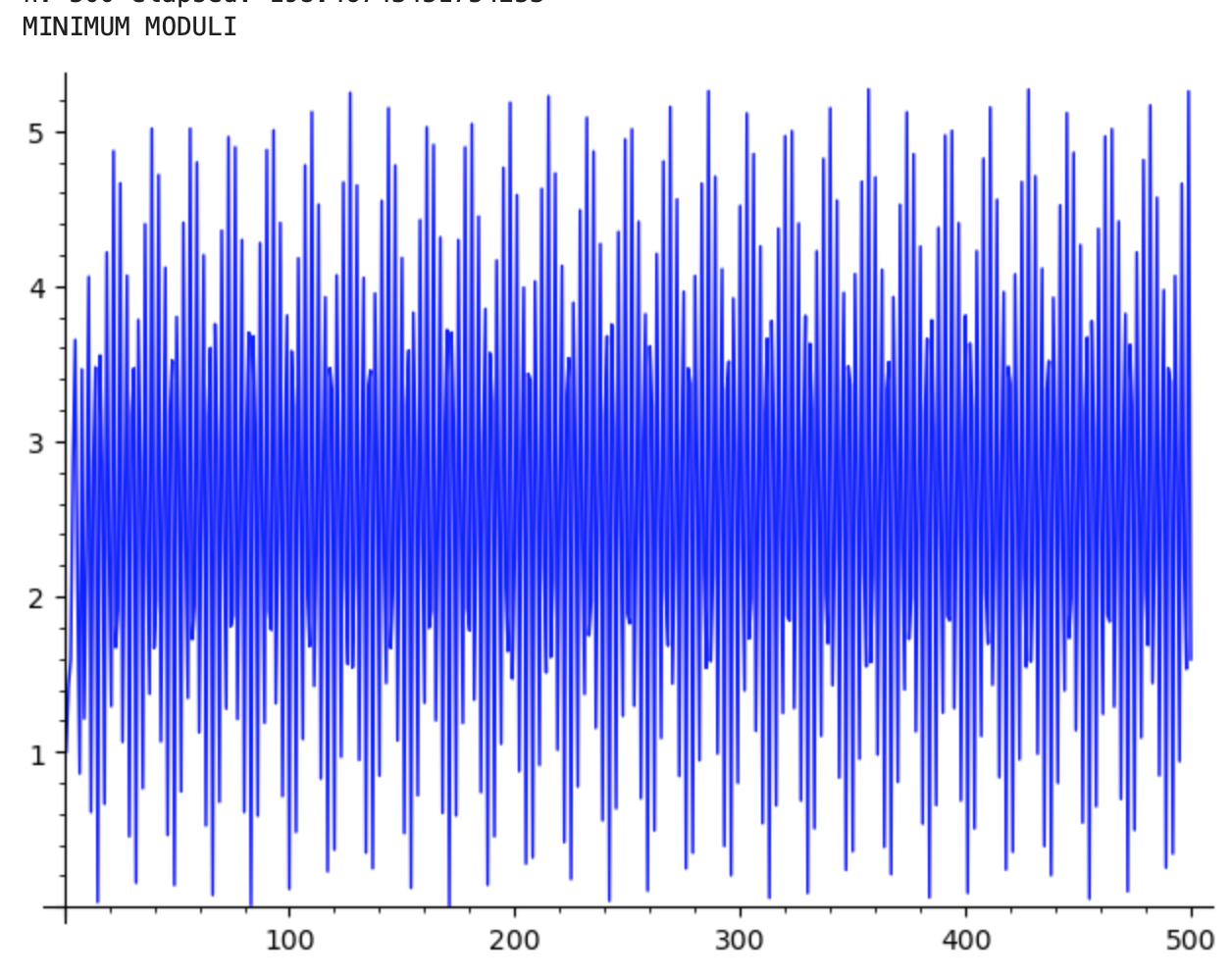}
        \caption{$a=3$.}
        \label{fig:h(n)=n^3}
    \end{subfigure}
     \begin{subfigure}[t]{0.48\textwidth}
        \centering
\includegraphics[width=\textwidth]{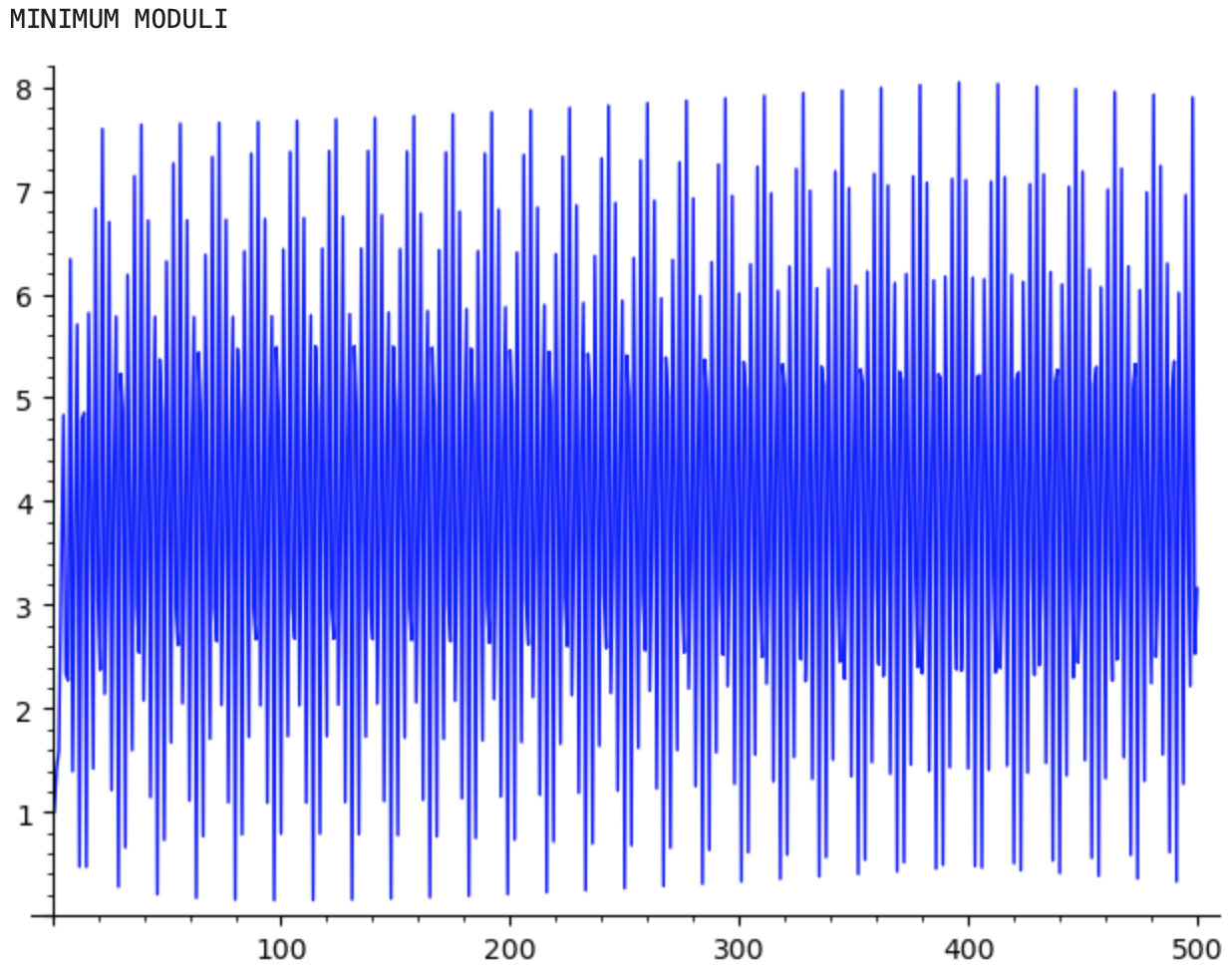}
        \caption{$a=4$.}
        \label{fig:h(n)=n^4}
    \end{subfigure}
    \caption{$h_n = n^a, 1 \le a \le 4$ with $c=1$. }
\end{figure}

\section{Appendix 1: index of figures}
The notebooks and the associated data files are stored at the listed URLs. 
\begin{longtable}{ll}
\caption*{}
\label{tab:figure-references} \\
\hline
\textbf{Figure} & \textbf{Notebook} \\
\hline
\endfirsthead
\hline
\textbf{Figure} & \textbf{Notebook} \\
\hline
\endhead
\hline
\endfoot
Figure 1A & \texttt{https://zenodo.org/records/21498743
}\\
Figure 1B & \texttt{https://zenodo.org/records/21499463
}\\
Figure 1C & \texttt{https://zenodo.org/records/21499645
}\\
Figure 2  & \texttt{https://zenodo.org/records/21499737
} \\
Figure 3  & \texttt{https://zenodo.org/records/21499926} \\
Figure 4  & \texttt{https://zenodo.org/records/21500027} \\
Figure 5  & 
\texttt{https://zenodo.org/records/21500153} \\
Figure 6  & \texttt{https://zenodo.org/records/21500362
} \\
Figure 7  & \texttt{https://zenodo.org/records/21500397
} \\
Figure 8A  & \texttt{https://zenodo.org/records/21334934} \\
Figure 8B  & \texttt{https://zenodo.org/records/21500444} \\
Figure 8C  & \texttt{https://zenodo.org/records/21500474} \\
Figure 9A & \texttt{https://zenodo.org/records/21343877} \\
Figure 9B  & \texttt{https://zenodo.org/records/21500582} \\
Figure 9C  & \texttt{https://zenodo.org/records/21344014} \\
Figure 10A & \texttt{https://zenodo.org/records/21347091} \\
Figure 10B & \texttt{https://zenodo.org/records/21500637} \\
Figure 10C & \texttt{https://zenodo.org/records/21347379} \\
Figure 11A & \texttt{https://zenodo.org/records/21500717
} \\
Figure 11B & \texttt{https://zenodo.org/records/21348988} \\
Figure 11C & \texttt{https://zenodo.org/records/21349261
} \\
Figure 12A & \texttt{https://zenodo.org/records/21366772
} \\
Figure 12B & \texttt{https://zenodo.org/records/21501779} \\
Figure 12C & \texttt{https://zenodo.org/records/21501835} \\
Figure 13A & \texttt{https://zenodo.org/records/21381035} \\
Figure 13B & \texttt{https://zenodo.org/records/21380598} \\
Figure 13C & \texttt{https://zenodo.org/records/21501927} \\
Figure 14A & \texttt{https://zenodo.org/records/21502012} \\
Figure 14B & \texttt{https://zenodo.org/records/21384000} \\
Figure 14C & \texttt{https://zenodo.org/records/21384030} \\
Figure 15A & \texttt{https://zenodo.org/records/21386313} \\
Figure 15B & \texttt{https://zenodo.org/records/21502075} \\
Figure 15C & \texttt{https://zenodo.org/records/21502189} \\
Figure 16A & \texttt{https://zenodo.org/records/21386106} \\
Figure 16B & \texttt{https://zenodo.org/records/21393699} \\
Figure 16C & \texttt{https://zenodo.org/records/21502224} \\
Figure 17A & \texttt{https://zenodo.org/records/21502284} \\
Figure 17B & \texttt{https://zenodo.org/records/21393532} \\
Figure 17C & \texttt{https://zenodo.org/records/21502412} \\
Figure 18A & \texttt{https://zenodo.org/records/21502481} \\
Figure 18B & \texttt{https://zenodo.org/records/21544634} \\
Figure 18C & \texttt{https://zenodo.org/records/21544718} \\
Figure 19A & \texttt{https://zenodo.org/records/21403041} \\
Figure 19B & \texttt{https://zenodo.org/records/21411071} \\
Figure 19C & \texttt{https://zenodo.org/records/21410675} \\
Figure 20A & \texttt{https://zenodo.org/records/21420653} \\
Figure 20B & \texttt{https://zenodo.org/records/21421288} \\
Figure 20C & \texttt{https://zenodo.org/records/21544791} \\
Figure 21A & \texttt{https://zenodo.org/records/21574898}\\
Figure 21B & \texttt{https://zenodo.org/records/21687161}\\
Figure 21C & \texttt{https://zenodo.org/records/21575214}\\
Figure 22A & \texttt{https://zenodo.org/records/21340487}\\
Figure 22B & \texttt{https://zenodo.org/records/21340575}\\
Figure 22C & \texttt{https://zenodo.org/records/21340648}\\
Figure 22D & \texttt{https://zenodo.org/records/21340835}\\
\hline
\end{longtable}

\clearpage
\begin{thebibliography}{LMFDB}
\bibitem[Ant26]{Anthropic2026} Anthropic, \textit{Claude}, 2026.
\url{https://claude.ai}.
\bibitem[BCDT01]{BCDT} C.~Breuil, B.~Conrad, F.~Diamond, and R.~Taylor, \textit{On the modularity of elliptic curves over $\mathbf{Q}$: wild 3-adic exercises}, J.\ Amer.\ Math.\ Soc.\ \textbf{14} (2001), 843--939.
\bibitem[Has36]{Hasse36} H.~Hasse, \textit{Zur Theorie der abstrakten elliptischen Funktionenk\"orper III. Die Struktur des Meromorphismenrings. Die Riemannsche Vermutung}, J.\ Reine Angew.\ Math.\ \textbf{175} (1936), 193--208.

\bibitem[HJ12]{Horn2012} Roger A. Horn and Charles R. Johnson, \textit{Matrix Analysis}, Cambridge University Press, 2012.
\bibitem[Leh47]{Le47} D.~H. Lehmer, \textit{The vanishing of Ramanujan's function $\tau(n)$}, Duke Math.\ J.\ \textbf{14} (1947), 429--433, \url{https://zenodo.org/records/21779774}.
\bibitem[LMFDB]{LMFDB} The LMFDB Collaboration,
\textit{The {L}-functions and Modular Forms Database},
\url{https://www.lmfdb.org}, 2025.
\bibitem[Mac95]{M} I.~G. Macdonald, \textit{Symmetric Functions and Hall Polynomials}, 2nd ed., Clarendon Press, Oxford, 1995.
\bibitem[Mar18]{Marcus} D.~A.~Marcus, \textit{Number Fields}, 2nd ed., Universitext, Springer, 2018.
\bibitem[Orr18]{Orr} M.~Orr, \textit{Algebraic Number Theory}, lecture notes, University of Manchester, 2018--19. Available at \url{https://personalpages.manchester.ac.uk/staff/Martin.Orr/2018-9/ant/lectures.pdf}.
\bibitem[Sil09]{Silverman2009} J.~H. Silverman, \textit{The Arithmetic of Elliptic Curves}, 2nd ed., Graduate Texts in Mathematics \textbf{106}, Springer, 2009.

\bibitem[VD99]{Vein1999} Robert Vein and Paul Dale,
\textit{Determinants and their Applications in Mathematical Physics},
Springer, 1999.
\bibitem[TW95]{TW} R.~Taylor and A.~Wiles, \textit{Ring-theoretic properties of certain Hecke algebras}, Ann.\ of Math.\ (2) \textbf{141} (1995), 553--572.
\bibitem[Wil95]{Wiles} A.~Wiles, \textit{Modular elliptic curves and Fermat's Last Theorem}, Ann.\ of Math.\ (2) \textbf{141} (1995), 443--551.
\end{thebibliography}
\end{document}